\documentclass[11 pt]{article} 
\usepackage[utf8]{inputenc}
\usepackage{array,amsmath,indentfirst,stmaryrd,qsymbols,graphicx,psfrag,amsfonts,hyperref,color,enumerate,amsthm,ulem,dsfont}
\usepackage[dvipsnames]{xcolor}
\usepackage{pgf,tikz}
\usetikzlibrary{arrows}
\usepackage{xstring}
\usepackage[left=2cm,right=2cm,top=2.5cm,bottom=2.5cm]{geometry}

\DeclareMathOperator{\Trace}{Trace}

\renewcommand{\mod}{{~\sf mod~}}
\def \1{\mathds{1}}
\def \HC{{\sf HC}}
\def \Tr{{\sf Tr}}

\def \Z{\mathbb{Z}}

\newcommand{\E}{\mathbb{E}}

\def \square#1#2#3#4{\begin{array}{|c@{\hspace{2pt}}c|}\hline #1 & #2 \\ #4 & #3\\\hline \end{array}}
\def \ssquare#1#2#3#4{\begin{array}{|c@{\hspace{2pt}}c|}\hline #4 & #3 \\ #1 & #2\\\hline \end{array}}

\def \tt#1#2#3#4#5#6#7#8{{\sf T}{\compact\small
    \square#1#2#3#4 \square#5#6#7#8}}

\newcommand\xoutpars[1]{\let\helpcmd\xout\parhelp#1\par\relax\relax}
\newcommand\soutpars[1]{\let\helpcmd\sout\parhelp#1\par\relax\relax}
\long\def\parhelp#1\par#2\relax{%
  \helpcmd{#1}\ifx\relax#2\else\par\parhelp#2\relax\fi%
}

\def \Sub#1#2{{\sf Seq}_{#2}(#1)}
\def \Subn#1#2#3{{\sf Sub}^{#3}_{#2}(#1)}

\def \Sys#1{{\sf Sys}(#1)}

\def \app#1#2#3#4#5{\begin{array}{rccl} #1:&#2&\longrightarrow&#3\\ &#4&\longmapsto&#5\end{array}}

\def \N{\mathbb{N}}
\def \R{\mathbb{R}}

\def \bZ{{\bf Z}}

\def \bar{\overline}
\def \ba{\begin{align}}
\def \ea{\end{align}}
\def \be{\begin{eqnarray*}}
\def \ee{\end{eqnarray*}}
\def \ben{\begin{eqnarray}}
\def \een{\end{eqnarray}}

\def \bp{{\bf p}}
\def \beq{\begin{equation}}
\def \eq{\end{equation}}
\def \bs{{\bf s}}
\def \build#1#2#3{\mathrel{\mathop{\kern 0pt#1}\limits_{#2}^{#3}}}

\def \ba{{\bf a}}

\def \captionn#1{\begin{center}\begin{minipage}{15cm}\sf\caption{\small #1}\end{minipage}\end{center}}
\def \Sq{{\sf Sq}}

\def \dis{\displaystyle}

\def \EkZ{E_{\kappa}^{\mathbb{Z}}}
\def \EkL{E_{\kappa}^{L}}
\def \Ek{E_{\kappa}}
\def \Lineq{{\sf Line}}
\def \Rep{{\sf Replace}}

\def \NLineq{{\sf NLine}}
\def \ME{{\sf Master}}

\def \Cycle{{\sf Cycle}}
\def \NCycle{{\sf NCycle}}

\def \equi{\Leftrightarrow}
\def \eref#1{(\ref{#1})}

\def \bF{{\bf F}}
\def \imp{\Rightarrow}

\def \l{\left}

\def \r{\right}
\def \sous#1#2{\mathrel{\mathop{\kern 0pt#1}\limits_{#2}}}
\def \sur#1#2{\mathrel{\mathop{\kern 0pt#1}\limits^{#2}}}

\newqsymbol{`B}{\mathcal{B}}
\newqsymbol{`E}{\mathbb{E}}
\newqsymbol{`I}{\mathbb{I}}
\newqsymbol{`N}{\mathbb{N}}
\newqsymbol{`O}{\Omega}
\newqsymbol{`P}{\mathbb{P}}
\newqsymbol{`Q}{\mathbb{Q}}
\newqsymbol{`R}{\mathbb{R}}
\newqsymbol{`W}{\mathbb{W}}
\newqsymbol{`Z}{\mathbb{Z}}
\newqsymbol{`a}{\alpha}
\newqsymbol{`e}{\varepsilon}
\newqsymbol{`o}{\omega}
\newqsymbol{`t}{\tau}
\newqsymbol{`w}{{\cal W}}

\def \P{\mathbb{P}}

\newcommand{\compact}{ \topsep0pt   \itemsep=0pt   \partopsep=0pt   \parsep=0pt}

\newcounter{c}
\def \bir{\begin{itemize}\compact \setcounter{c}{0}}
\def \itr{\addtocounter{c}{1}\item[($\roman{c}$)]}
\def \eir{\end{itemize}\vspace{-2em}~}

\newcounter{d}
\def \bia{\begin{itemize}\compact \setcounter{d}{0}}
\def \ita{\addtocounter{d}{1}\item[(\alph{d})]} 
\def \eia{\end{itemize}\vspace{-2em}~}

\newcounter{b}
\def \bi{\begin{itemize}\compact \setcounter{b}{0}}

\def \ei{\end{itemize}\vspace{-2em}~}

\def \bpar#1{\left\{\begin{array}{#1} }
\def \epar { \end{array}\right.}

\newtheorem{lem}{Lemma}[section]
\newtheorem{defi}[lem]{Definition}
\newtheorem{pro}[lem]{Proposition}
\newtheorem{theo}[lem]{Theorem}
\newtheorem{cor}[lem]{Corollary}

\newtheorem{rem}[lem]{Remark}

\usepackage{chngcntr}
\counterwithin{equation}{section}

\usepackage{hyperref}
\def \bs{\tiny $\blacksquare$}

\def \Sq{{\sf Sq}}

\renewcommand{\mod}{\,{\sf mod}\,}

\def \tb{\noindent {\tiny $\blacksquare$ }}

\def \T#1#2{
	\begingroup
	\expandarg
	\StrLen{#1}[\leng]
	\ifnum\leng<3
	{\sf T}_{\left[#1
	\else
	\StrLeft{#1}{1}[\prefic]%
	\StrRight{#1}{1}[\sufic]%
	\StrSubstitute[1]{\prefic}{[}{ }[\prefic]%
	\StrSubstitute[1]{\sufic}{]}{ }[\sufic]%
	\StrSubstitute[1]{\prefic}{(}{ }[\prefic]%
	\StrSubstitute[1]{\sufic}{)}{ }[\sufic]%
	\StrMid{#1}{2}{\leng}[\auxword]%
	\StrLen{\auxword}[\leng]%
	\StrMid{#1}{2}{\leng}[\worc]%
	{\sf T}_{\left[\prefic\worc\sufic
	\fi
	\StrLen{#2}[\lengw]%
	\ifnum\lengw<3
	|#2\right]}
	\else
	\StrLeft{#2}{1}[\preficw]%
	\StrRight{#2}{1}[\suficw]%
	\StrSubstitute[1]{\preficw}{[}{ }[\preficw]%
	\StrSubstitute[1]{\preficw}{(}{ }[\preficw]%
	\StrSubstitute[1]{\suficw}{]}{ }[\suficw]%
	\StrSubstitute[1]{\suficw}{)}{ }[\suficw]%
	\StrMid{#2}{2}{\lengw}[\auxwordw]%
	\StrLen{\auxwordw}[\lengw]%
	\StrMid{#2}{2}{\lengw}[\worcw]%
	|\preficw\worcw\suficw\right]}
	\fi
	\endgroup
}

\def \TTT#1#2#3{
	\begingroup
	\expandarg
	\StrLen{#1}[\leng]
	\ifnum\leng<3
	{\sf T}^{#3}_{\left[#1
	\else
	\StrLeft{#1}{1}[\prefic]%
	\StrRight{#1}{1}[\sufic]%
	\StrSubstitute[1]{\prefic}{[}{ }[\prefic]%
	\StrSubstitute[1]{\sufic}{]}{ }[\sufic]%
	\StrSubstitute[1]{\prefic}{(}{ }[\prefic]%
	\StrSubstitute[1]{\sufic}{)}{ }[\sufic]%
	\StrMid{#1}{2}{\leng}[\auxword]%
	\StrLen{\auxword}[\leng]%
	\StrMid{#1}{2}{\leng}[\worc]%
	{\sf T}^{#3}_{\left[\prefic\worc\sufic
	\fi
	\StrLen{#2}[\lengw]%
	\ifnum\lengw<3
	|#2\right]}
	\else
	\StrLeft{#2}{1}[\preficw]%
	\StrRight{#2}{1}[\suficw]%
	\StrSubstitute[1]{\preficw}{[}{ }[\preficw]%
	\StrSubstitute[1]{\preficw}{(}{ }[\preficw]%
	\StrSubstitute[1]{\suficw}{]}{ }[\suficw]%
	\StrSubstitute[1]{\suficw}{)}{ }[\suficw]%
	\StrMid{#2}{2}{\lengw}[\auxwordw]%
	\StrLen{\auxwordw}[\lengw]%
	\StrMid{#2}{2}{\lengw}[\worcw]%
	|\preficw\worcw\suficw\right]}
	\fi
	\endgroup
}

\def \Tp#1#2{
	\begingroup
	\expandarg
	\StrLen{#1}[\leng]
	\ifnum\leng<3
	{\sf T}'_{\left[#1
	\else
	\StrLeft{#1}{1}[\prefic]%
	\StrRight{#1}{1}[\sufic]%
	\StrSubstitute[1]{\prefic}{[}{ }[\prefic]%
	\StrSubstitute[1]{\sufic}{]}{ }[\sufic]%
	\StrSubstitute[1]{\prefic}{(}{ }[\prefic]%
	\StrSubstitute[1]{\sufic}{)}{ }[\sufic]%
	\StrMid{#1}{2}{\leng}[\auxword]%
	\StrLen{\auxword}[\leng]%
	\StrMid{#1}{2}{\leng}[\worc]%
	{\sf T}'_{\left[\prefic\worc\sufic
	\fi
	\StrLen{#2}[\lengw]%
	\ifnum\lengw<3
	|#2\right]}
	\else
	\StrLeft{#2}{1}[\preficw]%
	\StrRight{#2}{1}[\suficw]%
	\StrSubstitute[1]{\preficw}{[}{ }[\preficw]%
	\StrSubstitute[1]{\preficw}{(}{ }[\preficw]%
	\StrSubstitute[1]{\suficw}{]}{ }[\suficw]%
	\StrSubstitute[1]{\suficw}{)}{ }[\suficw]%
	\StrMid{#2}{2}{\lengw}[\auxwordw]%
	\StrLen{\auxwordw}[\lengw]%
	\StrMid{#2}{2}{\lengw}[\worcw]%
	|\preficw\worcw\suficw\right]}
	\fi
	\endgroup
}
\def \TE#1#2{
	\begingroup
	\expandarg
	\StrLen{#1}[\leng]
	\ifnum\leng<3
	{\sf T}^E_{\left[#1
	\else
	\StrLeft{#1}{1}[\prefic]%
	\StrRight{#1}{1}[\sufic]%
	\StrSubstitute[1]{\prefic}{[}{ }[\prefic]%
	\StrSubstitute[1]{\sufic}{]}{ }[\sufic]%
	\StrSubstitute[1]{\prefic}{(}{ }[\prefic]%
	\StrSubstitute[1]{\sufic}{)}{ }[\sufic]%
	\StrMid{#1}{2}{\leng}[\auxword]%
	\StrLen{\auxword}[\leng]%
	\StrMid{#1}{2}{\leng}[\worc]%
	{\sf T}^E_{\left[\prefic\worc\sufic
	\fi
	\StrLen{#2}[\lengw]%
	\ifnum\lengw<3
	|#2\right]}
	\else
	\StrLeft{#2}{1}[\preficw]%
	\StrRight{#2}{1}[\suficw]%
	\StrSubstitute[1]{\preficw}{[}{ }[\preficw]%
	\StrSubstitute[1]{\preficw}{(}{ }[\preficw]%
	\StrSubstitute[1]{\suficw}{]}{ }[\suficw]%
	\StrSubstitute[1]{\suficw}{)}{ }[\suficw]%
	\StrMid{#2}{2}{\lengw}[\auxwordw]%
	\StrLen{\auxwordw}[\lengw]%
	\StrMid{#2}{2}{\lengw}[\worcw]%
	|\preficw\worcw\suficw\right]}
	\fi
	\endgroup
}
\def \TN#1#2{
	\begingroup
	\expandarg
	\StrLen{#1}[\leng]
	\ifnum\leng<3
	{\sf T}^N_{\left[#1
	\else
	\StrLeft{#1}{1}[\prefic]%
	\StrRight{#1}{1}[\sufic]%
	\StrSubstitute[1]{\prefic}{[}{ }[\prefic]%
	\StrSubstitute[1]{\sufic}{]}{ }[\sufic]%
	\StrSubstitute[1]{\prefic}{(}{ }[\prefic]%
	\StrSubstitute[1]{\sufic}{)}{ }[\sufic]%
	\StrMid{#1}{2}{\leng}[\auxword]%
	\StrLen{\auxword}[\leng]%
	\StrMid{#1}{2}{\leng}[\worc]%
	{\sf T}^N_{\left[\prefic\worc\sufic
	\fi
	\StrLen{#2}[\lengw]%
	\ifnum\lengw<3
	|#2\right]}
	\else
	\StrLeft{#2}{1}[\preficw]%
	\StrRight{#2}{1}[\suficw]%
	\StrSubstitute[1]{\preficw}{[}{ }[\preficw]%
	\StrSubstitute[1]{\preficw}{(}{ }[\preficw]%
	\StrSubstitute[1]{\suficw}{]}{ }[\suficw]%
	\StrSubstitute[1]{\suficw}{)}{ }[\suficw]%
	\StrMid{#2}{2}{\lengw}[\auxwordw]%
	\StrLen{\auxwordw}[\lengw]%
	\StrMid{#2}{2}{\lengw}[\worcw]%
	|\preficw\worcw\suficw\right]}
	\fi
	\endgroup
}
\def \To#1{
	\begingroup
	\expandarg
	\StrLen{#1}[\leng]
	\ifnum\leng<3
	{\sf T}^{\sf out}_{[#1]}
	\else
	\StrLeft{#1}{1}[\prefic]%
	\StrRight{#1}{1}[\sufic]%
	\StrSubstitute[1]{\prefic}{[}{ }[\prefic]%
	\StrSubstitute[1]{\sufic}{]}{ }[\sufic]%
	\StrSubstitute[1]{\prefic}{(}{ }[\prefic]%
	\StrSubstitute[1]{\sufic}{)}{ }[\sufic]%
	\StrMid{#1}{2}{\leng}[\auxword]%
	\StrLen{\auxword}[\leng]%
	\StrMid{#1}{2}{\leng}[\worc]%
	{\sf T}^{\sf out}_{[\prefic\worc\sufic]}
	\fi
	\endgroup
}
\def \ToN#1{
	\begingroup
	\expandarg
	\StrLen{#1}[\leng]
	\ifnum\leng<3
	{\sf T}^{N,\sf out}_{[#1]}
	\else
	\StrLeft{#1}{1}[\prefic]%
	\StrRight{#1}{1}[\sufic]%
	\StrSubstitute[1]{\prefic}{[}{ }[\prefic]%
	\StrSubstitute[1]{\sufic}{]}{ }[\sufic]%
	\StrSubstitute[1]{\prefic}{(}{ }[\prefic]%
	\StrSubstitute[1]{\sufic}{)}{ }[\sufic]%
	\StrMid{#1}{2}{\leng}[\auxword]%
	\StrLen{\auxword}[\leng]%
	\StrMid{#1}{2}{\leng}[\worc]%
	{\sf T}^{N,\sf out}_{[\prefic\worc\sufic]}
	\fi
	\endgroup
}
\def \ToE#1{
	\begingroup
	\expandarg
	\StrLen{#1}[\leng]
	\ifnum\leng<3
	{\sf T}^{E,\sf out}_{[#1]}
	\else
	\StrLeft{#1}{1}[\prefic]%
	\StrRight{#1}{1}[\sufic]%
	\StrSubstitute[1]{\prefic}{[}{ }[\prefic]%
	\StrSubstitute[1]{\sufic}{]}{ }[\sufic]%
	\StrSubstitute[1]{\prefic}{(}{ }[\prefic]%
	\StrSubstitute[1]{\sufic}{)}{ }[\sufic]%
	\StrMid{#1}{2}{\leng}[\auxword]%
	\StrLen{\auxword}[\leng]%
	\StrMid{#1}{2}{\leng}[\worc]%
	{\sf T}^{E,\sf out}_{[\prefic\worc\sufic]}
	\fi
	\endgroup
}
\def \Ti#1{
	\begingroup
	\expandarg
	\StrLen{#1}[\leng]
	\ifnum\leng<3
	{\sf T}^{\sf{in}}_{[#1]}
	\else
	\StrLeft{#1}{1}[\prefic]%
	\StrRight{#1}{1}[\sufic]%
	\StrSubstitute[1]{\prefic}{[}{ }[\prefic]%
	\StrSubstitute[1]{\sufic}{]}{ }[\sufic]%
	\StrSubstitute[1]{\prefic}{(}{ }[\prefic]%
	\StrSubstitute[1]{\sufic}{)}{ }[\sufic]%
	\StrMid{#1}{2}{\leng}[\auxword]%
	\StrLen{\auxword}[\leng]%
	\StrMid{#1}{2}{\leng}[\worc]%
	{\sf T}^{\sf{in}}_{[\prefic\worc\sufic]}
	\fi
	\endgroup
}
\def \TD#1{
	\begingroup
	\expandarg
	\StrLen{#1}[\leng]
	\ifnum\leng<3
	{\sf T}^{\Delta}_{#1}
	\else
	\StrLeft{#1}{1}[\prefic]%
	\StrRight{#1}{1}[\sufic]%
	\StrSubstitute[1]{\prefic}{[}{ }[\prefic]%
	\StrSubstitute[1]{\sufic}{]}{ }[\sufic]%
	\StrSubstitute[1]{\prefic}{(}{ }[\prefic]%
	\StrSubstitute[1]{\sufic}{)}{ }[\sufic]%
	\StrMid{#1}{2}{\leng}[\auxword]%
	\StrLen{\auxword}[\leng]%
	\StrMid{#1}{2}{\leng}[\worc]%
	{\sf T}^{\Delta}_{\prefic\worc\sufic}
	\fi
	\endgroup
}
\def \TT{{\sf T}}
\def \TTp{{\sf T}'}

\def \Sup{{\sf Supp}}

\def \Out{{\sf out}}

\def \S#1#2{
	\begingroup
	\expandarg
	\StrLen{#1}[\leng]
	\ifnum\leng<3
	{\sf S}_{\left[#1
	\else
	\StrLeft{#1}{1}[\prefic]%
	\StrRight{#1}{1}[\sufic]%
	\StrSubstitute[1]{\prefic}{[}{ }[\prefic]%
	\StrSubstitute[1]{\sufic}{]}{ }[\sufic]%
	\StrSubstitute[1]{\prefic}{(}{ }[\prefic]%
	\StrSubstitute[1]{\sufic}{)}{ }[\sufic]%
	\StrMid{#1}{2}{\leng}[\auxword]%
	\StrLen{\auxword}[\leng]%
	\StrMid{#1}{2}{\leng}[\worc]%
	{\sf S}_{\left[\prefic\worc\sufic
	\fi
	\StrLen{#2}[\lengw]%
	\ifnum\lengw<3
	|#2\right]}
	\else
	\StrLeft{#2}{1}[\preficw]%
	\StrRight{#2}{1}[\suficw]%
	\StrSubstitute[1]{\preficw}{[}{ }[\preficw]%
	\StrSubstitute[1]{\preficw}{(}{ }[\preficw]%
	\StrSubstitute[1]{\suficw}{]}{ }[\suficw]%
	\StrSubstitute[1]{\suficw}{)}{ }[\suficw]%
	\StrMid{#2}{2}{\lengw}[\auxwordw]%
	\StrLen{\auxwordw}[\lengw]%
	\StrMid{#2}{2}{\lengw}[\worcw]%
	|\preficw\worcw\suficw\right]}
	\fi
	\endgroup
}
\def \SS{{\sf S}}

\def \bF{{\bf F}}

\def \wZ{Z}

\def \ZnZ{\mathbb{Z}/n\mathbb{Z}}

\def \AI{{\sf AlgInv }}
\def \CAI{{\sf CycAlgInv }}
\def \LL{\mathbb{L}}

\def \W#1{
	\begingroup
	\expandarg
	\StrLen{#1}[\leng]
	\ifnum\leng<3
	{\sf W}_{#1}
	\else
	\StrLeft{#1}{1}[\prefic]%
	\StrRight{#1}{1}[\sufic]%
	\StrSubstitute[1]{\prefic}{[}{ }[\prefic]%
	\StrSubstitute[1]{\sufic}{]}{ }[\sufic]%
	\StrSubstitute[1]{\prefic}{(}{ }[\prefic]%
	\StrSubstitute[1]{\sufic}{)}{ }[\sufic]%
	\StrMid{#1}{2}{\leng}[\auxword]%
	\StrLen{\auxword}[\leng]%
	\StrMid{#1}{2}{\leng}[\worc]%
	{\sf W}_{\prefic\worc\sufic}
	\fi
	\endgroup
}

\def\cro#1{\llbracket#1\rrbracket}

\begin{document}

\begin{center}  \bf \LARGE Invariant measures of interacting particle systems: Algebraic aspects \\ \sf
   \Large~\\
  Luis Fredes \& Jean-François Marckert\\~\\
 \normalsize  CNRS, LaBRI, Universit\'e de Bordeaux, \\
  351 cours de la Libération\\
  33405 Talence cedex, France
  \end{center}
\begin{abstract}
Consider a continuous time particle system $\eta^t=(\eta^t(k),k\in \LL)$, indexed by a lattice $\LL$ which will be either $\Z$,  $\Z/n\Z$, a segment $\{1,\cdots, n\}$, or $\Z^d$, and taking its values in the set $\Ek^{\LL}$ where $\Ek=\{0,\cdots,\kappa-1\}$ for some fixed $\kappa\in\{\infty, 2,3,\cdots\}$. Assume that the Markovian evolution of the particle system (PS) is driven by some translation invariant local dynamics with bounded range, encoded by a jump rate matrix $\TT$. These are standard settings, satisfied by the TASEP, the voter models, the contact processes... The aim of this paper is to provide some sufficient and/or necessary conditions on the matrix $\TT$   so that this Markov process admits some simple invariant distribution, as a product measure (if $\LL$ is any of the spaces mentioned  above), as the law of a Markov process indexed by $\Z$ or $[0,n]\cap \Z$ (if $\LL=\Z$ or $\{1,\cdots,n\}$), or a Gibbs measure if $\LL=\ZnZ$. 
   \par
Multiple  applications follow: efficient ways to find invariant Markov laws for a given jump rate matrix or to prove that none exists. The voter models and the contact processes are shown not to possess any Markov laws as invariant distribution (for any memory $m$)\footnote{As usual, a random process $X$ indexed by $\mathbb{Z}$ or $\mathbb{N}$ is said to be a Markov chain with memory $m\in\{0,1,2,\cdots\}$ if for $\P(X_k \in A ~| X_{k-i}, i \geq 1)= \P(X_k \in A ~| X_{k-i}, 1\leq i\leq m)$, for any $k$.}. We also prove that some models close to these models do. We exhibit PS admitting hidden Markov chains as invariant distribution and design many PS on $\Z^2$, with jump rates indexed by $2\times 2$ squares, admitting product invariant measures.
\color{black}

\end{abstract}

\noindent {\bf Acknowledgements : } This works has been partially supported by ANR GRAAL (ANR-14-CE25-0014)

\section{Introduction}
\label{sec:intro}
\subsection*{Some notation }
We let $\N=\Z^+=\{0,1,2,\cdots\}$ and $\N^{\star}=\N\setminus \{0\}$. 
For $-\infty \leq a \leq b \leq +\infty$, define $\cro{a,b}:=[a,b]\cap \Z$ as the set of integers in $[a,b]$. We will call such a set a  $\Z$-interval. \par
If $J$ is a finite subset of $\Z^d$, then $x(J)$ stands for the sequence $(x_i,i\in J)$ sorted according to the lexicographical order of the indices, so that, for example, if $x_{(1,3)} = a$, $x_{(7,2)}=c, x_{(7,5)}=b$, then $x(\{(1,3),(7,5),(7,2)\})=(a,c,b)$. If $I$ is a $\Z$-interval, for example $I=\cro{3,6}$,  $x(I)=(x_3,x_4,x_5,x_6)$, and we will often write $x\cro{3,6}$ instead. \par
 
If $E$ is a set and $I$ a subset of $\Z$, or a sequence in $\Z$, we denote by
\[E^I:=\{ x(I): \textrm{ the entries in }x(I) \textrm{ belong to } E \}\]
the set of sequences in $E$ indexed by $I$.
For $y=x(I)$, a sequence indexed by a set $I$, and for $A\subset \Z$, set  
\[y^{A} = x(I \setminus A),\]
the word obtained by suppressing the letters in position belonging to $A$ in $y$. Following the same idea, we denote by $M^{\{i\}}$ the matrix $M$ with the column and row $i$ suppressed.\\
For any set $E$, we denote by ${\cal M}(E)$ the set of probability measures on $E$ (for a topology which will be specified in the context).\\
A function $g:A\to \R$ is said to be equivalent to 0, we write $g\equiv 0$, if its image is reduced to 0.

\subsection{Models and presentation of results}
\label{sec:MPR}
All the results presented in this article (apart from Theorem \ref{theo:ppppp}) concern space and time homogeneous particle systems (PS), with finite range interactions defined on a lattice $\LL$, which will be $\Z$, $\Z/n\Z$, $\Z^d$, or a segment  $\cro{1,n}$. 
The set of colours is $\Ek=\cro{0,\kappa-1}$, where $\kappa$ (the number of colours) belongs to $\{2,3,\cdots\}\cup\{+\infty\}$. An element of the set of \textit{configurations} $\Ek^{\LL}$, is a colouring of the sites of $\LL$ by the elements of $\Ek$ (neighboring sites may have the same colour). When well defined, the PS will be a continuous time Markov process $\eta:=(\eta^t,t\geq 0),$ where for any $t$, $\eta^t=(\eta^t(k),k\in \LL) \in \Ek^\LL$. The set $\Ek^{\LL}$ is equipped with the product $\sigma$-algebra.  \par

The construction of the family of PS considered here is illustrated on $\Z$ first, but considerations for the analogues on $\Z/n\Z$, $\cro{1,n}$ and $\Z^d$ will appear progressively. 

\begin{defi}\label{defi:JRM} We call jump rate matrix (JRM) with range $L\in \N^\star$, a matrix  
  \ben\label{eq:JRM}
  \TT=\begin{bmatrix}\T{u}{v}\end{bmatrix}_{u,v\in \EkL},
  \een indexed by the size $L$ words on the alphabet $\Ek$, with non negative entries and with zeroes on the diagonal. 
\end{defi}
Assume for a moment {\bf that $\kappa$, the number of colours, is finite} and fix a JRM $\TT$ with range $L$.
With any element of the ``possible jumps set''
\ben\label{eq:J}
J= \l\{(i,w,w'), i \in \Z, w \in \Ek^L,w'\in \Ek^L\r\}=  \Z\times (\Ek^L)^2,
\een
where:\\
$\bullet$ $i$ encodes an abscissa  in an infinite word,\\
$\bullet$ $w$ and $w'$ encode respectively some size $L$ initial and final words,\\
associate the ``local map''
\ben\label{eq:miwwp}
\app{m_{i,w,w'}}{\Ek^\Z}{\EkZ}{\eta}{m_{i,w,w'}(\eta)},
\een
which:\\
-- if the subword  $\eta\cro{i+1,i+L} \neq w$ keeps $\eta$ unchanged (so that $m_{i,w,w'}(\eta)=\eta$)\\
-- if the subword  $\eta\cro{i+1,i+L}=w$, transforms this subword into $w'$ (formally: $m_{i,w,w'}(\eta)= \eta'$ with $\eta'_j=\eta_j$ if $j\notin \cro{i+1,i+L}$, and $\eta_{i+k}' = w'_k$,  the $k$th letter of $w'$ if $1\leq k \leq L$).

Define the generator
\ben\label{eq:gen}
(Gf)(\eta)=\sum_{(i,w,w')\in J} \T{w}{w'} \l[f(m_{i,w,w'}(\eta))-f(\eta)\r],
\een
acting on continuous functions $f$ sufficiently smooth, for example:\\
-- the set of bounded cylinder functions $g:\Ek^{\Z}\to \R$ (see e.g.  Kipnis \& Landim \cite[Section 2]{KL}) or,\\
-- following Liggett \cite{LTIPS} (starting p.21) or Swart \cite{Swart} (starting p.72), the class $C_\Delta$ of continuous functions $g:\Ek^{\Z}\to \R$ such that $\sum_x \Delta_g(x)<\infty$ for
\[\Delta_g(i)=\sup\l\{|g(\eta)-g(\xi)|, \eta,\xi\in \Ek^{\Z} \textrm{ and } \eta(j)=\eta(i), \forall j \neq i \r\}.\]
The sum \eref{eq:gen} represents a word $\eta$ indexed by $\Z$ whose size $L$ subwords jump: a subword equals to $w$ is transformed into $w'$ with rate $\T{w}{w'}$ (a jump is then possible only when $\T{w}{w'}>0$).
When $\kappa$ is finite, such a particle system is well defined (see references given above for all details).

Many such models have been studied in the literature, for example: \\
$\bullet$ \textit{The contact process}, for which $\kappa=2$, $L=3$, and all the entries of $\TT$ are 0 except $\T{(a,1,b)}{(a,0,b)}=1$ for any $(a,b)\in \{0,1\}^2$ (recovery rate),  $\T{(a,0,b)}{(a,1,b)}=\lambda (a+b)$ for some $\lambda>0$ the infection rate (the same model can be expressed using a JRM with range $L=2$ instead: $\T{(1,0)}{(1,1)}=\T{(0,1)}{(1,1)}=\lambda$, $\T{(1,1)}{(0,1)}=\T{(1,0)}{(0,0)}=1$).\\
$\bullet$ \textit{The voter model}, for which $\kappa=2$, $L=3$, $\T{(a,1-c,b)}{(a,c,b)}=\1_{c=b}+\1_{c=a}$ for any $(a,b)\in \{0,1\}^2$, the other entries of $\TT$ being 0: an individual makes its neighbors adopt its opinion after an exponential random time.\\
$\bullet$ \textit{The stochastic Ising model}, for which $\kappa=2$, $L=3$ and JRM $\TT$ with zero entries except for
\beq \label{eq:rsgsf}
\T{(a,b,c)}{(a,1-b,c)}=e^{-\beta (2b-1)(2a+2c-2)} \textrm{ for any }(a,b,c)\in \{0,1\}^3.
\eq
Here the state 1 represents a vertex on the line with positive magnetization, 0 a vertex with negative magnetization and $\beta$ a positive parameter, which, depending on its sign, favours or penalizes configurations in which vertices magnetization are aligned. \\
$\bullet$ \textit{The TASEP} on $\mathbb{Z}$ with $\kappa=2$, $L=2$, $\T{(1,0)}{(0,1)}=1$ and the others $\T{u}{v}$ being 0.\medskip

A distribution  $\mu$ on $\Ek^\Z$ is said to be \textit{invariant} by $\TT$ if $\eta^t\sim\mu$ for any $t\geq 0$, when $\eta^0\sim \mu$ (where the notation  $\sim$ means ``distributed as''). Following the discussion given below \eref{eq:gen}, this property can be rephrased  when $\kappa$ is finite, as  $\int Gfd\mu = 0$ for any $f$ bounded cylinder function $f$ (or function of $C_{\Delta}$). A simple argument (\cite[Lem. 1.3 p 23]{KL}) shows that it is also characterized by
$\int Gf\mu =0$ for any indicator function $f$ of the type
\ben\label{eq:IF} f(\eta)=\1_{\eta\cro{n_1,n_2}=x\cro{n_1,n_2}}\een for some fixed word $x\cro{n_1,n_2}$ and fixed indices $n_1\leq n_2$: this is the balance between the (infinitesimal) creation and destruction of the subword $x\cro{n_1,n_2}$ in the interval $\cro{n_1,n_2}$ under the distribution $\mu$. \par
Recall that under the product $\sigma$-algebra, a measure $\mu\in {\cal M}(\Ek^\Z)$ is characterized by its finite dimensional distributions.
\color{black}

\paragraph{ We are interested in the following question:} for what JRM $\TT$ does there exist a simple invariant distribution\,? Here the word ``simple'' stands for distributions as product measures, Markov laws or Gibbs measures (depending on the underlying graph where is defined the particle system).
It turns out that this question has a rich algebraic nature, and we then decided to focus on this question only.
The algebra in play depends on $\TT$ and on the fixed family of distributions whose invariance is under investigation. \medskip

Consider a function $f$ as given in \eref{eq:IF}.
The {\bf single} jumps of the PS that may affect the value of $f(\eta)$ take place in the \textit{dependence set} of $\cro{n_1,n_2}$ which is larger than $\cro{n_1,n_2}$:
\beq\label{def:DD}
D\cro{n_1,n_2}=\cro{n_1-(L-1),n_2+L-1}.
\eq
For any $w$ and $z$  in $\Ek^{\cro{n_1,n_2}}$, set the induced transition rate $\T wz$ from $w$ to $z$ as:
\ben\label{eq:Tind}
\T wz= \sum_{\cro{a+1,a+L}\subset \cro{n_1,n_2}}   \T{ w\cro{a+1,a+L}}{z\cro{a+1,a+L}} ~\1_{ w_j=z_j \textrm{ for all } j \in \cro{n_1,n_2}\setminus\cro{a+1,a+L}},
\een that is the sum of the transition rates which makes this transition possible in {\bf a single jump totally included} in $w$.
For a fixed pair $(w,z)$ the contribution of the  $\Z$-interval $\cro{a+1,a+L}$ is 0 if $\T{w\cro{a+1,a+L}}{z\cro{a+1,a+L}}=0$ (jump not allowed), or if $w$ and $z$ do not coincide outside $\cro{a+1,a+L}$. This includes the case where $n_2-n_1$ is too small, that is $<L-1$.\par
Notice that taking the same notation for the transition rate between two words as for the JRM is possible since they coincide if the lengths of $w$ and $z$ are both $L$.\par
We want to reformulate in a Lemma what has been said so far concerning the cases where $\kappa$ is finite:
\begin{lem}\label{lem:inva}Let $\kappa<+\infty$. A probability measure $\nu\in {\cal M}\l(\EkZ\r)$ is invariant under $\TT$ on the line if it solves the system of equations $\Sys{\Z,\nu,\TT}$ defined by 
	\beq\label{eq:mastereq}
	\left\{ \Lineq^{\Z}(x\cro{n_1,n_2},\nu)=0,~\textrm{ for any }n_1\leq n_2, \textrm{ for any }x\cro{n_1,n_2}\in \Ek^{\cro{n_1,n_2}}\right.,
	\eq
	where
	\beq\label{eq:lineq}
        \begin{array}{l}
          \Lineq^{\Z}(x\cro{n_1,n_2},\nu)=\dis\sum_{w,z \in E_{\kappa}^{D\cro{n_1,n_2}}} \l(\nu_{D\cro{n_1,n_2}}(w) \T wz- \nu_{D\cro{n_1,n_2}}(z) \T zw\r)\\
          \hspace{5.4 cm}\times 1_{z\cro{n_1,n_2}=x\cro{n_1,n_2}}.
          \end{array}
	\eq
  \end{lem}

We now define the notion of \bf algebraic invariance \rm of a probability measure with respect to a particle system. The aim of this notion is to disconnect the problem of well definition of a particle system which brings its own technical difficulties and obstructions when $\kappa=+\infty$ (see discussion in Section \ref{sec:OWD}) to the resolution of the systems \eref{eq:mastereq} which is ``just'' an algebraic system, which can be solved independently from other considerations.  
\begin{defi} For $\kappa$ finite or infinite, 
  a probability measure $\nu\in {\cal M}\l(\EkZ\r)$ is said to be algebraically invariant under $\TT$ on the line (we write $\nu$ is \AI by $\TT$ on the line) if it solves the system of equations \eref{eq:mastereq}.
\end{defi}

Again, in the case where $\kappa<+\infty$, standard invariance of measures and algebraic invariance are equivalent notions. When $\kappa=+\infty$, difficulties arise (see Section \ref{sec:OWD}) and the notion of algebraic invariance is indeed useful.  
\paragraph{Extension on $\Z/n\Z$.}
The previous considerations for PS $\eta$  indexed by $\Z$ can be extended to $\Z/n\Z$ (the finitness of $\Z/n\Z$ provides a more favourable setting). 
\begin{lem}Let $\kappa$ be finite. A probability measure  $\mu_n \in {\cal M}\l(\Ek^{\ZnZ}\r)$ is invariant under $\TT$ on the circle of length $n$ if
\beq\label{eq:mastereqCyl2}
\Sys{\ZnZ,\nu,\TT}~:=\left\{ \Cycle_n(x,\mu_n)=0,\textrm{ for any }x\in \Ek^{\ZnZ}\right.
\eq
for
\ben\label{eq:yc}
\Cycle_n(x,\mu_n)=\dis\sum_{w \in E_{\kappa}^{\ZnZ}} \mu_{n}(w) \T wx- \mu_{n}(x) \T xw,
\een
where $\T{w}z$ has to be adapted to fit with the structure of $\ZnZ$~:
\beq \label{eq:sfe}
\T wz= \sum_{{\cro{a+1,a+L}}\subset \ZnZ}   \T{ w{\cro{a+1,a+L}}}{z{\cro{a+1,a+L}}} 1_{ w_j=z_j \textrm{ for all } j \in (\ZnZ) \setminus {\cro{a+1,a+L}}},
\eq
where in this context, ${\cro{a+1,a+L}}$ stands for $(a+1\mod n, \dots, a+L \mod n)$.
\end{lem}
When $\kappa$ is finite, the existence of a measure $\mu_n$ solving the system \eref{eq:mastereqCyl2} is granted from the theory of finite state space Markov processes.\par
Again, we disconnect the problem of existence of particle systems with the solution of the algebraic system: 
\begin{defi}\label{eq:cycy}
For $\kappa$ finite or infinite, we say that a probability measure $\mu_n \in {\cal M}\l(\Ek^{\ZnZ}\r)$ is cyclically algebraic invariant  under $\TT$ on the circle of length $n$ (we write $\mu_n$ is \CAI by $\TT$ on the circle of length $n$) if it solves $\Sys{\ZnZ,\nu,\TT}$ as stated in \eref{eq:mastereqCyl2}.
\end{defi}
Invariance and algebraic invariance are equivalent when $\kappa<+\infty$.

\subsubsection{The results.}
\label{sec:tr}
\begin{defi}\label{defi:Mar}
\tb For $-\infty < a \leq b <+\infty$, a process $(X_k,k\in \cro{a,b})$ is said to be a Markov chain on $\Ek$, or to have a Markov law, if there exists $M:=\begin{bmatrix}M_{i,j}\end{bmatrix}_{i,j\in \Ek}$, a Markov kernel (we will say also simply kernel), and an initial distribution $\nu \in{\cal M}(\Ek)$  such that,
 \[\P(X_{k}=x_k, a\leq k \leq b )= \nu_{x_a}\prod_{j=a}^{b-1} M_{x_j,x_{j+1}},\textrm{ for any } x \in \Ek^{\cro{a,b}}.\]
 For short, we will say that $X$ (resp. $\mu$) is a $(\nu,M)$-Markov chain on $\cro{a,b}$ (resp. $(\nu,M)$-Markov law) if its kernel is $M$, and its initial distribution is $\nu$.  \par
\tb We will say that a law $\rho$ in ${\cal M}(\Ek)$ is invariant for $M$ (or for this Markov chain) if $\rho M=\rho$, for $\rho$ seen as a row vector. If the initial distribution is $\rho$, we say that $X$ is a $M$ Markov chain under (one of) its invariant distribution. \par
\tb For $\rho\in{\cal M}(\Ek)$ invariant for $M$, we call $(\rho,M)$-Markov chain $(X_k, k\in \Z)$ a process indexed by $\Z$ whose finite dimensional distribution are given by $\P(X_{k}=x_k, a\leq k \leq b )= \rho_{x_a}\prod_{j=a}^{b-1} M_{x_j,x_{j+1}},\textrm{ for any } x \in \Ek^{\cro{a,b}}$. Its distribution is called $(\rho,M)$-Markov law.
\end{defi}
-- A $M$-Markov law on $\Ek$ is said to be \it positive recurrent \rm if under this kernel, a Markov chain is {\bf positive recurrent} (we will say also that $M$ is positive recurrent).\\
-- If all the $M_{i,j}$'s are  positive, we say that $M$ is {\bf positive}, and write $M>0$. \par
Consider a Markov chain with kernel $M$ on $\Ek$ under its invariant distribution $\rho$.

Let us define 
\beq\label{eq:lineq2222}
\Lineq_n^{\rho,M,\TT}(x\cro{1,n}):=\Lineq^\Z(x\cro{1,n},\nu,\TT), \textrm{ for any }x\cro{1,n}\in\Ek^n
\eq
where $\nu(a\cro{1,m})=\rho_{a_1}\prod_{j=1}^{m-1}M_{a_j,a_{j+1}}$: in words, $\Lineq^{\rho,M,\TT}_n(\cdot)$ is the function which coincides with $\Lineq^\Z(\cdot,\nu,\TT)$ on $\Ek^n$ when $\nu$ is the $(\rho,M)$-Markov law (see Definition \ref{defi:Mar}). 

The system of equations $\{\Lineq^{\rho,M,\TT}_n \equiv 0 \textrm{, for any }n\}$, (as stated in \eref{eq:erhrehr2}) provides the necessary and sufficient algebraic relations between $\rho,M$ and $\TT$ for the $\AI$ of the $M$-Markov law. This is an infinite system of equations even when $\Ek$ is finite. It is linear in $\TT$, with unbounded degree in $M$. \medskip 

\noindent{\bs}  The first goal of this article is to produce an equivalent {\bf finite system} of algebraic equations to characterize the invariance of $(\rho,M)$-Markov law  by $\TT$ when {\bf the set $\Ek$ is finite.} The main result is the proof of equivalence of $\{\Lineq^{\rho,M,\TT}_n\equiv 0, \textrm{ for any } n\}$ with each of several (equivalent) algebraic systems of degree 6 in $M$ and linear in $\TT$ (Theorem \ref{theo:t0}, and Theorem \ref{theo:t2}, when the range is $L=2$ and the memory of the Markov chain is $m=1$).   These equivalent systems are finite, and moreover, they can be explicitly solved using some linear algebra arguments (Theorem \ref{theo:cand3}): in words, it is possible to decide if a PS with JRM $\TT$ possesses an invariant Markov law, or to describe the class of all $\TT$ that do (which provide some  applications discussed in Section \ref{sec:ap}). 
\begin{itemize}
\item[--] When the cardinality of $\Ek$ is infinite some additional complications arise (Section \ref{sec:kappainfinite}), but some results still hold.
\item[--] When $M$ possesses some zero entries, a plurality of algebraic behaviours for these systems of equations (and solutions) makes a global approach probably impossible (Section \ref{sec:relax}).
\end{itemize}
\noindent{\bs} Similar criteria are developed to characterize product measures $\rho^\Z$ invariant by $\TT$. In this case the finite representations use equations of degree 3 in $\rho$ and linear in $\TT$ (Theorem \ref{theo:t3b}, when the range $L=2$).\\
\noindent{\bs} The invariance of the Gibbs distribution with kernel $M$ on the circle $\Z/n\Z$ is also studied, when $\Ek$ is finite.
In Theorem \ref{theo:t0} the equivalence between the invariance of a Gibbs measure (see Definition \ref{defi:Gibbs}) with Markov kernel $M$ on $\Z/n\Z$ for $n=7$ with the invariance of the $(\rho,M)$-Markov law (for $\rho$ such that $\rho M=\rho)$ on the line $\Z$ is established  (Theorem \ref{theo:t0}). Besides, Corollary \ref{cor:t444} implies that if the Gibbs distribution with kernel $M$ is invariant by $\TT$ on $\Z/n\Z$ for $n=7$, then it is also invariant by $\TT$ on $\Z/n\Z$ for any $n\geq 3$ (when the range is $L=2$).\\
\noindent{\bs} When considering a PS indexed by the segment $\cro{1,n}$, some interactions $\beta^r$ and $\beta^\ell$ with the boundaries are introduced (Section \ref{sec:SB}). When the range $L=2$, if a Markov law is invariant for $n\geq 7$ on the segment (with fixed boundaries interactions), then it is invariant on the line (Theorem \ref{theo:segm}). Some relations between invariant measures on the line and on the segment are provided. 
\\
\noindent{\bs} The 2D case and beyond will be discussed in Theorem \ref{theo:2D}, where a simple necessary and sufficient condition for the invariance of a product measure will be provided (Section \ref{sec:G2D}).\\ 
\noindent{\bs} The case where $\TT$ has a larger range $L$ and/or where the invariant distribution is a Markov law with larger memory $m$ is discussed in Section \ref{sec:Ext}. 

Many extensions discussed in Section \ref{sec:Ext} to larger range and memory, are proved by the same ideas as those for $L=2$, with some extra technical complications. We think that the presentation of the proof in the case $L=2$ is needed in order to make the arguments understandable.

\subsubsection{Applications.}
\label{sec:ap}
As said above, the theorems we provide allow one to decide if there exists a Markov law  with kernel $M$ (with memory $m$) invariant under the dynamics of a PS with a given $\TT$. This is done ``by explicitly'' solving a finite polynomial system with ``small degree in $M$''. These kinds of problems are solved using some algebra, for example, the computation of a Gr\"obner basis (see Section \ref{sec:Grobn}), using some Computer algebra systems if needed.  The theorems also allow to find pairs $(\TT,M)$ for which this invariance occurs, and then, to design some PS having a simple known invariance distribution. \par
  Hence, having in hands a simple algebraic characterization of PS admitting invariant Markov law, allows to extend considerably the family of PS for which explicit invariant distributions can be found, and we think that, as illustrated by what we are saying below, the interest of these results  go far beyond invariant Markov laws.

In the sequel, when we say that we use a specific model with general rates, we mean that we let the positive rates as free variables. In addition to the results presented in the preceding section we present here several applications of our work.\\
  {\bs} In Section \ref{sec:MTSM2}, we prove that the voter models does not admit any Markov law of any memory as invariant distribution.  The general rate version is explored and the parameters for which there exist Markov law invariant on the line are discussed. \\
{\bs} In Section \ref{sec:MTSMcontact}, the contact process is discussed: we prove that this process does not have a Markov law of any memory $m\geq 0$ as invariant distribution.\\
{\bs} In Section \ref{sec:MTSM}, the TASEP and some variants are explored: Zero-range type processes, 3 colours TASEP and PushASEP.\\
-- For the zero range type processes we prove that there exists a family of distributions $F$, such that depending on $\TT$, either all the product measures $\rho^\Z$ are invariant by $\TT$ for all $\rho\in F$, or none of them is invariant by $\TT$.\\
  -- In the general rate 3-colour TASEP some sufficient and necessary conditions on $\TT$ are given so that there exists a Markov law with positive kernel $M$ that is invariant by $\TT$.\\
-- For the PushASEP we explain how some special types of $PS$ with range $L=\infty$ can be transformed and solved with our results.\\
{\bs} In Section \ref{sec:MTSMIsing}, the stochastic Ising model is analyzed and its well known Markov invariant measure on the line (Gibbs on the cycle) is found based on our results.\\
{\bs} The possibility offered by our theorems to find automatically parameters $(T,M)$, say, on the space $E_3=\{0,1,2\}$ (with 3 colours) and $L=2$ for which the PS with JRM $\TT$ let the Markov law with kernel $M$ invariant, allows to find some PS on $E_2=\{0,1\}$ with 2 colours and $L=3$ which possesses some hidden Markov chain distributions as invariant distributions, using some projection from $E_3$ to $E_2$. As far as we are aware of, this is the first time that a hidden Markov chain is shown to be invariant under a PS on the line. 
This is discussed in Section \ref{sec:HMC}. We think that this method will allow in the future to find many invariant distribution for PS with 2 colours, or more.\\
{\bs} In Section \ref{sec:AIMD}, the set of pairs $(\TT,M)$ for which the Markov law  with positive kernel $M$ is invariant under $\TT$, in the case $\kappa=2$ and $L=2$ is totally explicitly solved. This case corresponds to standard PS on the line, where 1 and 0 are used to model the presence, or absence of particles at each position. Under these assumptions and mass preservation (see Def. \ref{def:mp}) we prove that the unique Markov kernels that are AI by this type of $\TT$'s are the i.i.d. measures. \\
{\bs} In Section \ref{sec:AIMD2},  the set of pairs $(\rho,\TT)$ for which the product measure with marginal $\rho$ is invariant under $\TT$, in the case $\kappa=2$ and $L=2$ is totally explicitly solved.\\
{\bs} In Section \ref{sec:2dapp} we use our criteria of invariance of product measures under the dynamics of a PS defined on $\Z^2$, to provide many explicit PS admitting product measures as invariant measure.

\color{black}

\subsection{Some pointers to related papers}
\label{sec:OWD}
Given an infinitesimal generator (or a JRM) of a particle system, the existence of a stochastic Markov process with this generator can be proved when the number of colour is finite, or if $\sup_{w}\sum_{w'}\T{w}{w'}<+\infty$ 
is uniformly bounded, using for example the  so-called graphical representation due to Harris \cite{HT} see also Swart \cite{Swart}, or by the Hille-Yosida theorem and other considerations coming from functional analysis and measure theory (see e.g. Liggett \cite{LTIPS}, Swart \cite{Swart},  Kipnis \& Landim \cite{KL}) see also  Andjel \cite{AE}, where proofs of existence and construction can be found in some particular cases. \par
When the state space $\Ek$ is infinite, complications arise since even the state at a site may diverge in a finite time, and then, in general a JRM $\TT$ does not allow to define a particle system properly. Sufficient condition for the well definition of a particle system in the infinite case can be found in Liggett \cite[Chap. IX]{LTIPS},  Kipnis \& Landim\cite{KL}, Bal{\'a}zs \& al. \cite{BM2}, Andjel \cite{AE},  Fajfrová \& al. \cite{FGS}.

Other works related to the present one concern the computation of invariant distribution(s) of a given PS, or the characterization of its ergodicity (Blythe \& Evans \cite{BE}, Crampe \& al. \cite{CRV}, Fajfrová \& al. \cite{FGS}, Greenblatt \& Lebowitz \cite{GL}). 
Numerous results concern works not directly related to the present paper:  study of PS out of equilibrium, their speed of convergence, their time to reach a certain state, among others. 

As far as we are aware, the paper whose point of view is the closest to the present work, is  Fajfrová \& al. \cite{FGS}, in which some conditions for the invariance of product measures are designed, for mass migration processes (see  def. \ref{def:mp2} and below).\medskip   

We add that the present work has been inspired by some similar works on probabilistic cellular automata, where the transition matrices for which simple invariant measures exist, have been  deeply investigated, and are at the heart of the theory, Toom \& al. \cite{dkt},  Dai Pra \& al. \cite{PLR}, Marcovici \& Mairesse \cite{1207.5917}, Casse \& the second author \cite{CM}.

\subsection{Main results}

The case $L=1$ being non interesting here, we examine in details the case where the range is $L=2$ ,  representative of this kind of models as will be seen in  Section \ref{sec:Ext} where larger ranges will be investigated.

\par
For a given  JRM $\TT$, the \textit{exit rate out} of $w\in \Ek^2$ is defined by
\be
\To{w}=\sum_{w' \in \Ek^2}\T{w}{w'}.
\ee  
Consider a Markov chain with kernel $M$, and let $\rho$ be one of its invariant distribution.  The equation $\Lineq_n^{\rho,M,\TT}(x\cro{1,n})=0$ (as defined in \eref{eq:lineq2222}) rewrites 
\beq \label{eq:erhrehr2}
\begin{array}{l}
\dis\sum_{x_{-1},x_0,\atop x_{n+1},x_{n+2}}\sum_{j=0}^{n} 
 \sum_{u,v}\T{(u,v)}{(x_j,x_{j+1})} \rho_{x_{-1}}\Big(\!\!\!\!\prod_{-1\leq k \leq n+1\atop{k\not\in\{j-1,j,j+1\}}}M_{x_k,x_{k+1}}\Big)M_{x_{j-1},u}M_{u,v}M_{v,x_{j+2}}\\
 -\dis\sum_{x_{-1},x_0,\atop x_{n+1},x_{n+2}} \Big(\rho_{x_{-1}}\prod_{k=-1}^{n+1}M_{x_k,x_{k+1}}\Big) \sum_{j=0}^{n} \To{x_j,x_{j+1}}=0.
 \end{array}
\eq
From Lemma \ref{lem:inva}, a $(\rho,M)$ Markov law under its invariant distribution is  invariant by $\TT$ on the line when  $\Lineq^{\rho,M,\TT}_n\equiv 0$, for all $n\in \N$.

Since the range is $L=2$, the value of $x_0$ and $x_{n+1}$ ``just outside'' $x\cro{1,n}$ play a role (they are in the dependence set of $\cro{1,n}$, as defined in Def. \ref{def:DD}) we then need to sum on all the possible values of $(x_0,x_{n+1})$. But, because of the appearance of the pattern $M_{x_{i},u}M_{u,v}M_{v,x_{i+3}}\T{(u,v)}{(x_{i+1},x_{i+2})}$, it is a bit simpler to consider also additionally the extra values $(x_{-1},x_{n+2})$ in the sum even if they are not in the dependence set: these additional terms concern only the representation of the  Markov law, and also the fact that $\rho$ is the invariant distribution of $M$ (not the JRM). 

We now present the main theorems of the paper. The proofs that are not given in this section, are postponed to Section \ref{sec:proofs}.

\subsubsection{Invariant Markov laws with positive kernel}

In this section, $\Ek$ is finite, and the Markov kernel $M=\begin{bmatrix} M_{i,j}\end{bmatrix}_{i,j\in \Ek}$ has positive entries. The measure $\rho$ is the invariant law for a Markov chain with kernel $M$, and is characterized by $\rho=\rho M$.

Define  the normalized version of $\Lineq$ by:
\beq\label{eq:tjht}
\NLineq_n^{\rho,M,\TT}(x):= \frac{\Lineq_n^{\rho,M,\TT}(x)}{\prod_{j=1}^{n-1}M_{x_j,x_{j+1}}},
\eq
so that, for $n=1$, and any $x\in\Ek$, $\NLineq_1^{\rho,M,\TT}(x): = \Lineq_1^{\rho,M,\TT}(x)$ and for $n\geq 2$, and any $x\in \Ek^n$, 
\beq
\label{eq:rsgfqe12} \NLineq_n^{\rho,M,\TT}(x) = \sum_{x_{-1},x_0,\atop x_{n+1},x_{n+2}}  \l(\rho_{x_{-1}}\prod_{k\in\{ -1,0,n,n+1 \}}M_{x_k,x_{k+1}}\r) \sum_{j=0}^{n} \wZ^{M,\TT}_{x\cro{j-1,j+2}},
\eq
with
\ben\label{eq:rgdq}
\wZ_{a,b,c,d}^{ M,\TT}:=\l(\sum_{(u,v)\in \Ek^2} \T{(u,v)}{(b,c)} \frac{M_{a,u}M_{u,v}M_{v,d}}{M_{a,b}M_{b,c}M_{c,d}}\r)-\To{b,c}.\een
 We will drop the exponents $M,\TT$ and write $\wZ_{a,b,c,d}$ instead when they are clear from the context.
\begin{rem}[Key point] One can say that the leading idea of the paper is the following: when a Markov law $(\rho,M)$ is invariant by $\TT$ then $\wZ$ possesses a huge amount of nice algebraic additive properties: this will be seen in all the theorems of the paper. Some of the additive properties of $\wZ$ will allow to control $\NLineq^{\rho,M,\TT}$ and show its nullity. The more natural object $\Lineq^{\rho,M,\TT}$ from a probabilistic perspective (without normalisation) is not the right object to deal with these additive properties.
\end{rem}

Now, for $u\cro{1,\ell}$ a $\ell$-tuple of elements of $\Ek$,  denote by
\be
\Sub{u\cro{1,\ell}}{k}=\l\{ u\cro{m+1,m+k}, 0\leq m \leq \ell-k \r\}
\ee
the multiset\footnote{a multiset is a set in which elements may have multiplicities $\geq 1$} ``of $k$-subwords'' of $u\cro{1,\ell}$ so that, for example 
\[\Sub{a\cro{1,7}}{4}=\{~a\cro{1,4},~a\cro{2,5},~a\cro{3,6},~a\cro{4,7}~\}.\]
Define the map $\ME_7^{M,\TT}:\Ek^7 \to \R$ by 
\beq\label{eq:Meq}
\ME_7^{M,\TT}(a\cro{1,7})=\sum_{w \in \Sub{a\cro{1,7}}{4}} \wZ_w -\sum_{w \in \Sub{a\cro{1,7}^{\{4\}}}{4}} \wZ_w
\eq
where (following our notation, below the abstract) $a\cro{1,7}^{\{4\}}=(a_1,a_2,a_3,a_5,a_6,a_7)$. The map $\ME_7^{M,\TT}$ will play an important role in the sequel. 
Let us expand for once, this compressed notation: 
\be
\ME_7^{M,\TT}(a\cro{1,7})&=&Z_{a_1,a_2,a_3,a_4}+Z_{a_2,a_3,a_4,a_5}+Z_{a_3,a_4,a_5,a_6}+Z_{a_4,a_5,a_6,a_7}\\
&&-Z_{a_1,a_2,a_3,a_5}-Z_{a_2,a_3,a_5,a_6}-Z_{a_3,a_5,a_6,a_7}.\ee

Now, extend the notion of subwords to $\ZnZ$: for $a\cro{0,n-1} \in \Ek^{\ZnZ}$, write 
\be
\Subn{a\cro{0,n-1}}{k}{\ZnZ}=\l\{a\cro{m+1\mod n, m+k\mod n}, 0\leq m \leq n-1\r\}
\ee
for the multiset formed by the $n$ words made of $k$ successive letters of $a\cro{0,n-1}$ around $\Z/n\Z$. 
Define the map $\Cycle_n^{M,\TT}: \Ek^{\ZnZ}\to \R^+$ by   
\ben\label{eq:CyMT}
\Cycle^{M,\TT}_n(x):=\dis\sum_{w \in E_{\kappa}^{\ZnZ}} \l(\prod_{j=0}^{n-1} M_{w_i,w_{i+1\mod n}}\r)\T wx- \l(\prod_{j=0}^{n-1} M_{x_i,x_{i+1\mod n}}\r) \T xw.
\een
This formula coincides with $\Cycle_n(x,\mu_n)$ given in Defi. \ref{eq:cycy} for a Gibbs measure with kernel $M$:
\begin{defi}\label{defi:Gibbs} A process  $(X_k,k\in \ZnZ)$ indexed $\ZnZ$ for some $n\geq 1$ and taking its values in $\Ek^{\ZnZ}$ is said to have a Gibbs measure with kernel $\begin{bmatrix} M_{i,j} \end{bmatrix}_{i,j\in \Ek}$, a non negative Matrix, if
\be
\P(X\cro{0,n-1}=x\cro{0,n-1})&=& \frac{\prod_{j=0}^{n-1} M_{x_j,x_{j+1 \mod n}}}{\Trace(M^n)},\textrm{ for any } x\cro{0,n-1}\in \Ek^{\cro{0,n-1}}.
\ee
For short, we will say that $X$ follows the $M$-Gibbs measure on $\ZnZ$.
\end{defi}
Perron-Frobeni\"us theorem asserts that if a square matrix $A$ is non negative and irreducible, then $A$ has a real eigenvalue $\lambda$ larger (or equal, if $A$ is periodic) than the modulus of the other ones, and the corresponding right and left eigenvectors may be chosen with positive entries. We qualify by ``main'' in the sequel these eigenvectors and eigenvalue. Hence, if $M$ is irreducible, one may suppose w.l.o.g that $M$ is a classical Markov kernel, since $\begin{bmatrix} M'_{i,j} \end{bmatrix}_{i,j\in \Ek}=\begin{bmatrix}M_{i,j} q_i/(c q_j) \end{bmatrix}_{i,j\in \Ek}$, where $q$ is the main right eigenvector of $M$ and $c$ is the corresponding eigenvalue of $M$, is a Markov kernel which induces the same Gibbs measure as $M$.\par
When $\#\Ek<+\infty$, a $M$-Gibbs measure is invariant by $\TT$ on $\ZnZ$ iff $\Cycle^{M,\TT}_n\equiv 0$. Again, when $\#\Ek=+\infty$, independently of the good definition of the PS with JRM  with kernel $M$, we will say that cyclically algebraic invariant (or \CAI) by $\TT$ on $\ZnZ$ when  $\Cycle^{M,\TT}_n\equiv 0$.

It must be noticed at this point that $\Cycle^{M,\TT}_n(x)$ coincides with $\Cycle_n(x,\mu_n)$ as defined in \eref{eq:yc} where $\mu_n$ is the $M$-Gibbs measure with kernel $M$. 
Further for any $n\geq 1$, define the map $\NCycle_n^{M,\TT}:\Ek^{\ZnZ}\to \R$ by
\[
\NCycle_n^{M,\TT}(x)=\frac{\Cycle_n^{M,\TT}(x)}{\prod_{j=0}^{n-1}M_{x_i,x_{i+1 \mod n}}}.
\]

By inspection, it can be checked that, for $n\geq 3$, diving \eref{eq:CyMT} by $\prod_{j=0}^{n-1} M_{x_i,x_{i+1\mod n}}$ gives
\ben\label{eq:CyZ}
\NCycle_n^{M,\TT}(x)= \sum_{w \in \Subn{x}{4}{\ZnZ}} \wZ_w,~~\textrm{ for any }x\in \Ek^{\ZnZ}.
\een
This identity fails for $n=1$ and $n=2$.

Also define the map $\Rep_7^{M,\TT}:\Ek^7\times \Ek \to \R$ by
\[
\Rep_7^{M,\TT}(a\cro{1,7}; a_4')=\sum_{w \in \Sub{a\cro{1,7}}{4}} \wZ_w -\sum_{w \in \Sub{a_1a_2a_3a_4'a_5a_6a_7}{4}} \wZ_w.
\]
In fact
\[\Rep_7^{M,\TT}(a\cro{1,7}; a_4')=\NCycle_7^{M,\TT}(a\cro{1,7})-\NCycle_7^{M,\TT}(a_1a_2a_3a_4'a_5a_6a_7).\]It is then the balance in $\NCycle_7^{M,\TT}(a\cro{1,7})$ when the ``central letter'' $a_4$ of a word $a\cro{1,7}$
is replaced by $a'_4$. 

A key result of the paper is the following: the infinite system of equations $\{\Lineq^{\rho,M,\TT}_n\equiv 0, n\geq 1\}$, which by definition is the invariance of the Markov law by $\TT$ on the line is equivalent to many different finite systems of equations with bounded degree (in  $M$):
\begin{theo} \label{theo:t0} Let $\Ek$ be finite and $L=2$. 
If $M>0$ then the following statements are equivalent:
\bir
\itr $(\rho,M)$ is invariant by $\TT$ on the line.
\itr $\Rep_7^{M,\TT}(a,b,c,d,0,0,0;0) = 0$ for all $a,b,c,d \in \Ek$.
\itr $\Rep_7^{M,\TT} \equiv 0$.
\itr $\ME_7^{M,\TT}(a,b,c,d,0,0,0)=0$ for all $a,b,c,d \in \Ek$.
\itr $\ME_7^{M,\TT}\equiv0$.
\itr $\NCycle_n^{M,\TT}\equiv0$ for all $n\geq 3$.
\itr $\NCycle_7^{M,\TT}\equiv0$.
\itr $\NCycle_7^{M,\TT}(a,b,c,d,0,0,0)= 0$ for all $a,b,c,d \in \Ek$.
\itr There exists a function ${\sf W}:\Ek^3 \to \R$ such that $\wZ^{M,\TT}_{a,b,c,d} = \W{b,c,d}-\W{a,b,c}$.
\eir~
\end{theo}
The implication $(vii)\imp (vi)$ and $(vii)\imp(i)$ gives the following Corollary:
 \begin{cor}\label{cor:t444}
 Let $\Ek$ be finite. If the Gibbs measure with a positive Markov kernel $M$ is invariant on $\Z/n\Z$ by $\TT$ for $n=7$, then it is also invariant on $\Z/n\Z$ by $\TT$ for every $n\geq 3$, and the $M$-Markov law  under its invariant distribution is invariant by $\TT$ on the line.
 \end{cor}
\begin{rem}
 \noindent {\bs} The appearance of ``0'' everywhere in the Theorem is arbitrary. It may be replaced by any constant element of $\Ek$ in the previous statements.\\ 
\noindent {\bs} The positivity of $M$ is a strong condition whose relaxation entails many difficulties. It is discussed in Section \ref{sec:relax}.\\
\noindent {\bs} [link with reversibility]  The condition
  \begin{equation}\label{eq:etj}
    \T{(u,v)}{(b,c)}M_{a,u}M_{u,v}M_{v,d}= M_{a,b}M_{b,c}M_{c,d}\T{(b,c)}{(u,v)} \textrm{ for any }a,b,c,d,u,v\in \Ek \end{equation}
  is equivalent to the fact that PS with JRM $\TT$ is reversible with respect to the Gibbs measure with kernel $M$ on any cylinder with size $\geq 3$. As usual reversibility implies invariance.  However, invariance and reversibility are not equivalent even for Gibbs measures: Theorem \ref{theo:t0} gives the complete picture. In particular, \eref{eq:etj} implies $\wZ_{a,b,c,d}^{M,\TT}\equiv 0$, which implies Conditions $(ii)$ to $(ix)$ of Theorem \ref{theo:t0}. The converse does not hold.\\
\noindent {\bs} Further in the paper, we will state  Theorem \ref{eq:fqqds} which implies a result somehow stronger that Theorem \ref{theo:t0} in some conditions: $\NCycle_n^{M,\TT}\equiv 0$ for any $n\leq \kappa$ is necessary and sufficient for the Markov chain $(\rho,M)$ to be invariant by $\TT$. When the number of colours $\kappa<7$ this provides a criterion potentially simpler to check than those given in Theorem \ref{theo:t0}.
\end{rem}

We state here a theorem which is important in many applications.
Consider a PS with JRM $\TT$ defined on $\Z$, and its analogue on $\Ek^{\ZnZ}$. For $a$ and $b$ two elements of this last configuration set,  $b$ is said to be accessible from $a$, if $a=b$ or if it is possible to go from $a$ to $b$ using jumps with positive rates. 
A strict subset $S$ of $\Ek^{\ZnZ}$ is said to be absorbing, if
for any $b$ in $\Ek^{\ZnZ}$, an element of $S$ is accessible from $b$, and if  $\Ek^{\ZnZ}\setminus S$ is not accessible from $S$.
\begin{theo}\label{theo:fs} Consider a finite alphabet $\Ek$ with $|\kappa|\geq 2$. Consider $\TT$ a JRM with range $L$, such that $\TT$ is not identically 0. \\
Suppose that for infinitely many integers $n$ the PS with JRM $\TT$ possesses an absorbing subset $S_n$ of $\Ek^{\ZnZ}$, with $\varnothing \subsetneq S_n \subsetneq \Ek^{\ZnZ}$. Under these conditions, there does not exist any Markov law with memory $m$, for any $m$, with full support, invariant by $\TT$ on the line. 
\end{theo}
In fact only the case $m=1$ is a Corollary of Theorem \ref{theo:t0}, the strongest form for  general memory $m\geq 1$ and range $L$ is a Corollary of Theorem \ref{theo:t0prime} which treats the invariance of Markov law with memory $m$.
\begin{rem}
  {\bs} Notice that if $\TT$ is identically 0, then all states are absorbing states, then all Markov law are invariant.\\
  {\bs} If the hypothesis of the theorem holds for some fixed $n$, then the conclusion holds if the memory size $m$ satisfies $m+L\leq n$.
  \end{rem}

We will use this theorem in Section \ref{appl} for some applications on the contact process and the voter model.

\begin{proof} By Theorem \ref{theo:t0} (for $m=1$) or Theorem \ref{theo:t0prime} (for $m\geq 1$), if there exists a Markov law with memory $m$ and full support invariant for $\TT$ on the line, then the same property holds on $\ZnZ$ for $n\geq m+L$ for the corresponding Gibbs measure. But the invariance of a full support measure is incompatible with the existence of a non trivial absorbing subset.
\end{proof}
\color{black}
\begin{rem}[Crucial]\label{rem:kyrje} 
\tb $\ME_7^{M,\TT}\equiv 0$ is equivalent to
\beq\label{eq:sfjghdt}
\sum_{w \in \Sub{a\cro{1,7}}{4}} \wZ_w =\sum_{w \in \Sub{a\cro{1,7}^{\{4\}}}{4}} \wZ_w,~~\textrm{ for any } a\cro{1,7}\in \Ek^7.
\eq
This relation allows to see that the LHS of \eref{eq:sfjghdt}, does not depend on $a_4$, and in many places it will allow us to remove ``one letter'' in linear combinations involving $\wZ$: for any words $a\cro{1,n}$ with at least $n\geq 7$ letters, for any $4\leq k \leq n-3$, 
\[
\sum_{w \in \Sub{a\cro{1,n}}{4}} \wZ_w =\sum_{w \in \Sub{a\cro{1,n}^{\{k\}}}{4}} \wZ_w.
\]
This property is reminiscent to other algebraic properties, as rewriting systems, dependence in a vector space or as relation in the presentation of a group by generators and relations. \par 
\tb The fact that $\Rep_7\equiv 0$ is a necessary condition for the Markov law $(\rho,M)$ to be invariant by $\TT$ on the line appears naturally since it is the comparison of the balance of the outgoing and incoming rate of two similar words. The sufficiency of this condition is not obvious (see Section \ref{sec:relax} for extension when $M$ is not supposed positive).
\end{rem}

There are many links between the systems $\NCycle_n\equiv 0$ for different values of $n$, here are some of them, which prove that the Markov law with Markov kernel $M$ is invariant by $\TT$ if $(M,\TT)$ solves a system of equations with degree 6 in $M$, linear in $\TT$ (the system being finite when $\kappa$ is finite):
\begin{theo} \label{theo:t2} The system  $\NCycle_7^{M,\TT}\equiv 0$ is equivalent to:
	\bir
	\itr $\{\NCycle_6^{M,\TT}\equiv 0, \NCycle_5^{M,\TT}\equiv 0\}$,
	\itr   $\{\NCycle_6^{M,\TT}\equiv 0, \NCycle_4^{M,\TT}\equiv 0\}$.
	\eir
\end{theo}
The proof is given in  Section \ref{sec:ann}

\begin{rem}A natural question is: are $ \NCycle_6^{M,\TT}\equiv 0$ and $\NCycle_7^{M,\TT}\equiv 0$ equivalent\,? We tested this with a computer for $\kappa=5$ (by the computation of some Gröbner basis), where the answer turns out to be negative. We will see in the sequel that $7$ is the ``critical'' length of the systems associated to the range $L=2$. In section \ref{sec:Ext} we will give the critical length associated with a general range $L$.
\end{rem}

In the sequel, $W_1\cdots W_k$ will stand for the word obtained by the concatenation of the words $W_1,\cdots,W_{k-1}$ and $W_k$.
\begin{rem} [Linearity principle]\label{rem:LP}
  From Theorem \ref{theo:t0}, if a Markov law with Markov kernel $M>0$ is invariant by $\TT$ on the line, then the $M$-Gibbs measure is invariant in $\ZnZ$ for any $n\geq 3$. This is something which can be guessed and proved as follows. Take three words: $p$, $w$, and $s$, the ``prefix'', the ``pattern'', and the ``suffix''. Consider the word $W_n= p w^n s$. If the $M$-Markov law is invariant by $\TT$ on the line, then $\NLineq_{|p|+n|w|+|s|}^{\rho,M,\TT}(W_n)=0$. But it is easy to see that $\NLineq_{|p|+n|w|+|s|}^{\rho,M,\TT}(W_n)=(n-1) \NCycle_{|w|}^{M,\TT}(w) + O(1)$, so that one infers that $\NCycle^{M,\TT}_{k}\equiv 0$ for every 
  $k\geq 3$. \par
  In fact, this remark is also valid for any range $L$, and even the converse holds (see Theorem \ref{eq:fqqds}). 
\end{rem}

\subsubsection{Invariant Product measures}

\begin{defi}
A process $(X_k,k \in I)$ indexed by a finite or countable set $I$ is said to have the \textit{product distribution} $\bp^{I}$ for a distribution $\bp$ on $\Ek$ if the random variables $X_k$'s are i.i.d. and have common distribution $\bp$. \end{defi}
Since product measures are special Markov laws, we can use what has been said so far to characterize invariant product measure by $\TT$ by replacing $M_{i,j}$ by $\rho_j$ in the previous considerations (and rewrite, for example Theorem \ref{theo:t0} restricted to this special case). But, the ``7'' appearing everywhere in this theorem is no more relevant for product measure... the crucial length here is ``3''! To see this,  observe that when $M_{i,j}=\rho_{j}$, the quantity $Z^{M,\TT}_{a,b,c,d}$ does not depend on $(a,d)$, so that we may set
\ben\label{eq:newZ}
\wZ^{\rho,\TT}_{b,c}:=\wZ^{M,\TT}_{a,b,c,d}=\sum_{(u,v)\in \Ek^2} \T{(u,v)}{(b,c)} \frac{\rho_{u}\rho_{v}}{\rho_b \rho_c}-\To{b,c}.\een
$\ME_7^{M,\TT}, \Rep_7^{M,\TT} $ and $\NCycle_n^{M,T}(a\cro{0,n-1})$ respectively ``simplify to''
\ben
\nonumber\ME_3^{\rho,\TT}(a_0,a_1,a_2)&:=&\wZ^{\rho,\TT}_{a_0,a_1}+\wZ^{\rho,\TT}_{a_1,a_2}-\wZ^{\rho,\TT}_{a_0,a_2},\\
\nonumber\Rep_3^{\rho,\TT}(a_0,a_1,a_2;a_1')&:=&\wZ^{\rho,\TT}_{a_0,a_1}+\wZ^{\rho,\TT}_{a_1,a_2}-\wZ^{\rho,\TT}_{a_0,a_1'}-\wZ^{\rho,\TT}_{a_1',a_2},\\
\NCycle_n^{\rho,\TT}(a\cro{0,n-1})&:=&\label{lasduhf} \sum_{j=0}^{n-1} \wZ^{\rho,\TT}_{a_j,a_{j+1 \mod n}} \textrm{ for }  n \geq 2.
\een
We have the following analogue of Theorem \ref{theo:t0}, which provides some finite certificate/criteria for the algebraic invariance of product measures.
\begin{theo} \label{theo:t3b}
If $\Ek$ is finite, $L=2$, and if $\rho\in {\cal M}(\Ek)$ with support $\Ek$, then the following statements are equivalent:
\bir 
\itr $\rho^\Z$ is invariant by $\TT$ on the line.
\itr $\Rep_3^{\rho,\TT}(a,b,0;0) = 0$ for all $a,b \in \Ek$.
\itr $\Rep_3^{\rho,\TT} \equiv 0$.
\itr $\ME_3^{\rho,\TT}(a,b,0)=0$ for all $a,b \in \Ek$.
\itr $\ME_3^{\rho,\TT}\equiv0$.
\itr $\NCycle_n^{\rho,\TT}\equiv0$ for all $n\geq 2$.
\itr $\NCycle_3^{\rho,\TT}\equiv0$.
\itr $\NCycle_3^{\rho,\TT}(a,b,0)= 0$ for all $a,b \in \Ek$.
\itr There exist a function ${\sf W}:\Ek \to \R$ such that $\wZ^{\rho,\TT}_{a,b} = \W{b}- \W{a}$ for all $a,b\in \Ek$.
\eir
\end{theo}
In fact Theorem \ref{theo:t3b} is not a corollary of Theorem \ref{theo:t0}, but its proof is almost the same. 

\begin{rem} [Comparison with detailed balance condition] Consider a probability distribution $\rho$ on $\Ek$ with full support. A natural/folklore sufficient condition for this measure to be invariant by $\TT$ on the line is the fact that it solves the following system:
  \ben\label{eq:fqdq}\rho_b\rho_c\T{(b,c)}{(u,v)}= \rho_u\rho_v\T{(u,v)}{(b,c)}\textrm{ for any}, b,c,u,v \in \Ek.\een
 Summing this over $(u,v)$, one sees that this condition implies $\wZ^{\rho,\TT}\equiv 0$.
Theorem \ref{theo:t3b} applies to these situations since when $\wZ^{\rho,\TT}\equiv 0$, $(ii)$ to $(iii)$ are clearly satisfied. \par
  The crucial point here is that $\wZ^{\rho,\TT}\equiv 0$ is \underbar{just} a sufficient condition, not a necessary one (as we will see by providing examples in Section \ref{appl}):Theorem \ref{theo:t3b} gives the complete necessary and sufficient conditions.
\end{rem}

\begin{rem} For the sake of simplicity,  in Section \ref{sec:intro} we restrict  ourselves to criteria/properties for invariant of product measures with full support. Nevertheless, contrary to the Markov case, the case of product measures with a smaller support can be also considered without any problem [see Section \ref{sec:pefos}].
\end{rem}

\paragraph{``Range 2'' on a more general class of graphs.}

Most of the previous discussions on $\AI$ Markov law rely on the geometry of $\Z$, but it turns out that for $\AI$ product measures, some of the previous properties still hold when one defines a PS on a more general graph -- in the case where it still relies on a JRM with range 2.

Formally, consider a continuous time Markov process $X=(X_v,v \in V)$ defined on a lattice like $G =\Z^d$ or $(\ZnZ)^d$. 
Assume that the pair of states $(\eta(x),\eta(y))$  of two vertices $x$ and $y$, jumps to the new pair of states $(a,b)$ with rates $p(x,y)\T{(\eta(x),\eta(y))}{(a,b)}$ for $p(\cdot,\cdot)$, a translation invariant non negative function (that is $p(u,v) = p(0,v-u)$). By $p$, the rates also depend on the positions.  We suppose that there exists $\mathcal{N}\in\N$ such that $p(x,y)=0$ if $\|x-y\|_1\geq \mathcal{N}$ for all $x,y\in G$.

In this case the equilibrium equations for $x(A) \in \Ek^A$ for $A\subset G$ finite is
	\be
		&&\Lineq^{\rho,\TT,p}(x(A)) \\
		&:=& \sum_{w \in E^{D(A)}\atop w(A)=x(A)}\sum_{(i,j)\in  D(A)^2\atop i\in A \textrm{ or }j \in A} p(i,j)\left(\sum_{(u,v)\in \Ek^2}\frac{\rho_u\rho_v}{\rho_{w_i}\rho_{w_j}}\T{(u,v)}{(w_i,w_j)}-\To{(w_i,w_j)}\right)\prod_{i\in D(A)}\rho_{w_i}\\
		&=&  \sum_{w \in E^{D(A)}\atop w(A)=x(A)}\sum_{(i,j)\in \in D(A)^2\atop i\in A \textrm{ or }j \in A} \wZ^{\rho,\TT}_{w_i,w_j}p(i,j)\prod_{i\in D(A)}\rho_{w_i}
	\ee
where $D(A)$ denotes as before the dependence set, which definition needs to be extended for this type of graphs to 
	$D(A)=\{v\in V:  \max\{ p_{u,v}+p_{v,u}, u \in A\} >0 \}  \}$.
	\begin{defi}
	We will say that $\rho^{G}$ is $\AI$ by $p\TT$ if $\rho^{G}$ satisfies $\Lineq_{A}^{\rho,\TT,p}\equiv 0$ for all finite $A\subset G$ (again when $\Ek$ is finite and $G$ locally finite, invariance and algebraic invariance are equivalent notions.
      \end{defi}
      	\begin{theo}\label{theo:ppppp}
		Let $\#\Ek<+\infty$, $L=2$ and $\rho\in{\cal M}(\Ek)$ with full support. Depending on $p$ we have the following equivalences
		\begin{enumerate}
			\item If $p$ is symmetric. $\rho^{G}$ is invariant by $p\TT$ iff $\NCycle_2^{\rho,\TT} \equiv 0$.
			\item If p is asymmetric. $\rho^{G}$ is invariant by $p\TT$ iff the product measure $\rho^\Z$ is invariant by $\TT$ on the line (see the characterizations in Theorem \ref{theo:t3b}).
		\end{enumerate}
	\end{theo}
	Hence, if $p$ is asymmetric, the geometry does not matter since  $\wZ^{\rho,\TT}$ only depends on  the states (given by $\eta$), and not on the positions.

\subsection{A glimpse in 2D and beyond }
\label{sec:G2D}

We consider in this part PS indexed by $\Z^d$, and then whose configuration space is $\Ek^{\Z^d}$.
We suppose that the JRM instead of being defined (as done in \eref{eq:JRM}) by ``the jump rate of size $L$-subwords'' is defined by
\[\TT=\l( \T{w}{w'})_{u,v \in \Ek^{\HC[L,d]}}\r)\]
where
\[\HC[L,d]=\cro{0,L-1}^d\] is the hypercube with range $L$ in $\Z^d$ (see discussion in second point of Remark \ref{rem:generalshape} for other shapes).
For example for $d=2$, $\HC[2,2]$ is the square ${\sf \Sq}$ formed by the cells $(0,0) , (1,0),$ $(1,1)$ and $(0,1)$. An example of JRM is the following $\TT$ with all entries equal to 0, except
\ben\label{eq:egrda}
\tt11100001=1, \tt00011110=1,
\een
meaning that the $2\times 2$ square sub-configurations jumps with rate 1, if they are equal to {\small $\square1110$} or {$\square0001$}, in which case, the colours of the 4 vertices are flipped.\medskip

Formally, replace $J$ defined in \eref{eq:J}  by
\ben\label{eq:Jd}
J^{(d)}= \l\{(i,w,w'), i \in \Z^d, w \in \Ek^{\HC[L,d]},w'\in \Ek^{\HC[L,d]}\r\}=  \Z^d\times (\Ek^{\HC[L,d]})^2,
\een
and $m$ by $m^{(d)}$, which again is the family of endofunctions $m_{i,w,w'}^{(d)}$ on the set of configurations defined for any $(i,w,w')\in J^{(d)}$, generalizing naturally the $m_{i,w,w'}$'s defined in \eref{eq:miwwp}. The corresponding generator is
\ben
\l(G^{(d)}f\r)(\eta)=\sum_{(i,w,w')\in J^{(d)}} \T{w}{w'} \l[f(m_{i,w,w'}^{(d)}(\eta))-f(\eta)\r],
\een
acting on continuous functions $f$ sufficiently smooth (see discussion below \eref{eq:gen}). The dynamics of this PS is as follows:  starting from a (random or not) configuration $\eta^0=(\eta_{z}^0, z\in \Z^d)$, each sub-configuration  $(\eta_z^0, z\in h)= u$ indexed by a hypercube $h$ equal to $\HC[L,d]$ up to a translation, is replaced by the sub-configuration with same shape $v$ with rate $\T{u}{v}$. When $\Ek<+\infty$, this defines a Markov process (see discussion below \eref{eq:gen}).\par

Again a measure $\mu \in {\cal M}\l(\E_k^{\Z^d}\r)$ is said to be $\AI$ by $\TT$ in $\Z^d$ if its finite dimensional distributions are preserved by $\TT$. 
It is then possible to state the analogue of $\Lineq^{\Z}$ in these settings: let $C$ be a finite subset of $\Z^d$. Set 
\beq\label{eq:lineq2D}
\Lineq^{\Z^d}(x(C),\nu)=\dis\sum_{w,z \in E_{\kappa}^{D(C)}} \l(\nu_{D(C)}(w) \T wz- \nu_{D(C)}(z) \T zw\r)1_{z(C)=x(C)}
\eq
where $x(C)=(x_c,c \in C)$ is any element of $\Ek^C$, and where $D(C)$ is the dependence set of $C$: for any subset $F$ of $\Z^d$, the dependence set of $F$ is
\[D(F)=F-\HC[L,d].\]
Again, for any $w,z \in \Ek^{D(C)}$, the global transition rate from $w$ to $z$ is
\beq \label{eq:dec-T}
\T{w}z= \sum_{c \in \Z^2: (c+h) \cap C \neq \varnothing} \T{w(c+h))}{z(c+h))} \1_{w(x)=z(x), \forall x \in D(C)\setminus (c+h)},
\eq
where $h=\HC[L,d]$. 
Finally, the normalized version  $\Lineq^{\rho,\TT}$ is defined , for any finite domain $C$ by 
\ben\label{eq:nl}
\NLineq^{\rho,\TT}(x(C)):=\frac{\Lineq^{\rho,\TT}(x(C))}{ \prod_{c \in C} \rho_{x(c)}} \textrm { for any } x(C)\in \Ek^{C}.\een

The first theorem we want to state gives a necessary and sufficient condition for a product measure $\rho^{\Z^d}$ to be invariant by some PS with JRM $\TT$. Again, when $\Ek<+\infty$ it provides a criterion involving a system composed by  {\bf a  finite number of equations}. After that, we will explain how to obtain an equivalent system with a much smaller number of equations.

Let $C$ be a finite subset of $\Z^d$ and $D(C)$ its dependence set. The dependence set by definition is a union of hypercubes $h$ with sides $L$: depending on $C$, some of them may be included completely in $C$, some contains some points in $C$ and some points outside. The balance $\NLineq^{\rho,\TT}(x(C))$ can be decomposed as a sum on these hypercubes. Indeed, using the decomposition of $\TT$ along simple jump \eref{eq:nl}, one gets 
\ben\label{eq:nl2}
\NLineq^{\rho,\TT}(x(C)):=\sum_{h \subset C} {\bf Z}_{x(h)}+\sum_{h: h\cap C \neq \varnothing \atop{h\not\subset C}} {\bf Z}_{x(h\cap C)}^{h\cap C,h}
\een
depending on whether $h$ is totally included in $C$ or not. Here, the geometry of $\Z^d$ appears: when $h$ is not included in $C$, $h\cap C$ can be (depending on $C$) any subset of $h$, and we then need to mark this dependence with the pair  $(h\cap C,h)$  as an exponent of ${\bf Z}$. A simple analysis on the summation variables and the simplification of the quotient of weights of unchanged colours, give:
\ben\label{eq:zrgze}
{\bf Z}_{x(h)}= \l(\sum_{y\in \Ek^{h}} \frac{\prod_{c \in y} \rho_c}{\prod_{c \in x(h)} \rho_c} \T{y}{x(h)\,}\r)- \To{x(h)\,}
\een
and more generally, for $h$ such that $ h\cap C \neq \varnothing,h\not\subset C$, 
\ben\label{eq:Zh}
{\bf Z}_{x(h\cap C)}^{h\cap C,h}=\sum_{w(h)\in \Ek^h} {\bZ_{w(h)}}\1_{w(h\cap C)=x(h\cap C)} \prod_{j \in h\setminus C} \rho_{w_j}.
\een
We said, ``more generally'' because when $h\subset C$,
\ben\label{eq:Zh2}{\bf Z}_{x(h\cap C)}^{h\cap C,h} =  {\bf Z}_{x(h)}^{h,h}={\bf Z}_{x(h)}.
\een

When $|\Ek|<+\infty$, a product measure $\rho^{\Z^d}$ is \AI by $\TT$ if and only if all the maps $\NLineq^{\rho,\TT}\equiv 0$. For this, it is not needed that ${\bf Z}\equiv 0$ (but it is sufficient):
\begin{theo}\label{theo:qdq} When $|\Ek|<+\infty$, a product measure $\rho^{\Z^d}$ is invariant by $\TT$ if and only if the two following conditions hold:
  \bir
\itr $\sum_{h: 0\in h}{\bf Z}_{x(0)}^{0,h}=0$ where $0$ is the origin of $\Z^{(d)}$,
\itr For all subsets $C$ and $C'=C\cup\{c\}$ of $\HC[2L-1,d]$ (where $c$ is a single vertex), and any $x(C')\in \Ek^{C'}$,
\eir
\ben
\label{eq:DNL2}
\NLineq^{\rho,\TT}(x(C'))-\NLineq^{\rho,\TT}(x(C))\equiv 0.
\een
\end{theo}
as a trivial consequence we get the following condition, weaker than reversibility:
\begin{cor}\label{cor:dsdq} If $\Ek<+\infty$ and if ${\bf Z}\equiv 0 $ then the product measure $\rho^{\Z}$ is invariant by $\TT$ on $\Z^2$.
  \end{cor}
  \begin{proof} The product measure $\rho^{\Z^d}$ is invariant by $\TT$ if and only if \underbar{for at least one} sequence $(C_i,i\geq 0)$ of finite subsets of $\Z^{d}$, such that:\\
    {\tiny \bs} $C_{i+1}=C_i \cup \{c_{i+1}\}$ (a simple vertex),\\
    {\tiny \bs} $(C_i, i\geq 0)$ eventually contains an arbitrarily large hypercube,\\
  the property  $\NLineq^{\rho,\TT}(x(C_0))\equiv 0$, and for all $i\geq 0$, $\NLineq^{\rho,\TT}(x(C_{i+1}))-\NLineq^{\rho,\TT}(x(C_{i}))=0$ for any $x\in \Ek^{C_{i+1}}$ hold.\par
Due to \eref{eq:zrgze},  \eref{eq:Zh} and \eref{eq:Zh2}, if $C'=C\cup\{c\}$ for a vertex $c$ not in $C$, the difference $\NLineq^{\rho,\TT}(x(C'))-\NLineq^{\rho,\TT}(x(C))$ can be written as a sum of the contributions of the  hypercubes $h$ such that $(h\cap C) \neq (h\cap C')$. A simple inspection of the balance in the corresponding sums  as expressed in \eref{eq:nl2}, gives
\ben\label{eq:CCp}
\NLineq^{\rho,\TT}(x(C'))-\NLineq^{\rho,\TT}(x(C))=\sum_{h:h\cap C' \neq h \cap C} {\bf Z}_{x(h\cap C')}^{h\cap C',h} -  {\bf Z}_{x(h\cap C)}^{h\cap C,h}.\een
The theorem states something stronger than the fact that this property holds for all $C'=C_{i+1}, C=C_i$: it suffices that this property holds for those included in $\HC[2L-1,d]$. It remains to say that this last condition comes from \eref{eq:CCp}: the difference between the two $\NLineq$ concerns only the hypercubes $h$ that intersect the new vertex $c$, and then the union of these hypercubes is included in $\HC[2L-1,d]$. A given union of hypercubes appearing in such a difference can be realised by taking two sets $C_{i+1},C_i$ included in $\HC[2L-1,d]$. 
    \end{proof}

    \begin{rem}
\label{rem:generalshape}
$(i)$ It is possible to reduce the number of necessary and sufficient conditions in Theorem \ref{theo:qdq} by designing a particular growing sequence $(C_i)$ in such a way that the family $(h,C_i\cap h, C_{i+1}\cap h) $ (up to translation) involved in the right hand side of \eref{eq:CCp} for some $i$, take only a very small number of values: in $\Z^2$ for $2\times 2$ squares, we can manage to get only 2 (kind of) differences, starting from $C_0=\{(0,0),(0,1),(1,0)\}$. This is exemplified in Theorem \ref{theo:2D} and in its proof. \\
$(ii)$ What has been said so far concerns JRM indexed by hypercubes. If the PS of interests is given using some JRM indexed by some other ``shape $F$'', it is still possible to represent such a PS using a JRM indexed by hypercube (by taking a hypercube $h$ large enough to contain $F$, and by letting the colours in $h\setminus F$ unchanged). However, in $\Z^d$ the number of equations grows rapidly if one uses this kind of expedient. The best thing to do, is to adapt what has been said above to this special shape.
       \end{rem}  
    
\subsection{JRM indexed by $2\times 2$ squares in $2D$}

Following Remark \ref{rem:generalshape}, we design a set of necessary and sufficient conditions for invariance of a product measure $\rho^{\Z}$ ``less abundant'' than those given in Theorem \ref{theo:qdq}. We examine this in the 2D case, for a PS with JRM indexed by $2\times 2$ squares, denoted further $\Sq$ (as the one given in \eref{eq:egrda}).

Consider the three following sets:
\[
\Gamma_0=\{(0,0),(0,1),(1,0)\},~~ 
\Gamma_1=\Gamma_0 \cup \{(2,0)\}, ~~
\Gamma_2=\Gamma_1 \cup \{(1,1)\}.
\]
\begin{theo}\label{theo:2D} Let $\kappa<+\infty$. Consider $\rho$ a probability distribution with full support on $\Ek$ and  $\TT=\begin{bmatrix} \T{u}{v}\end{bmatrix}_{u,v \in \Ek^{\Sq}}$ a JRM indexed by $\Sq$. 
The measure $\rho^{\Z^2}$ is invariant by $\TT$ on $\Z^2$ iff the two following conditions hold simultaneously:
\bir
\itr $\NLineq^{\rho,\TT}\equiv 0$ on $\Ek^{\Gamma_0}$,
\itr for any $x\in \Ek^{\Gamma_2}$,
\ben\label{eq:yjil}
\NLineq^{\rho,\TT}(x)-\NLineq^{\rho,\TT}(x(\Gamma_1))=0.\een
\eir
\end{theo}
As a simple corollary: if $\NLineq^{\rho,\TT}\equiv 0$ on $\Gamma_2$ for a $\rho$ with full support then $\rho^{\Z}$ is invariant by $\TT$ on $\mathbb{Z}$.
\begin{proof} We give a picture based proof, using some representation of computations by pictures.\par
  We insist on the fact that $\rho^{\Z}$ is invariant by $\TT$ iff all the $\NLineq^{\rho,\TT}(x(C))=0$ for any sub-configuration $x(C)\in\Ek^{C}$, for any subset $C$ of $\Z^2$. As noticed in Theorem \ref{theo:qdq}, we just need to prove that for any $s\geq 0$, any square $C=\cro{0,s}^2$ is included in a finite domain $C'$ for which $\NLineq^{\rho,\TT}(x(C'))=0$ for all $x(C')\in \Ek^{C'}$.
  We will construct a well designed sequence $(C_i)$ satisfying the hypothesis of Theorem \ref{theo:qdq} and containing eventually $\cro{0,m}^2$.\par

Recall formula \eref{eq:nl2}, which expresses $\NLineq(x(C))$ as a sum of some ``Z'' indexed by the hypercube included in the dependence domain $D(C)$. 
In view of Figure \ref{fig:2D1}  \begin{figure}[h]
    \centerline{\includegraphics{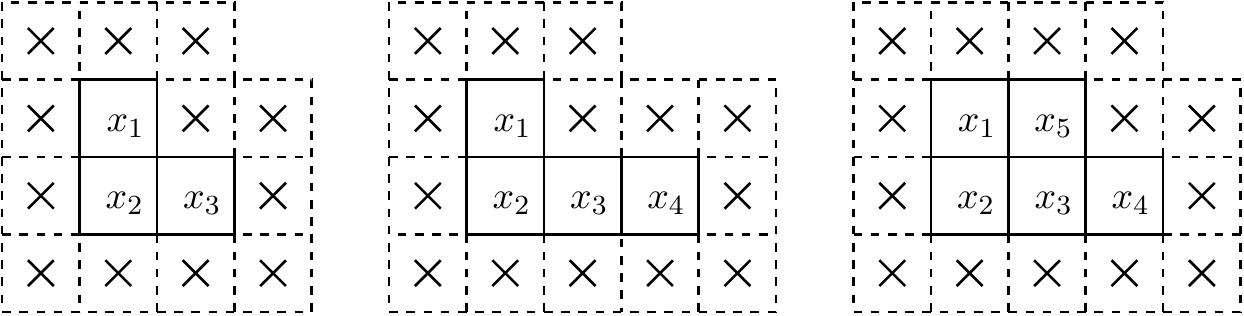}}
    \captionn{\label{fig:2D1} Shapes $\Gamma_0, \Gamma_1$ and $\Gamma_2$ appearing in Theorem \ref{theo:2D}.}
  \end{figure}
  the first hypothesis of Theorem \ref{theo:2D} says that the sums of these $Z$ over the eight $2\times2$ squares contained in the first picture  of Fig. \ref{fig:2D1} is 0. Let us express this by
    \[\Lineq^{\rho,\TT}(x(\Gamma_0))=Z_{\tiny \compact\ssquare{{\sf x}}{x_1}{{\sf x}}{{\sf x}}}+Z_{\tiny \compact\ssquare{{\sf x}}{x_2}{x_1}{{\sf x}}}+Z_{\tiny \compact\ssquare{{\sf x}}{{\sf x}}{x_2}{{\sf x}}}+Z_{\tiny \compact\ssquare{x_1}{{\sf x}}{{\sf x}}{{\sf x}}}+Z_{\tiny \compact\ssquare{x_2}{x_3}{{\sf x}}{x_1}}+Z_{\tiny \compact\ssquare{{\sf x}}{{\sf x}}{x_3}{x_2}}+Z_{\tiny \compact\ssquare{x_3}{{\sf x}}{{\sf x}}{{\sf x}}}+Z_{\tiny \compact\ssquare{{\sf x}}{{\sf x}}{{\sf x}}{x_3}}.\]
In $Z_{\tiny \compact\square{y_1}{y_2}{y_3}{y_4}}$,  the variables $x_1$, $x_2$, $x_3$ refers to some fixed specified values  and the ``${\sf x}$'' refers to free variables on which a sum is taken (as in the definition of ${\bf Z}_{x(h\cap C)}^{h\cap C,h}$, the ``variables in $h\setminus C$" are free variables on which a sum is taken).
    \end{proof}
    Further $\Lineq^{\rho,\TT}(x(\Gamma_1))$ and $\Lineq^{\rho,\TT}(x(\Gamma_2))$ are respectively sums of 10 and 11 such $Z$: each of these $Z$ must be seen at this stage as indexed by a $2\times 2$ square included in the second and third picture in Fig. \ref{fig:2D1} where $\Gamma_1$ or $\Gamma_2$ are drawn. Many of these $Z$ are common between these structures.
    It appears then that
    \be
    \Lineq^{\rho,\TT}(x(\Gamma_2))-\Lineq^{\rho,\TT}(x(\Gamma_1))&=&
\l(Z_{\tiny \compact\ssquare{x_2}{x_3}{{\sf x}}{x1}}-Z_{\tiny \compact\ssquare{x_2}{x_3}{x_5}{x_1}}\r)
+\l(Z_{\tiny \compact\ssquare{x_3}{x_4}{{\sf x}}{ {\sf x}}}-Z_{\tiny \compact\ssquare{x_3}{x_4}{{\sf x}}{x_5}}\r)\\
&+&
\l(Z_{\tiny \compact\ssquare{x_1}{{\sf x}}{{\sf x}}{{\sf x}}}-Z_{\tiny \compact\ssquare{x_1}{x_5}{{\sf x}}{{\sf x}}}\r)-  Z_{\tiny \compact\ssquare{x_5}{{\sf x}}{\sf x}{\sf x}}
.\ee
The terms have been assembled to make clear what changes the ``appearance'' of $x_5$ in  $x(\Gamma_2)$ compared to $x(\Gamma_1)$. Graphically, we use the  shortcut given in Figure \eref{fig:2D2}. 
\begin{figure}[h]
    \centerline{\includegraphics{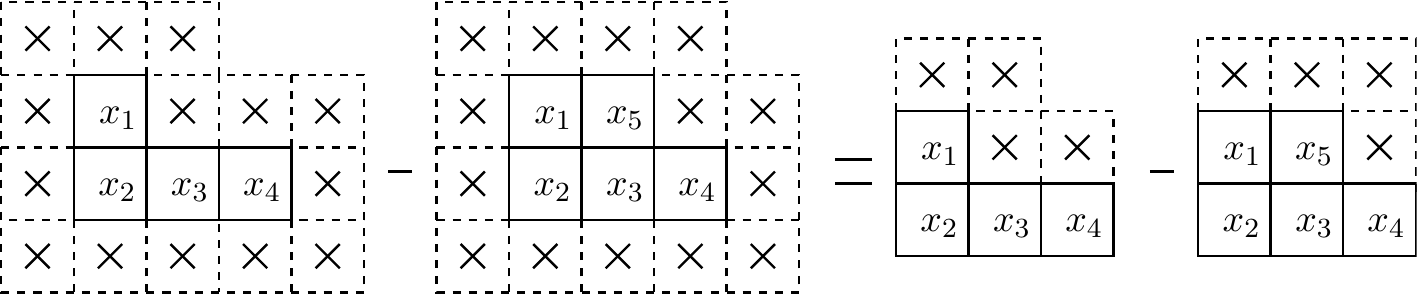}}
    \captionn{\label{fig:2D2} Expression of the difference between $\Lineq^{\rho,\TT}(x(\Gamma_2))$ and $\Lineq^{\rho,\TT}(x(\Gamma_1))$. }
  \end{figure}
  This Picture has to be understood as when one expressed the difference $\Lineq^{\rho,\TT}(x(\Gamma_2))$ and $\Lineq^{\rho,\TT}(x(\Gamma_1))$ by summing on the $Z$ indexed by the squares included in $\Gamma_2$ and those included in $\Gamma_1$, one gets the same results as if we do the same computation in the small figures in the right hand side in Fig. \ref{fig:2D2}. \par
  Consider some $n\geq 5$ (to avoid border effects due to the size of $\Gamma_2$), and consider the triangle
  \[\Delta_n=\{(i,j), 0\leq i \leq n, 0\leq j \leq n, 0\leq i+j \leq n\}.\] We will show that under the hypothesis of the theorem, for any $x(\Delta_n)\in \Ek^{\Delta_n}$, $\Lineq^{\rho,\TT}(x(\Delta_n))=0$. For this, we will need the four following steps:\\
$(a)$ if $\Lineq^{\rho,\TT}(x(\Gamma_0))=0$ for any $x(\Gamma_0)\in \Ek^{\Gamma_0}$, then $\Lineq^{\rho,\TT}(x(G))=0$ if $G$ is the $2\times 1$ or $1\times 2$ domino, or if $G$ is a single vertex ($1\times 1$). Indeed, these structures are included in $\Gamma_0$, and, for any $G\subset \Gamma_0$,  $\Lineq^{\rho,\TT}(x(G))=0$ can be obtained by summing $\Lineq^{\rho,\TT}(x(\Gamma_0))$ on the variables which are in $\Gamma_0\setminus G$.\\
$(b)$ From $(a)$, we deduce that if $L_n$ is the $n\times 1$ line, then $\Lineq^{\rho,\TT}(x(L_n))=0$ for any $x(L_n)\in\Ek^{L_n}$. The graphical proof of this property is drawn on Fig. \ref{fig:balance2}. A single argument is needed: the set of $2\times 2$ square contributions that do not vanish is the same in the right and left hand side.\\   
\begin{figure}[htbp]
\centerline{\includegraphics{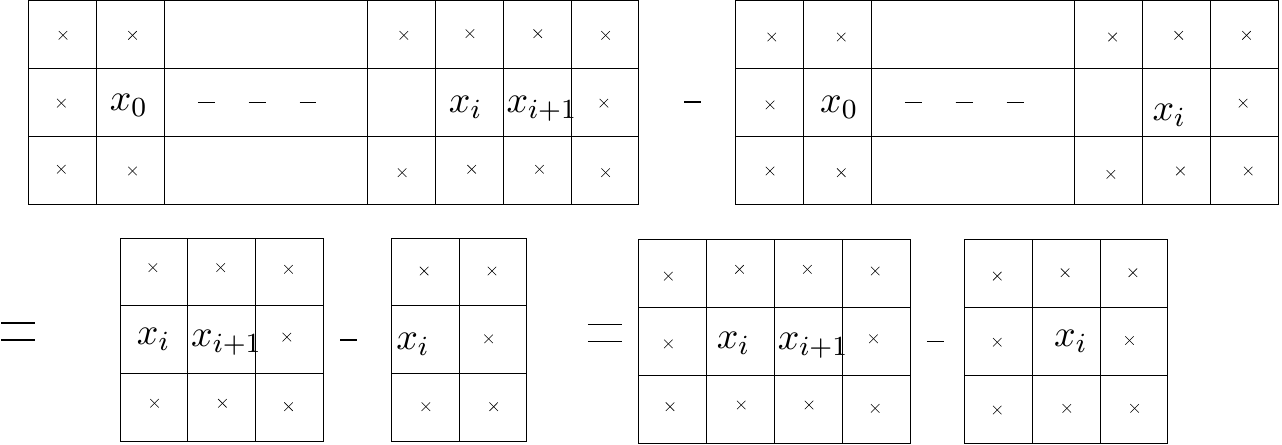}}
\caption{\label{fig:balance2} Representation of the geometry of the computation of $\NLineq^{\rho,\TT}(x(L_{i+1}))-\NLineq^{\rho,\TT}(\eta(L_{i}))$: each sum has to be taken on the set of $2\times 2$ squares included in the drawn rectangles. All squares appearing in both  pictures simplify and then the geometry of the summation reduces to that of the second line. In the third line, some squares are added, but since they correspond to the same contributions, this is allowed.}
\end{figure}  
$(c)$ We now, extend the construction of this row $L_n$ by adding a single vertex  $y$ just above the right-most element, getting a new shape $L_n'$ as represented in the top-left picture in Fig. \ref{fig:balance3}. The graphical proof provided in Fig. \ref{fig:balance3} allows to prove that $\Lineq^{\rho,\TT}(x(L_n'))=0$ using the nullity of $\Lineq$ on $\Ek^{L_n}$, $\Ek^{\Gamma_0}$ and on dominoes. \\
\begin{figure}[h]
\centerline{\includegraphics{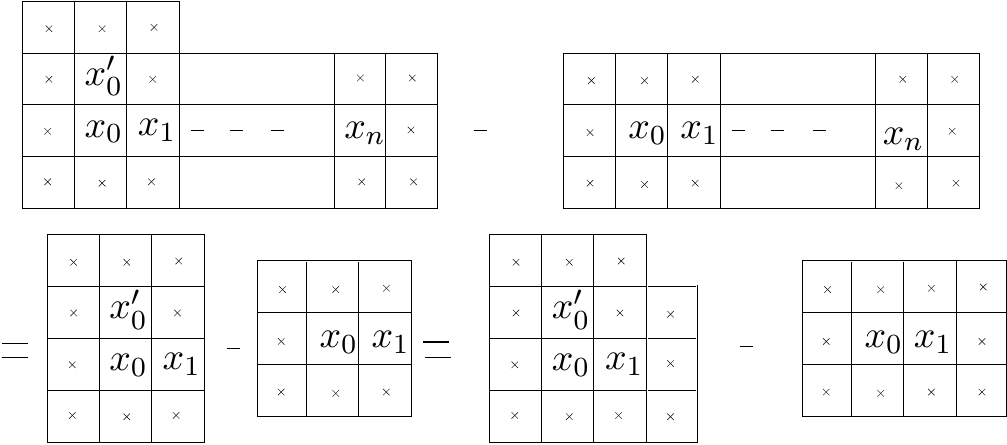}}
\caption{\label{fig:balance3} The equation resulting of the addition of a vertex above the leftest corner of the upest row.}
\end{figure}
$(d)$ The argument given in $(c)$ is independent from the fact that $L_n$ was the first row. Since the difference $\Lineq^{\rho,\TT}(x(L_n'))-\Lineq^{\rho,\TT}(x(L_n))$ does not involve the square below row at level 1 (say), if we ``complete'' both $L_n$ and $L_n'$ by the same fixed row at level 0, the difference $\Lineq^{\rho,\TT}(x(L_n'))-\Lineq^{\rho,\TT}(x(L_n))$ would be unchanged. Hence, if two structures $S$ and $S'$ are equal up to a given row at level $h$, and differs only because $S'$ possesses an additional point just above the leftest position of this row, then we still have $\Lineq^{\rho,\TT}(x(S'))-\Lineq^{\rho,\TT}(x(S))=0$. \par
Adding a single vertex above the left-most point of the top-most row is a construction which does not allows to pass from $L_n$ to $\Delta_n$. We still need an elementary growing trick to allow to put some new vertices at the right of the top-most vertex in $L_n'$ to complete the second row (in fact, we will construct a new row with one vertex less than $L_n$, leading iteratively to $\Delta_n$): a slight generalization of Figure \ref{fig:2D2} that do the job, and the graphical computation is represented in Fig. \ref{fig:balance4}:
\begin{figure}[h]
\centerline{\includegraphics{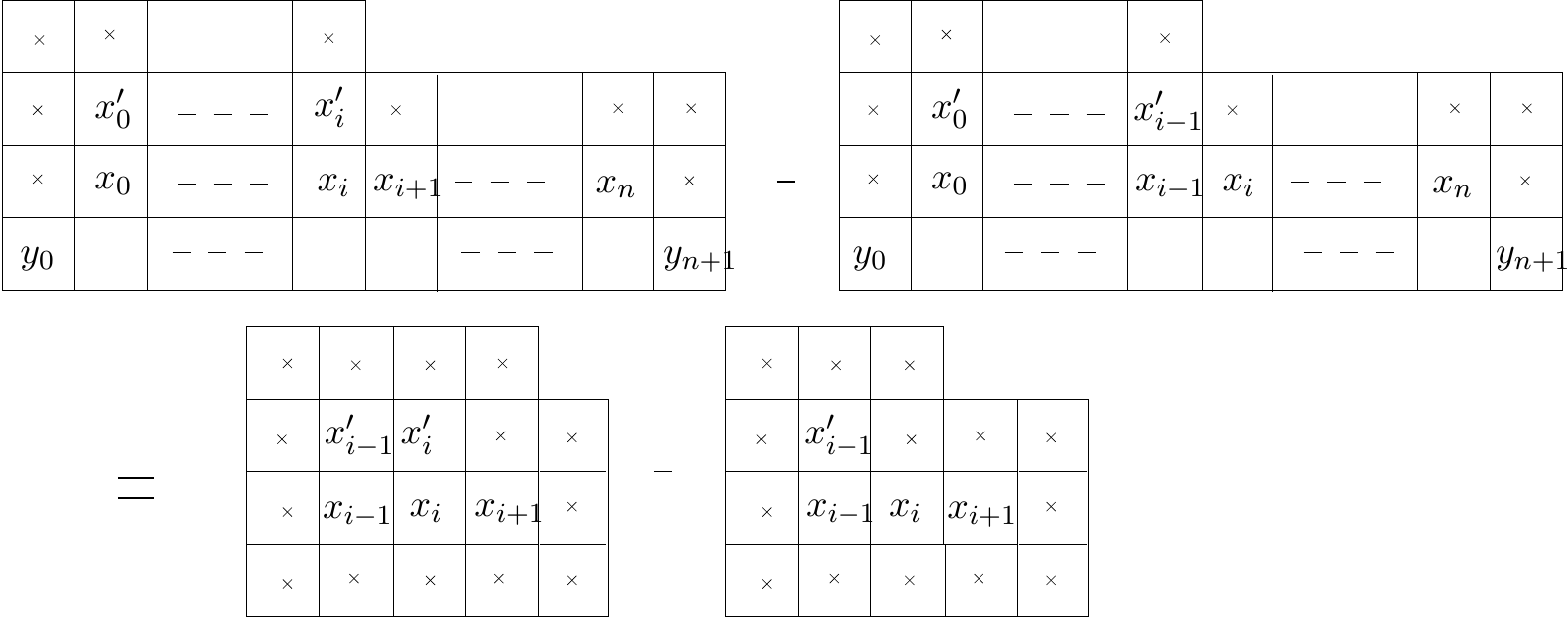}}
\caption{\label{fig:balance4} For this formula, observe that in the left hand side, the result is unchanged if, instead of taking a first specified row $y\cro{0,n}$ one takes unspecified values ${\sf x}$, since the squares involving any values of the first row vanishes. The right hand side is $0$ because of the second hypothesis of the theorem}
\end{figure}

\subsection{How to explicitly find invariant Markov law or invariant product measures on the line?}

In real applications, often $\TT$ is given, and the need is to find a Markov kernel $M$  so that the $M$-Markov law is $\AI$ by $\TT$. Let us call
\[S_j(\TT)=\{ \textrm{Markov kernel }M ~: M>0, \NCycle^{M,\TT}_j\equiv 0\}.\]
When $|\Ek|<+\infty$, by Theorem \ref{theo:t0}, to find such $M$ amounts to finding $S_7(\TT)$  (which can be empty).
The algebraic system  $\NCycle^{M,\TT}_7\equiv 0$ is huge even when $\kappa$ is small, and then quite difficult to solve: many equations of degree 6 in $M$ (by Theorem \ref{theo:t2}) and linear in $\TT$.   
From Theorem \ref{theo:t0}, we know that if the $M$-Markov law is invariant by $\TT$, 
$S_7(\TT)\subset S_3(\TT)$.
It turns out that computing $S_3(\TT)$ can be done (see Theorem \ref{theo:cand3}), and then, in practice, these solutions can be tested in $\NCycle_7^{M,\TT}$ afterwards. \par

To find $\TT$ when $M$ is given so that $\TT$ preserves the $M$-Markov law is a linear algebra problem since e.g. $\Cycle^{M,\TT}_7(a, b , c, d, 0, 0, 0)=0$ is a linear system in $\TT$; the set of solutions is a convex set. Notice that if for a fixed $M$ some tools of linear algebra are used to find the $\TT$'s solution of e.g.  $\Cycle^{M,\TT}_7\equiv 0$, then an additional work of identification of non negative solutions is needed.

\subsubsection{Computation of $S_3(\TT)$}
\label{sec:trkyh}
Assume $\TT$ is given, and let us determine $S_3(\TT)$. 
Setting
  \ben\label{eq:gfghs}
  \nu_{a,b,c}:= \frac{M_{a,b}M_{b,c}M_{c,a}}{\Tr(M^3)}, \textrm{ for every } a,b,c\in \Ek,
  \een
  the equation $\Cycle_3^{M,\TT}(a,b,c)= 0$ is equivalent to
\beq\label{eq:ukfkd}
\left\{ 
\begin{array}{r l}
	&  \sum_{(u,v)} \left(\nu_{c,u,v}\T{(u,v)}{(a,b)}+\nu_{a,u,v}\T{(u,v)}{(b,c)}+\nu_{b,u,v}\T{(u,v)}{(c,a)}\right)\\
	&=\nu_{a,b,c} \l(\To{a,b}+\To{b,c}+\To{c,a}\r)
\end{array}\right.
\eq
This is a linear system in $\nu$, therefore it can be solved by means of linear algebra. If no positive solution $\nu$ exists, then $S_3(\TT)=\varnothing$.
Assume that a positive solution $\nu$ exists.
Define for any $a,b\in\Ek$ the row matrices $L_{a,b}$, the square matrices $N_a$, and the vector $R$:
\ben\label{eq:qff}
N_{a}&=& \begin{bmatrix} \dis \frac{\nu_{a,x,y}\,\nu_{a,a,a}^{1/3}}{\nu_{a,x,a}} \end{bmatrix}_{x,y\in \Ek}, \\
L_{a,b} &=&   \begin{bmatrix} \nu_{a,b,x}, x \in \Ek \end{bmatrix},\\
R&=& {}^{t}\begin{bmatrix} 1, x \in \Ek \end{bmatrix}.
\een
For each $a$, take the pair of left and right eigenvector $(\ell=\ell_{a},r=r_{a})$ with positive entries of $N_{a}$ corresponding to the main eigenvalue (notion defined below Def. \ref{defi:Gibbs}), normalized so that $\|\ell_a\|_1=\ell_aR=1$, and $r_{a}\ell_{a}=1$.
Recall the considerations just above \eref{eq:ukfkd}. 
\begin{theo}\label{theo:cand3}
Let $\#\Ek< +\infty$, $L=2$ and $\nu$ be a given probability measure on $\Ek^3$, invariant under rotation, and solving \eref{eq:ukfkd}. 
If there exists a positive recurrent  $M$-Markov law such that \eref{eq:gfghs} holds
then all the matrices $(N_{x},x\in \Ek)$ possess the same main eigenvalue $\lambda$. \\
In case of existence of  a positive recurrent Markov kernel $M$ solving  \eref{eq:gfghs}, $M$ is unique and is characterized together with its invariant distribution $\rho$ by
\ben\label{eq:grsfq}
\rho_{a}M_{a,b}=  \frac{L_{a,b}r_a}{\lambda^3}.
\een
\end{theo}

\begin{rem}We don't know if the fact that the matrices $(N_{x},x\in \Ek)$ possess the same main eigenvalue $\lambda$ implies that there exists a Markov kernel $M$ such that \eref{eq:gfghs} holds.
\end{rem}

\paragraph{An algorithm to  compute $S_3(\TT)$:}~\\
-- search the set of probability measures $\nu$ solving \eref{eq:ukfkd},\\
-- for each element of this set (which is moreover invariant by rotation), check if the corresponding $N_x$'s possess the same main eigenvalues $\lambda$,\\
-- if yes, compute $M$ using \eref{eq:grsfq},\\
-- if this $M$ satisfies \eref{eq:gfghs}, then add it to the set $S_3(\TT)$.

\paragraph{Another point of view on the uniqueness of $M$:} The system of equations $\ME_7^{M,\TT}$,  $\NCycle_n^{M,\TT}$ are linear in the $\T{(a,b)}{(c,d)}$'s, and linear in the rational fractions of the family 
\beq\label{eq:ExtZeq}
\bF:=\l(F_{(a,u,v,d)\atop (b,c)} := \frac{M_{a,u}M_{u,v}M_{v,d}}{M_{a,b}M_{b,c}M_{c,d}}\r)_{a,b,c,d,u,v \in  \Ek}, 
\eq
since $\wZ^{M,\TT}_{a,b,c,d}$ has this property.
Finding $M$ satisfying $\ME_7^{M,\TT}\equiv 0$ for a given $\TT$ can be done in two steps: first, solve the system of linear equations $\ME_7^{M,\TT}\equiv 0$ with the vector $\bF$ as unknown variable, and then when ${\bF}$ is found, search if there exists a Markov kernel $M$ which satisfies \eref{eq:ExtZeq}. The second step is algebraically the most difficult since \eref{eq:ExtZeq} is a cubic system in $M$ for a given $\bF$, nevertheless, we have:
\begin{theo}\label{theo:prop1} Given $\bF$, there exists at most one positive recurrent Markov kernel $M$  solving \eref{eq:ExtZeq}.
\end{theo}
The proof is provided in Section \ref{sec:ann}.

\subsubsection{Finding the set of invariant product measures}
Let $\TT$ be given and $|\Ek|<+\infty$. Now we explore some necessary and/or sufficient conditions for the existence of product measures invariant  by $\TT$ on the line. Define the symmetric version of $\TT$ by
\be
\S{(a,b)}{(c,d)} = \T{(a,b)}{(c,d)}+\T{(b,a)}{(d,c)}.
\ee 
\begin{theo}\label{theo:t5}  Let $\#\Ek<+\infty$, $L=2$ and $\rho\in {\cal M}(\Ek)$ with full support.\\
  $(i)$ If the product measure $\rho^{\Z}$ is invariant by $\TT$, then $\rho^\Z$ is also invariant by $\SS$. \\
 $(ii)$ The product measure $\rho^\Z$ is invariant by $\SS$ (or any symmetric JRM $\SS$) on the line iff $\wZ^{\rho,\SS}\equiv 0$.
\end{theo}
\begin{proof}$(i)$ The set of JRM that preserve a given invariant distribution is a cone. Now, the product measure $\rho^{\Z}$ is preserved by ``space'' reversibility: if $\rho^{\Z}$ is invariant by the JRM $\TT$ then it is also invariant by $\TTp$ defined by $\Tp{(a,b)}{(c,d)}=\T{(b,a)}{(d,c)}$.\\
  $(ii)$ When the product measure $\rho^\Z$ is invariant by $\SS$, then $\NCycle_4^{\rho,\SS}\equiv 0$ by  Theorem \ref{theo:t3b} which implies   $\NCycle_4^{\rho,\SS}(a,b,a,b)=2\l(Z^{\rho,\SS}_{a,b}+Z^{\rho,\SS}_{b,a}\r)=4Z^{\rho,\SS}_{a,b}=0$ for any $a,b\in \Ek$, and then $\Z^{\rho,\SS}\equiv 0$. Conversely, if $Z^{\rho,\TT}\equiv 0$, by all the criteria of Theorem \ref{theo:t3b}, the product measure $\rho^\Z$ is invariant by $\SS$ on the line.
\end{proof}
Hence, {\bf to know if there exist some product measures invariant by some  given $\TT$}, one can proceed as follows:\\
(a) compute $\SS$,\\
(b) solve the equation $\Z^{\rho,\SS}\equiv 0$ with unknown $\SS$ (a pretreatment, can consist to replace in $Z^{\rho,\SS}$, each occurrence of $\rho_{x}\rho_y$ by $\rho_{x,y}$ in order to get a linear equation in the vector $(\rho_{u,v}, (u,v)\in \Ek^2))$. After that, it remains to check if indeed $\rho_{u,v}$ can be written under the the form $\rho_u \rho_v$ (notice that in this case $\rho_u=\sqrt{\rho_{u,u}^2}$).\\
(c) If $(b)$ provides no solution, then no product measure are invariant under $\TT$. If $(b)$ provides some solutions, they are candidate to be invariant by $\TT$, and it remains to check if whether  $\Cycle_3^{\rho,\TT}\equiv 0$ or not.

\color{black}
\subsection{Models in the segment with boundary conditions}
\label{sec:SB}
In general when one defines a PS on $\Z$ or the segment $\cro{1,n}$, where a special behaviour at the boundary of the domain is forced.
\begin{defi}
	A probability measure $\gamma_n\in {\cal M}\l(\Ek^{\cro{1,n}}\r)$ is said to be \AI by $\TT^{\cro{1,n}}$ on the segment $\cro{1,n}$ if it solves the following system:
	\[
	\Sys{\cro{1,n},\gamma_n,\TT^{\cro{1,n}}}~:=\left\{ \Lineq^{\cro{1,n},B}(x)=0, \textrm{ for any }x\in \Ek^{\cro{1,n}},
	\right.
	\]
	where for an extra $B$ in the notation denote the presence of a boundary
	\[
	\Lineq^{\cro{1,n},B}(x)=\sum_{w \in E_{\kappa}^{\cro{1,n}}} \gamma_{n}(w) \T wx^{\cro{1,n}}- \gamma_{n}(x) \T xw^{\cro{1,n}}, 
	\]
 Recalling that $\T{w}{z}$ is the induced jump rate on an interval defined in \eref{eq:Tind}, we define $\TT^{\cro{1,n}}$ as the sum of this induced jump rate to which  we add some boundary effects at the left and at the right of the segment given by some jump rate matrices $\beta^{\ell}$ and $\beta^{r}$  with range $L-1$:
	\be
	\T wz^{\cro{1,n}}&=& \T wz 	\\
        &+&         \beta^{\ell}\l[w\cro{1,L-1},z\cro{1,L-1}\r] \,\1_{ w_j=z_j, \forall j \in \cro{L,n}}\\
	&+&         \beta^{r}\l[w\cro{n-(L-2),n},x\cro{n-(L-2),n}\r]\, \1_{ w_j=z_j, \forall j \in \cro{1,n-(L-1)}}.
	\ee
\end{defi}	

We go on focusing on $\AI$ Markov law here. Take again $M$ a positive Markov kernel, and $\rho$ its invariant distribution.
Define
\[\NLineq^{M,B}(x\cro{1,n}):=\Lineq^{M,B}(x\cro{1,n})/ (\rho_{x_1}\prod_{i=1}^{n-1} M_{x_i,x_{i+1}})\] where $\rho$ is the unique element of ${\cal M}(\Ek)$ such that $\rho M=\rho$.
For $n\geq 3$, a simple computation shows (see if needed the forthcoming Section \ref{sec:PF_theo:t0}), that $\NLineq_n^{M,B}(x\cro{1,n})=$
\be
&& \sum_{j=2}^{n-2}\wZ_{x\cro{j-1,j+2}}\\
&& -\beta^{\Out,\ell}_{x_1}-\To{x_1,x_2}+  \sum_{u_1,u_2} \frac{\rho_{u_1}M_{u_1,u_2}M_{u_2,x_3}}{\rho_{x_1}M_{x_1,x_2}M_{x_2,x_3}}
\l(\T{(u_1,u_2)}{(x_1,x_2)}+\beta^{\ell}_{u_1,x_1}\1_{u_2=x_2}\r)\\
&&-\beta^{\Out,r}_{x_n}-\To{x_{n-1},x_{n}} + \sum_{u_{n-1},u_n} \frac{M_{x_{n-2},u_{n-1}}M_{u_{n-1},u_n}}{M_{x_{n-2},x_{n-1}}M_{x_{n-1},x_n}}(\T{(u_{n-1},u_n)}{(x_{n-1},x_n)}+1_{u_{n-1}=x_{n-1}}\beta^{r}_{u_n,x_n}).
\ee
\begin{theo}\label{theo:segm} Let  $\Ek$ be finite, $L=2$ and $M$ be a positive Markov kernel on $\Ek$.
If for some $n_0\ge 7$ the $(\rho,M)$-Markov law is invariant by $(\beta^r,\beta^\ell,\TT)$ on $\cro{1,n}$ for $n=n_0$ and for $n=n_0+1$, then:
\begin{itemize}
\item[\bs]  the $(\rho, M)$-Markov law  is invariant by $\TT$ on the line,
\item[\bs]  the $(\rho,M)$-Markov law  is invariant by $(\beta^r,\beta^\ell,\TT)$ on $\cro{1,n}$ for any $n\geq n_0$.
\end{itemize} 
\end{theo}

\begin{proof} To prove the first point: By Theorem \ref{theo:t0}, it suffices to prove that  $\ME_7^{M,\TT}\equiv 0$.	So assume that $\NLineq_{n_0+1}^{M,B}\equiv0$ and $\NLineq_{n_0}^{M,B}\equiv 0$, and observe that for any $x\in \Ek^{n_0+1}$,
\beq\label{eq:rgethe}
\NLineq_{n_0+1}^{M,B}(x)-\NLineq_{n_0}^{M,B}(x^{\{4\}})
= \ME_7^{M,\TT}(x(\cro{1,7})),\eq
(the boundary terms cancel out) 
Now, to prove the second point, it suffices to observe that \eref{eq:rgethe} still holds if one replaces $n_0$ by a larger integer, so that one can infer from the nullity of $\NLineq_{n_0+1}^{M,B}$ and $\NLineq_{n_0}^{M,B}$ the nullity of $\ME_7^{M,\TT}$, and after that of all $\NLineq_{n}^{M,B}$ for $n\geq n_0$.
\end{proof}
\begin{theo}\label{theo:lkajsfdgv} Let  $\Ek$ be finite, $L=2$ and $M$ be a positive Markov kernel on $\Ek$.
	If  the $(\rho,M)$-Markov law  is invariant by $\TT$ on the line, then there exist two vectors $\beta^r$ and $\beta^\ell$ such that  the $(\rho,M)$-Markov law is invariant by $(\beta^r,\beta^\ell,\TT)$ on $\cro{1,N}$, for any $N\geq 1$. 
\end{theo}
\begin{proof}
Suppose that a $(\rho,M)$-Markov law is invariant by $\TT$ on the line. The key point in the proof if that, if $(X_{0},\cdots,X_{n+1})$ has the $(\rho,M)$-Markov law under its invariant distribution, then $(X_1,\cdots,X_n)$ has also the $(\rho,M)$-Markov law under its invariant distribution. Hence, it is possible to build explicitly $\beta^{\ell}$ and $\beta^r$ in such a way they emulate the exterior effects of the segment $\cro{1,N}$.
It suffices then to take simply
\[
\bpar{ccl}
\beta^{\ell}_{z,a}&=&\l(\sum_{u,v} \rho_uM_{u,z}\T{(u,z)}{(v,a)}\r)/\rho_z\\
\beta^{r}_{x_{n},a}&=&\sum_{v,b} \T{(x_n,v)}{(a,b)} M_{x_n,b}.
\epar
\]
 \end{proof}

 \begin{rem} What is done in this section is a bit related to the matrix ansatz
 used by Derrida \& al. \cite{DEHP} in order to find and describe the invariant distribution $\mu_n$ of the TASEP on a segment $\cro{1,n}$, in the sense that it relies on a telescopic scheme.
   \end{rem}

\section{Extension to larger range, memory, dimension, etc.}\label{sec:Ext}
\subsection{ Extension of Theorem \ref{theo:t0} to larger range and memory}

{The case $L>2$ can be treated as the case $L=2$ has been treated, with some adjustments. Also the case of \AI Markov law with memory $m>1$ can be managed. We discuss both extensions simultaneously here. A first change concerns the ``7'' which played a special role in Theorem \ref{theo:t0} which will be replaced by
\beq
{\sf h}=4m+2L-1.
\eq
As usual, a Markov chain with Markov kernel $M$ and memory $m\geq 0$, is a process $(X_k,k\geq k_0)$ (for some $k_0$) whose distribution is characterized by 
\[`P\l( X_j =x_j ~|~(X_{j-i}=x_{j-i},i\geq 1)\r)=M_{x\cro{j-m,j}},\]
for $j-m\geq k_0$, and an initial distribution $\mu\in{\cal M}(\Ek^m)$, the distribution $\mu$ of $(X_{k_0},\cdots,X_{k_0+m-1})$.
The Markov kernel $M$ is a matrix with size $\kappa^m \times \kappa$ with non negative entries, such that, for any $x\in \Ek^{m}$, $\sum_{y \in \Ek}M_{xy}=1$. We call such a Markov kernel, a Markov kernel with memory $m$. \par 
We let $\Lineq_n^{\rho,M,\TT}(x\cro{1,n})$ be the equation $\Lineq^\Z(x\cro{1,n},\nu)$ where $\nu$ is the $M$-Markov law with memory $m$ and JRM $\TT$ (we may use the same notation as before, since in the case $(L,m)=(2,1)$ we recover the same definition as before). The equation $\Lineq_n^{\rho,M,\TT}(x\cro{1,n})=0$ rewrites:
\be
0&=& \sum_{ w \in \Ek^{\cro{-(L-1),n+L-1}}} \sum_{j=-L+2}^{n} 1_{w_k=x_k, k \in \cro{1,n}\setminus{\cro{j,j+L-1}}} \sum_{u\in \Ek^L} \mu_{G(w,u,j)} \T{u}{x\cro{j,j+L-1}}\\
&&-\sum_{w \in \Ek^{\cro{-(L-1),n+L-1}}}   \sum_{j=-L+2}^n\To{w\cro{j,j+L-1}} \mu(w) \1_{w\cro{1,n}=x\cro{1,n}},
\ee
for $G(w,u,j)$ being the word $w$ in which $w\cro{j,j+L-1}$ has been replaced by $u$:
\be
G[w,u,j]&=&w\cro{-(L-1),j-1}\,u\, w\cro{j+L,n+L-1},\\
\mu(w\cro{1,N})&=&\rho_{w\cro{1,m}}\prod_{j=1}^{N-m} M_{w\cro{j,j+m}}
\ee
and where $\rho$ is the invariant distribution of the Markov kernel $M$.
\begin{rem} Mimicking what has been done in \eref{eq:erhrehr2}, and explained below  we may write a variant of this formula, by summing on $w$ with index set enlarged, by taking $w\in\Ek^{\cro{-q,n+L-1}}$ with $q= L-1+m$ (and summing on these words), keeping unchanged the sum on $j$, in such a way that in the representation of $\mu(w)$ there are no intersection of indices between those involved in $\rho$ and in $\TT$.
  \end{rem}
We also extend the definition of $\NLineq^{\rho,M,\TT}$ to the present case,:
\beq\label{eq:nli}
\NLineq^{\rho,M,\TT}_n(x\cro{1,n})=: \frac{\Lineq^{\rho,M,\TT}(x\cro{1,n})}{\prod_{j=1}^{n-m}M_{x\cro{j,m+j}}}.
\eq
The quantity which plays the role of $\wZ$ in these settings is:
\beq\label{eq:rgerh}
\wZ_{a\cro{1,m},b\cro{1,L},c\cro{1,m}}=\sum_{u\cro{1,L}\in \Ek^L} \T{u\cro{1,L}}{b\cro{1,L}} \prod_{j=1}^{m+L} \frac{M_{w'\cro{j,j+m}}}{M_{w\cro{j,j+m}}} - \To{b\cro{1,L}}, 
\eq 
where $\To{u\cro{1,L}}=\sum_{v\cro{1,L} \in \Ek^L} \T{u\cro{1,L}}{v\cro{1,L}}$ and
\[
\bpar{ccl}
w & = & a\cro{1,m}\,b\cro{1,L}\,c\cro{1,m},  \\
w'& = & a\cro{1,m}\,u\cro{1,L}\,c\cro{1,m}.
\epar\]
}

The quantity which will play the role of ``4'' as in \eref{eq:CyZ} is 
\[{\sf s}= 2m +L = ({\sf h}+1)/2.\]
We extend the definition of $\NCycle_n$ for $n\geq m+1$:
\[\NCycle^{M,\TT}_n(x\cro{1,n})=\sum_{w \in \Subn{x\cro{1,n}}{{\sf s}}{\ZnZ}} \wZ_w;\]
for any $x\in \Ek^{{\sf h}}$ and $y\in \Ek$, extend $\ME^{M,\TT}$ and $\Rep^{M,\TT}$ by:
\be
\ME^{M,\TT}_{{\sf h}}(x)&=&\sum_{w \in \Sub{x}{{\sf s}}} \wZ_w -\sum_{w \in \Sub{x^{\{{\sf s}\}}}{{\sf s}}} \wZ_w,\\
\Rep^{M,\TT}_{{\sf h}}(x;y)&=&\sum_{w \in \Sub{x}{{\sf s}}} \wZ_w -\sum_{w \in \Sub{x\cro{1,{\sf s}-1}\,y\,x\cro{{\sf s}+1,{\sf h}}}{{\sf s}}} \wZ_w.
\ee
In words, the second sum ranges on the subwords of size ${\sf s}$ of $w$ with the central letter removed in the case of $\ME^{M,\TT}$ and changed in the case of $\Rep^{M,\TT}$. For $(L,m)=(2,1)$ we recover the standard definition of $\ME_7^{M,\TT}$ and $\Rep_7^{M,\TT}$.
Here is the main result of this section:
\begin{theo} \label{theo:t0prime} Theorem \ref{theo:t0} holds, for $M$ a Markov kernel with memory $m$, and $\TT$ with range $L$ 
by replacing $(a,b,c,d,0,0,0)$ by $(a\cro{1,{\sf s}}\, 0^{{\sf s}-1})$, $n\geq m+L$ in $(vi)$, ${\sf W}$ by ${\sf W}^{m,L} :\Ek^{{\sf s}-1}\to \R$ and $\Rep_7^{M,\TT}, \ME_7^{M,\TT}$, $\NCycle_7^{M,\TT}$ by the variants defined above (with ${\sf h}$ instead of 7).
\end{theo}
\begin{rem}
  The constraint $n\geq m+L$ in $(vi)$ comes from the fact that, if $n<m+L$, the cyclic structure imposes the repetition of some letters in the product of $M$'s inside $\Cycle$. This fact is reflected in Theorem \ref{theo:t0} $(vi)$ and Theorem \ref{theo:t3b} $(vi)$, where a product measure is seen as a Markov law with memory $m=0$.
\end{rem}
\begin{rem}
  \bia
  \ita A given PS may be represented in several different ways using different JRM: the PS with jump rate $\T{(0)}{(1)}=\T{(1)}{(0)}=1$ with range 1, can be represented using as JRM  $\T{(a,b)}{(1-a,b)}=\T{(a,b)}{(a,1-b)}=1/2$ for any $a\neq b, a,b \in \{0,1\}$ instead, on the line and on any $\Z/n\Z$ for $n\geq 2$. In Theorem \ref{theo:t0prime}, we do not assume that the smallest possible range has been used, but there is a price to pay to use a representation with JRM with a non minimal range since the equations provided by Theorem \ref{theo:t0prime} are more numerous, and have a larger degree in $M$. 
\ita The previous point may lead to think that it could be a good idea to represent any PS with a JRM with range 2, which is always possible, by changing the alphabet: if $\TT$ has range $L> 2$, then by taking the map which sends the set of configurations $\E_\kappa^\Z$ onto $A^{\mathbb{Z}}$ where the alphabet $A= \E_\kappa^{L-1}$ by sending $\eta\in \Ek^Z$ on to $(\eta'_j, j \in \Z)$ where
\[\eta'_j=  [\eta_{j+x}, 1\leq x \leq L-1],\]
that is by rewriting $\eta$ as a sequence of overlapping subwords on size $L-1$, then one can express on this new space the JRM thanks to a jump rate of range 2. However, since our theorem allows to characterize the invariant Markov law with some fixed memory $m$, \textbf{with full support} they are not suitable to characterize invariant Markov law for $\eta'$ (since consecutive states $\eta_j'$ and $\eta_{j+1}'$ must be consistent, that is the suffix of $\eta'_j$ must coincide with the prefix of $\eta'_{j+1}$).
\eia
\end{rem}
The proof of Theorem \ref{theo:t0prime} is a bit more complex that that of Theorem \ref{theo:t0}  (Section \ref{sec:ann}).

\paragraph{Equivalence between  $\NCycle_n^{M,\TT}\equiv 0$ for small $n$'s and invariant of a $M$-Markov law on the line.}
Theorem \ref{theo:t0} which states that  $\Cycle_n^{M,\TT}\equiv 0$ for any $n$ is equivalent to $\ME_7^{M,\TT}\equiv 0$ is valid only when $L=2$ and $m=1$, but the fact that $\NCycle_n^{M,\TT}\equiv 0$ for every ``small $n$'' is equivalent to the invariance of the $M$-Markov law on the line, is true in all generality, and can be proved using arguments that are interesting from their own sake:
\begin{theo}\label{eq:fqqds} Let $\kappa<+\infty$, $L<+\infty$. For a Markov kernel $M$ with memory $m$ and positive entries to be invariant by $\TT$, it is necessary and sufficient that  $\NCycle_n^{M,\TT}\equiv 0$ for any $n\leq \kappa^m$.
\end{theo} 
The proof is given in Section \ref{sec:ann}.

\subsection{The case $\Ek=\N$ (that is $\kappa = \infty$)}

Here we will consider Markov kernels with positive entries (the case with possibly zero entries is discussed in Section \ref{sec:relax}). The main problem in the case $\kappa = +\infty$ is that the sums defining $\Lineq$ are now infinite series and therefore some conditions need to be satisfied in order to rearrange terms as done in the proofs, for example, to write $\wZ$. The first problems come from the infinitesimal generator (see \eref{eq:gen}) which may fail to have an interesting domain, in other words, in general, it does not define a Markov process. But even if we jump directly to the $\AI$ considerations a second problem arising is that it is no more clear that $\Lineq_n^{\rho,M,\TT}\equiv 0$ and $ \NLineq_n^{\rho,M,\TT}\equiv 0$ are equivalent. The series appearing in both members of \eref{eq:erhrehr2} are composed with positive terms.
It is necessary and sufficient that each of them converges for $\Lineq$ to be well defined. 
If each of them converges, Fubini's theorem ensures that we can rearrange globally their terms as wished. Hence, we have under this condition 
	\beq\label{eq:dfvasbv}
	\l(\Lineq^{\rho,M,\TT}\equiv 0\r) \imp \l(\NLineq^{\rho,M,\TT}\equiv 0\r).
	\eq 
	The problem is that it is often the converse which is needed, since all criteria we gave rely on $\ME$, $\NLineq$, $\NCycle$.        When $\#\Ek<+\infty$, 
	\beq\label{eq:fqd}
	\l(\NLineq^{\rho,M,\TT}\equiv 0\r) \imp \l(\Lineq^{\rho,M,\TT}\equiv 0\r),
	\eq
        but, when $\kappa$ is infinite, when a pair $(M,T)$ solving $\NLineq^{\rho,M,\TT}\equiv 0$ is found, \eref{eq:fqd} must be checked.

        The following proposition gives a sufficient condition for the validity of both \eref{eq:fqd} and \eref{eq:dfvasbv}.

\label{sec:kappainfinite}
{
\begin{pro} Assume that $\kappa\in \N\cup\{+\infty\}$, and $M$ is a positive Markov kernel. 	If 
	\[
        \bpar{ccl}
	C_1&:=&\sup_{a,b,c,d \in \Ek}\sum_{u,v\in \Ek} \frac{M_{a,u}M_{u,v}M_{v,d}}{M_{a,b}M_{b,c}M_{c,d}}\T{(u,v)}{(b,c)} <\infty\\
	C_2&:=& \sup_{b,c \in \Ek}\To{(b,c)}<\infty,
	\epar
        \]
	then $\NLineq_n$ and $\Lineq_n$ as defined in \eref{eq:lineq} and \eref{eq:rsgfqe12} are well defined, and satisfy \eref{eq:tjht}, and then \eref{eq:dfvasbv} and \eref{eq:fqd} are satisfied.
\end{pro}
\begin{proof}
  Following the discussion above, we verify that under the hypothesis above, the series arising in each term of \eref{eq:erhrehr2} are absolutely convergent.
  For this notice that it suffices to replace the sign ``minus'' by "plus" in   \eref{eq:erhrehr2} and to bound  it by 
\[	\leq (n+1)(C_1+C_2)\rho_{x_{1}}\prod_{k=1}^{n-1}M_{x_k,x_{k+1}}.
	\]
Hence, if $C_1$ and $C_2$ are finite, the sums in $\Lineq^{\rho,M,\TT}$ are well defined and can be rearranged. \par
In the same way, the positive and negative contributions in \eref{eq:rsgfqe12} can be separated and each of them converge absolutely. The conclusion follows.	
\end{proof}

\begin{theo}\label{theo:tyuttr}
  Assume that  $\kappa\in \N\cup\{+\infty\}$. If the three following conditions holds:\\
{\bs} \eref{eq:dfvasbv} and \eref{eq:fqd} hold,\\
{\bs} $M$ has positive entries,\\
{\bs} $M$ is positive recurrent,\\
 then the conclusion of Theorem \ref{theo:t0} holds.
\end{theo}
\begin{proof} The positive recurrence ensures that the Markov kernel $M$ admit  an invariant distribution $\rho$ with full support, from what the proof of Theorem \ref{theo:t0} can be proved as in the finite case.
\end{proof}

\begin{rem}The additional assumption of positive recurrence is a natural condition for several reasons.
  The first one is, in the definition of $\Lineq^{\rho,M,\TT}$, the need of an initial distribution for the Markov chain. When several invariant distributions for $M$ exist (or if none exists), everything is more complex, as discussed in Section \ref{sec:relax}.\par
  One way to see the appearance of multiple \AI Markov laws is to consider two continuous time Markov processes  $X^t$ and $Y^t$ respectively on $A^\Z$, and  $B^\Z$ with $A\cap B=\varnothing$. With them, one may construct a continuous time Markov process $Z^t$ which coincides with $X^t$ and $Y^t$ if the starting configurations are in $A^\Z$,  and $B^\Z$, respectively by defining:\\
$\bullet$ for $u\in A^2$, $\TTT{Z}{u}{v}=\TTT{X}{u}{v}$ for $v\in A^2$, and 0 if $v\notin A^2$\\
$\bullet$ for $u\in B^2$, $\TTT{Z}{u}{v}=\TTT{Y}{u}{v}$ for $v\in B^2$, and 0 if $v\notin B^2$\\ 
$\bullet$ and if $u$ is not in $A^2\cup B^2$, choose any value for $\TTT{Z}{(u_1,u_2)}{(v_1,v_2)}$. In this case, the set of configurations $A^\Z$ and $B^\Z$ do not communicate; if both $X^t$ and $Y^t$ possess a $\AI$ Markov law, then $Z^t$ possess several invariant Markov laws, including those that are mixture of these.  Theorem \ref{theo:t0} and all its criteria do not allow to characterize this kind of invariant measures.  
\end{rem}

\subsection{Invariant product measures with a partial support in $\Ek$.}
\label{sec:pefos}
We discuss here an iff criterion to show the invariance distribution of a product measures $\nu^{\mathbb{Z}}$ with support $S$ strictly included in $\Ek$. The idea to get some criteria is just to discard the set  $\Ek\setminus S$ which should not be reachable from $S$ if an invariant distribution with support $S$ exists:

Consider $\TT$ a JRM on $\Ek$, and let $S$ be a strict (non empty) subset of $\Ek$ and $\nu$ a measure with support $S$. 
Assume that for any $u,v,a,b \in S$
  \beq\label{eq:LPCC}
  \l(\nu_{u}\nu_v>0, \T{(u,v)}{(a,b)}>0\r)\imp \nu_a \nu_b>0
  \eq
and interpret this condition as: if the word $w'$ is obtained from the word $w \in S^{\Z}$ by a jump with positive rate,  then $w'$ must be in the support of $\nu^{\Z}$. This implies that the restriction $\TTp$ of $\TT$ to $S$ defined by
  \[\Tp{(a,b)}{(c,d)}=\T{(a,b)}{(c,d)} \textrm{ for }a,b,c,d\in S\] has the following property: the PS on $\Ek^\Z$ (resp. $S^\Z$) with JRM $\TT$  (resp. $\TT'$) coincide if starting from a measure $\nu$ with support in $S$. The following theorem is a direct consequence of this fact:
  \begin{theo}\label{theo:partial_support}  Let $|\Ek|<+\infty$. A product measure  $\nu^\Z$
     with support  $S=\Sup(\nu)\subset \Ek$ is invariant by $\TT$ on the line, for $\TT$ a JRM on $\Ek$ iff \eref{eq:LPCC} holds as well as any of the equivalent conditions listed in Theorem \ref{theo:t3b} holds within $S$.
  \end{theo}

\subsection{Invariant for Markov laws  for kernel with some 0 entries}
\label{sec:relax}

In Theorem \ref{theo:partial_support} is discussed the invariance of a product measure which has a partial support in $\Ek$, and in fact, our criteria apply to this situation up to a simple restriction of the state space. \par
The same kind of conditions can be imagined for a $M$-Markov law satisfying $M_{i,j}>0$ for $i,j\in S$, and such that for any $i \in S$, $\sum_{j\in S} M_{i,j}=1$, meaning that the states in $S$ just communicate with other states in $S$. If 
\[\forall a,b \in S, \T{(a,b)}{(u,v)}>0 \imp u,v \in S,\]
then, the JRM $\TT$ can be restricted to $S$. Denoting by $\TT'$ this restriction, the criteria we have (Theorem \ref{theo:t0}) allows to decide if the $M$-Markov law is invariant by $\TT'$ on $S^{\Z}$. 
Under these conditions, everything is then somehow trivial, since $S^\Z$ is close under the action of the jumps with positive rate. \\
The general case is much more complicated!\\
Consider a general Markov kernel $M=(M_{i,j})_{i,j \in E_\kappa}$. Consider the directed graph $G=(\Ek,E)$ whose vertex set is the alphabet $\Ek$ and the edge set is 
$
E=\l\{(i,j): M_{i,j}>0\r\}.
$
Consider the strongly connected components $({\cal C}_j, j \in J)$ of this graph, where $J$ is a set of indices. Starting from any point $v\in E_\kappa$, the Markov chain $(X_n,n\geq 0)$ with kernel $M$ will eventually reach one of these strongly connected components ${\cal C}_j$ and will stay inside a.s., forever. The invariant distributions of $M$ naturally decomposes as a mixture of the invariant distributions $\rho^{(j)}$, where $\rho^{(k)}$ is the invariant distribution of $M$ on ${\cal C}_k$. \par
The strongly connected components do not communicate, then, one may partition the vertex sets $E_\kappa$ along these connected components. The Markov chain on each of this connected component is irreducible and can be treated separately: the fact that one of them is invariant by $\TT$ does not interfere with the fact that the ``other sub-Markov chains'' have the same property or not. \par

 The property of being irreducible does not mean that $E$ is the complete graph and some $M_{i,j}$'s can still be 0 in this case. It may also happen that $M$ is periodic, meaning that again, it may exist several invariant distributions with the Markov kernel $M$ (for example, equal up to a translation, alternating between even and odd states).  \par
Again, the range considered here is $L=2$, and some adjustments need to be made in the next considerations if $L>2$. 
Consider an irreducible Markov chain $(X_n,n\geq 0)$ with kernel $M$. Its invariant distribution has full support on $E_\kappa$. 
Let 
\[\Sup_n=\l\{x\cro{1,n} \in E_\kappa^n:  \rho_{x_1}\prod_{j=1}^{n-1} M_{x_j,x_{j+1}} >0\r\}.\]
be the support of the distribution of $n$ consecutive positions of this Markov chain.
A necessary condition for the $(\rho,M)$-Markov law to be invariant by $\TT$ on the line is the following \it local preservation condition {\bf (LPC)}: \rm 
\[
\textrm{if }(a,b,c,d)\in \Sup_4 \textrm{ and  if }\T{(b,c)}{(u,v)}>0 \textrm{ then }(a,u,v,d)\in \Sup_4.
\]
If $x\cro{1,n}$ belongs to $\Sup_n$ then all its subwords $x\cro{m,m+3}$ with 4 letters are in $\Sup_4$. Assume that $\TT$ possesses the LPC, then the $(\rho,M)$-Markov law is AI by $\TT$ if for any  $x\cro{1,n}\in\Sup_n^M$,  $\Lineq^{\rho,M,\TT}(x\cro{1,n})=0$. 
Under the LPC, we may still pass from $\Lineq^{\rho,M,\TT}(x\cro{1,n})$ to $\NLineq^{\rho,M,\TT}(x\cro{1,n})$ by dividing by $\prod_{i=1}^{n-1} M_{x_i,x_{i+1}}$ as far as $x\cro{1,n}\in \Sup_n^M$. Besides,
$\wZ_{a,b,c,d}^{ M,\TT}$ is still well defined for $(a,b,c,d)\in \Sup_4$.\par
Now, solving $\NLineq_n^{\rho,M,\TT}\equiv 0$ \textbf{ cannot at all be done  according to the same lines as before}, since one cannot compare simply $\NLineq_n^{\rho,M,\TT}(x[n])$ with $\NLineq_n^{\rho,M,\TT}(x[n]^{\{k\}})$ (with a suppressed letter) simply, since $x[n]\in \Sup_n \not\imp x[n]^{\{k\}}\in \Sup_{n-1}$.\par

It turns out that for a general JRM $\TT$, the support of an (algebraic) invariant distribution possesses its own combinatorial structure, which may be really complex.
\\
Indeed $\TT$ may have some combinatorial properties with a flavour reminiscent to group theory: it is possible to design some JRM $\TT$ which preserves several non communicating subsets of $\Ek^{\mathbb{Z}}$, for example the subset of words  $w=(w_i,i\geq \Z)$ satisfying $w_{i}+w_{i+1} \in  17\Z \cup \19\Z$ for all $i$'s (such a property holds for any JRM $T$ satisfying : for any $(x,y)$ such that $x+y  \in 17\Z \cup \19\Z$, $\T{(x,y)}{(x',y')}>0\imp  x'+y' \in 17\Z \cup \19\Z$. It is also possible to imagine and design invariant Markov laws with a much more complex support. In any case, the characterization of the set of pairs $(M,\TT)$ such that a $M$-Markov law is invariant is quite complex, and all the tools we used to prove Theorem \ref{theo:t0} fail. In few words, this happens because it is no longer possible to compare the balance for two close words:  for instance, the word obtained from the removal of a letter of a word in the support may not belong to it -- and the number of letters to remove in order to go back to the support is a parameter of the system, and of the initial word; it is in general not constant, and unbounded.
  
\section{Applications}
\label{appl}
\color{black}

\subsection{Explicit computation :  Gr\"obner basis}
\label{sec:Grobn}
In this subsection we generalize and revisit some well known models using our theorems. Before that, we would like to discuss a bit the ``explicit'' resolution of systems of algebraic equations. \par
First, the simplest systems of equations are linear systems: they are systems of polynomial equations of degree $1$ in some unknown variables $(x_1,...,x_n)$, with some coefficients in $\R$ or, possibly, with coefficients being some functions of some parameters $(y_1,\cdots,y_n)$. Such systems can be solved using linear algebra.
  If some parameters $(y_i)$ are present, then the study is in general much more complicated: typically, even the dimension of the set of solutions can vary when the parameters change.\par
  To solve these systems a computer algebra system can be used:  only simple operations as multiplications, additions are needed: if the coefficients are integers, or for examples, polynomials in the $y_i$'s with integer coefficients, the results obtained are exact.\par
  For polynomial system with only one unknown $x$, of the form ${\sf Sys}=\{P_i(x)=0,  1\leq i\leq k\}$, the first step is the computation of the $\gcd$ $G$ of these polynomials (using Euclidean algorithm): $x$ is solution to the system  ${\sf Sys}$ if and only if $G(x)=0$. Assuming that the $P_i$ are not all 0 (in which case the question is trivial, but what follows does not work) if $G$ is a constant, then ${\sf Sys}$ has no solution, and if $G$ is a polynomial, then the solutions of ${\sf Sys}$ is the set of roots of $G$, which exists in $\mathbb{C}$ by d'Alembert-Gauss theorem. Finding explicit solutions can be done by numerical approximations, and in some cases, explicit exact solutions can be found; in any case, the set of solutions of ${\sf Sys}$ is implicitly known by $G(x)=0$.\par

 Here the situation we face is more complex: Take for example $\ME_7^{M,\TT}\equiv 0$ in the case where $\Ek$ is finite. This system is linear in $\TT$ and involves quotient of cubic monomials in the $M_{i,j}'s$. We can transform this system into a polynomial systems in several variables as follows:
 A pair $(M,\TT)$ solves the system $S=\{\Cycle_7^{M,\TT}\equiv 0, M>0\}$ iff it solved the following  system of {\bf polynomial} equations 
\be
S':=\left\{ 
\begin{array}{r c l l}
        \Cycle_7^{M,\TT}(x\cro{1,7})&=&0&, \forall x\cro{1,7}\in \Ek^7,\\
	M_{a,b}y_{a,b}-1&=&0&, \forall (a,b)\in \Ek^2\\
	-1+\sum_{b\in\Ek}M_{a,b}&=&0& ,\forall a\in E_k
\end{array}\qquad,
\right. 
\ee
where the $y_{a,b}$ are additional variables which prevent the $M_{a,b}$'s to be 0. \par
Equivalence of systems means here that $(M,\TT)$ is solves  $S$ iff there exists $y$ such that $(M,\TT,y)$ solves $S'$, and $M>0$. Any $M$ such that $(M,\TT,y)$ solves $S'$ has non zero entries, but could have some negative ones, or even complex ones: it depends how/where the system is solved.\par 
We then need to solve polynomial systems in several variables. In this case again, we cannot expect a better situation than for polynomial systems of a single variable: in general no close formulas exist for solutions, but again, it is possible to know if solutions exist, and in this case, find some minimal representations of the solution set (if $\TT$ is given, the problem is almost the same).

A common way to solve this kind of problems amounts to computing a Gröbner basis: given a finite set of polynomials $S=\{P_i,i \in I\}$ where the $P_i$'s belong to $\R[x_1,...,x_n]$, a Gröbner basis of $S$ is a basis of the ideal generated by $S$ which have some additional properties. It depends on a good monomial order (preserved by multiplications, if $x^{(\alpha)}< x^{(\beta)}$ then $x^{(\alpha)}x^{(\gamma)}< x^{(\beta)}x^{(\gamma)}$ where $x^{(\alpha)}=\prod_{i} x_i^{\alpha_i}$ for $\alpha=(\alpha_1,\dots,\alpha_n)$). We cannot go too far in the description of the Gröbner basis properties, or to explain how they are computed: we refer the interested reader to Adams and Loustaunau \cite{AL} to get an overview and to Jean-Charles Faugère webpage \cite{JCF} for many resources on this topic, including fast algorithms.

In order to be understandable to the reader unaware of these methods, we will just stress on the following fact:\\
{\tb} Computation of Gröbner basis for polynomials with integer coefficients relies on simple elementary operations as Euclidean division of polynomials, sorting of polynomials according to their coefficients/and or degrees and then can be performed by a computer algebra system working on integers (and then it is decidable). \\
{\tb} When the basis $B$ has been computed, the basis is a finite sequence of polynomials, equivalent to the initial system $S$.\\
\indent{--} if the Gr\"obner basis is $G=[1]$ then there are no solution to the initial system (whatever is the order used),\\
\indent{--} if it is not $G=[1]$ then there are some solutions to the initial system in $\mathbb{C}$: some extra work could be needed to see if there are some solutions in $\R$, $\R^+$ or $[0,1]^n$ if these are some additional requirements,\\
\indent{--} since $B$ is a basis of the ideal generated by $S$, each polynomial $p$ in $B$ is a necessary condition on the solution set. Hence  if a Gr\"obner basis contains a polynomial, for example $(2x_1+x_7-9)(3x_7-8x_9^{17}+1)$ for a system $\{P_i,1\leq i \leq k\}$ in the variables $x_1,\cdots,x_{100}$: then they are some solutions to the system in $\mathbb{C}^{100}$, and each solution $(x_1,\cdots,x_{100})$  satisfies either $2x_1+x_7-9=0$ or $3x_7-8x_9^{17}+1=0$ (inclusive ``or'' of course).\\
{\tb} Computing a Gröbner basis is time and memory consuming, so computing a Gröbner basis is sometimes impossible in practice by hand, and even by computer. \par
{\tb} There are different notions of Gröbner basis as said above, since they rely on a (good) order on monomials; this order is also needed to define the Euclidean division in the set of polynomials in several variables. Each order leads to a specific representation of the ideal. 
For example, if $P_1= x^2+y^2-z^2-3,P_2=x^2+2y^2-4,P_3=y^2+3z^2-x-2$, the computation of the Gr\"obner basis relative to the graded reverse lexicographical order gives as a basis $G=[2z^2-x-1, 2y^2+x-1,x^2-x-3]$. If alternatively, the lexicographical order (plex) is chosen, the basis is $G=[4z^4-6z^2-1, y^2+z^2-1, -2z^2+x+1]$. Both results ensure the existence of solutions in $\mathbb{C}^3$. It is somehow trivial in this case that solution exists if we take as granted the equivalence to solve the initial system $\{P_1=0,P_2=0,P_3=0\}$ and (one of) the system(s) $G$. If we add the polynomial $P_4=xz-y^2+2$, then this time (any) Gr\"obner basis is $G=[1]$: there are no solutions. What happens here is different from the ``linear algebra settings'': the number of polynomials in the Gr\"obner basis depends on the order chosen, and often, the basis is huge, containing many more polynomials than the initial system.\medskip

Till here, our main theorems assert that ``a $M$-Markov law'' is invariant by a PS with JRM $\TT$ can be reduced to checking if a polynomial system in $(M,\TT)$ has some solutions. 
The paragraph above on the Gr\"obner basis is here to say that checking the existence of a $M$ that solves for example $\Cycle_5^{M,\TT}=0$ when $M$ is fixed, is possible at the  price of computing a Gr\"obner basis. If the Gröbner basis in $G=[1]$ there are no solution. If $G$ is not 1, it will be a list of polynomials in the variables $M$ and $\TT$ simpler than the initial problem (the order plex allows to obtain a kind of triangular system in which a well chosen order of the variables make apparent the conditions on $\TT$, for example). Again, a work still remains to be done to check that real solutions exist.\par
When a solution $(M,\TT)$ has been found by this mean, an independent proof of the invariance of the Markov chain with kernel $M$ by $\TT$ can be done by checking directly -- without using a Gr\"obner basis computation -- that $\Cycle_7^{M,\TT}\equiv 0$. 
\begin{rem} There are some good reasons to be confident on the Gr\"obner basis computations with computer algebra systems which rely on simple computations on integers and which are used by many users, for many reasons including cryptography motivations, but for the reader which prefers to stay away from this kind of automatic tools, we insist on the fact that these computations can be done by hands (and patience). 
\end{rem}

To follow in details the following examples, the reader can download in \cite{LF-JFM} a maple-file or a pdf file, where all the computations are done.

\subsubsection{Stochastic Ising models}
\label{sec:MTSMIsing}

The stochastic Ising model (given below Def. \ref{defi:JRM}) possesses a unique Markovian invariant measure on the line with  kernel $M$ characterized by
\beq\label{eq:rghfzq}
M_{0,1} = \frac{1}{1+e^{2\beta}}\quad\text{ and }\quad M_{1,1} = \frac{1}{1+e^{-2\beta}}=M_{0,0}.
\eq
When $\beta=0$, this is the Bernoulli(1/2) product measure (Liggett \cite[Introduction]{LTIPS}).

Let us see how to recover this with our approach. Since $M$ and $\TT$ are given. 
By  Theorem \ref{theo:t0prime} , since $m=1$, $\kappa=2$, and the range is $L=3$, it suffices to check that $\Cycle_9^{M,T}\equiv 0$.  Plug the values of $M$ and of $\TT$  (given in \eref{eq:rghfzq} and in \eref{eq:rsgsf}) in the corresponding $Z$ (which is found in \eref{eq:rgerh}). Here $Z$ owns 5 indices, and then 32 values $Z_{a,b,c,d,e}$ need to be computed: one finds that these 32 values are all zeroes! As a consequence $\Cycle_9^{M,T}\equiv 0$.
 \par

Assume now, that the existence of an invariant Markov law is unknown for this PS. Let us see how to recover this property.
Again, since the range is $L=3$, we need to find a $M$, for which $\Cycle_9^{M,T}(a,b,c,d,e,0,0,0,0)= 0$ for all $a,b,c,d,e \in E_2$ as specified by Theorem \ref{theo:t0prime}. First, we make rid of the ``exponential function'' in $\TT$ by a change of variables and a deterministic linear change of time (the computation of a Gröbner basis must be done in a polynomial ring). To do this, we set $x=e^{-\beta}$ and use $\bar{\TT}=x^2\TT$ instead of $\TT$, since this does not alter the set of invariant distributions. We obtain
 \[\bar{\TT}_{[a,b,c \,|\,a,1-b,c]} = x^2 \T{[a,b,c]}{[a,1-b,c]}=x^{2+(2b-1)(2a+2c-2)}.\]
 We also add the polynomials $M_{a,b}g_{a,b}-1$ in the basis computation (this prevents each $M_{a,b}$ to be 0) and for simplicity we imposed $M_{i,0}=1-M_{i,1}$ for all $i\in E_2$. Then with a computer algebra system, compute the Gröbner basis: this is immediate, and two solutions appear, one of them being negative. The unique positive solution (after inverting the change of variable) is given in \eref{eq:rghfzq}.

\subsubsection{The voter model and some variants}
\label{sec:MTSM2}

Consider the JRM $\TT$ of the voter model: $\TT$ is not identically 0 and besides, the voter model possesses $0^n$ and $1^n$ as absorbing states on $\ZnZ$ (and this can be generalized if more ``opinions'' are represented). The following corollary is an immediate consequence of  Theorem \ref{theo:fs}. 
\begin{cor}\label{cor:contvoter}
	If $\mu$ is an invariant distribution for the voter model (or in the generalized model with $\kappa$ opinions) on the line different from a mixture on the Dirac measures on an opinion $\delta_{j^{\mathbb{Z}}}$ for $j\in\Ek$, then $\mu$ is not a Markov law with memory $m$, for any $m$.    
\end{cor}
Let us now consider some variants of the voter model and the existence or not of invariant Markov law for them.  Consider a JRM $\TT$  in which  all the entries of $\TT$ are taken equal to 0 except for $\T{[a,b,c]}{[a,b',c]}$ for $b'\neq b$, $b'\in\{a,c\}$: In words, this is the voter model in which the rate at which an individual change his minds is a function of its own opinion and of those of its neighbours.
A Gröbner basis for the sequence of polynomials $\{M_{a,b}g_{a,b}-1,a,b \in \{0,1\}, \Cycle_9^{M,\TT}(a,b,c,d,e,0,0,0,0)\equiv 0\}$ gives
\be
\T{[1,0,1]}{[1,1,1]},\,\T{[0
,1,0]}{[0,0,0]},\,
g_{{1,0}}g_{{1,1}}-g_{{1,0}}-g_{{1,1}},
g_{{0,0}}g_{{0,1}}-g_{{0,0}}-g_{{0,1}},\\
M_{{1,1}}g_{{1,1}}-1,\,g_{{1,0}}M_{{1,1}}-g_{{1,0}}+1,M_{{0,1}}
g_{{0,1}}-1,\,M_{{0,1}}g_{{0,0}}-g_{{0,0}}+1,\\
-M_{{1,1}}\T{[0,1,1]}{[0,0,1]}g_{{0,0}}-M_{{1,1}}\T{[1,1,0]}{[1,0,0]}g_{{0,0}}+\T{[0,0,1]}{[0,1,1]}+\T{[1,0,0]}{[1,1,0]}
\ee
Hence if a $M$-Markov law with positive Markov kernel $M$ is invariant then
\ben\label{eq:TTTTT}
\T{[0,1,0]}{[0,0,0]}=\T{[1,0,1]}{[1,1,1]}=0
\een(which then excludes the original voter model). 
From here if we replace $g_{a,b}=1/M_{a,b}$ in the basis and look at the remaining equations, apart those corresponding to $M_{i,0}+M_{i,1}=1$ and to the non nullity of the $M_{a,b}$'s, hen it only remains:
\ben
M_{{0,0}}(\T{[0,0,1]}{[0,1,1]}+\T{[1,0,0]}{[1,1,0]})-M_{{1,1}
}(\T{[0,1,1]}{[0,0,1]}+\T{[1,1,0]}{[1,0,0]})\een
whose nullity is the only constraint (together with \eref{eq:TTTTT}) for the Markov law with positive kernel $M$ to be invariant by  $\TT$ on the line (since $M_{0,0}/M_{1,1}$ can take any value in $(0,+\infty)$, this is a trivial system to solve from here:  there is a Markov law invariant by this dynamic on the line iff $\T{[0,1,1]}{[0,0,1]}+\T{[1,1,0]}{[1,0,0]}$ and $\T{[0,0,1]}{[0,1,1]}+\T{[1,0,0]}{[1,1,0]}$ are both positive (or both 0, in which case all Markov laws are invariant).
\subsubsection{The contact process and some extensions}
\label{sec:MTSMcontact}

The Dirac measure on $0^{\Z}$ is invariant for the contact process. Another invariant distribution exists for $\lambda$ large enough with no atom at $0^{\Z}$ (Liggett \cite[Theo.1.33, Sec. VI]{LTIPS}).
We prove that this other invariant distribution is not Markovian with memory $m$, for any $\lambda>0$ and any $m\geq 1$.
The JRM $\TT$ of the contact process is not identically 0 and the contact process possesses $0^n$ as absorbing state on $\ZnZ$, thus, an immediate consequence of Theorem \ref{theo:fs} is:
\begin{cor}\label{cor:contcontact}
If a distribution $\mu\neq \delta_{0^{\mathbb{Z}}}$ is invariant for the contact process on the line, then $\mu$ is not a Markov law with memory $m$, for any $m$.    
\end{cor}
In fact, Theorem \ref{theo:fs} just states the non existence of invariant Markov law with memory $m$ with positive kernel. By the nature of the contact process, no other kernels are neither possible.

When solving the system for general rates and using $\T{[0,0,0]}{[0,1,0]}$ as a free parameter, meaning that it can take any real value, we found that a necessary condition to have a Markov law invariant by $\TT$ is that $\T{[0,0,0]}{[0,1,0]}>0$ which means that there is a sort of ``spontaneous infection''.

\subsubsection{Around TASEP}
\label{sec:MTSM}

The TASEP is the PS defined on the line, or on a segment (see Section \ref{sec:SB}) whose JRM $\TT$ is null, except for $\T{[1,0]}{[0,1]}=1$. Some variants of this model have been defined, we will explore some of them.

\paragraph{3-coloured TASEP}
The 3 coloured model ($\Ek=\{0,1,2\}$) for which again the JRM is null except for
\beq\label{eq:sdt2tasep}
\T{[1,0]}{[0,1]}=\T{[2,0]}{[0,2]}=\T{[2,1]}{[1,2]}=1
\eq
meaning that a particle can overtake smaller ones, with a constant rate. For more information on this type of PS and its invariant measure on some special cases see Angel \cite{AO} or Zhong-Jun \& al \cite{DZ}.

Here, we propose to replace the common value \eref{eq:sdt2tasep} by parameters, and use our theorems to characterize the set of 3-tuple $(\T{[1,0]}{[0,1]},\T{[2,0]}{[0,2]},\T{[2,1]}{[1,2]})$ for which an invariant Markov law with positive $M$ exists ($\TT$ being null besides).
The computation of the Gr\"obner basis of the system (with the additional polynomials $M_{i,j}g_{i,j}-1$ to prevent the $M_{i,j}$ to be 0) is rapid, but the expression of a Gr\"obner basis is too large to be written here. What can be observed is that a polynomial of the basis
is $-\T{[2, 0]}{[0, 2]}+\T{[2, 1]}{[1, 2]}+\T{[1, 0]}{[0, 1]}$ so that the nullity of this polynomial
is a necessary condition for existence of an invariant Markov law, in which case appears that $M$ must have constant lines, so that the distribution is a product measure with marginal $M_{0,.}$. Examining further the very simple Gr\"obner basis, appears that any product distribution $\rho^\Z$ with $\rho$ having support over $\{0,1,2\}$ is invariant! This can be checked by hand on $\Cycle_3^{\rho,\TT}\equiv 0$. 

\paragraph{-Variant: if particle $i$ can overtake $i-1 \mod 3$ only.}

Here the parameters are
\[\T{[0, 2]}{[2, 0]}, \T{[1, 0]}{[0, 1]}, \T{[2, 1]}{[1, 2]},\] meaning that now 0 can overtake 2, but not the contrary. The computation of a Gröbner basis provides a list of polynomials, among which one can found: 
$\T{[0, 2]}{[2, 0]}+\T{[2, 1]}{[1, 2]}+\T{[1, 0]}{[0, 1]}$. Since the $\TT$ are non negative numbers, the 3 parameters must be 0. Hence, the only case where a Markov law with positive kernel $M$ is invariant, is when no particle are allowed to move!

\paragraph{-Variant with parameters $\protect\T{[a,b]}{[b,a]}$}
This is a generalization of the two previous points. In this case, each particle can overtake the other ones. This is a case where the Gröbner basis are huge (more than seven hundred polynomials), with many very simple polynomial of the following kind:
\[\T{(2, 0)}{(0, 2)}\T{(2, 1)}{(1, 2)}(g_{0,1}-g_{2,1})^2\]
meaning that one of this three factors must be 0 to have a solution. In order to study completely this system a method consists from here, to choose such an equation and to constitute 3 systems from here, each of them, constituted by the initial system at which is added one of the factor above, as a new polynomial.

Due to the complexity of the system, the constitutions of these 3-subsystems is not enough to conclude (the obtained Gröbner basis stays large), but this method can be iterated if the complete set of solutions need to be found. 

\paragraph{Zero-range type processes}

We start with a preliminary definition
\begin{defi}\label{def:mp}
	A JRM is said to be mass preserving if $\T{(a,b)}{(c,d)}>0\imp a+b = c+d$. 
\end{defi}
In the literature, PS's associated to mass preserving $\TT$'s are called \textit{mass migration processes (MMP)} \cite{FGS} and \textit{mass transport models} \cite{GL,EMZ}.

The following definition will be useful to define this type of systems.
\begin{defi}\label{def:mp2} A mass preserving $\TT$ is said to be zero range mass preserving if there exists a function $g : \Ek^2 \to \R$ such that 
\[
\T{(a,b)}{(c,d)} = g(a,k) \1_{(c,d)=(a-k,b+k)}, \qquad\forall a,b,c,d,1\leq k\leq a.
\]
\end{defi} 
In words: the rate at which a part $k$ of the mass $a$ jumps to the next vertex at its right is $g(a,k)$ (for any $k$ legal, that is $1\leq k\leq a$).

PS's associated to zero range mass preserving $\TT$'s are called \textit{zero range mass migration processes (MMP-ZR)} \cite{FGS}. These types of processes are generalizations of TASEP, since they could be interpreted as particle systems where each site can host more than one particle and where particles in the same sites can jump at the same time (See \cite{AE,FGS}).

The zero-range mass migration process is a process on $E_\infty$ whose JRM is zero range mass preserving.

Let $\rho \in {\cal M}(\Ek)$ such that $\rho_0>0$.
In \cite[Proposition 3.10]{FGS} they obtained that $\rho^{\Z}$ is invariant for the MMP-ZR iff $\rho_a\rho_kg(k,k) = \rho_{a+k}\rho_0g(a+k,k)\quad \forall k\geq 1,\; \forall a\geq 1$.
\begin{defi}We say that a distribution $\rho$ on $\Ek$ is almost-geometrically distributed if there exists a function ${\sf g}:\Ek\to \R^+$ such that 
	\beq\label{eq:ergtu}
	\rho_u\rho_v = {\sf g}(u+v) \textrm{ for any }(u,v)\in\Ek^2.
	\eq
\end{defi}
The support of an almost-geometric distribution can be either finite or infinite. If the support is $\mathbb{N}$, then it is a geometric distribution (since $\rho_{a}\rho_b=\rho_{a+b}\rho_0$).

From \cite[Proposition 3.10]{FGS} and Theorem \ref{theo:alsdh}, we get immediately:
\begin{cor}
	Consider a zero range mass preserving $\TT$, with $g : \Ek^2 \to \R$ positive in the diagonal and $\rho\in {\cal M}(\Ek)$ with $\rho_0>0$ such that $\rho^\Z$ is $\AI$ by $\TT$. If for some $h:\Ek\to[0,\infty)$, $g(b,k) = h(b)g(k,k)$ for all $b\in \Ek$, then  $\nu^\mathbb{Z}$ is also $\AI$ by $\TT$, for all almost-geometric distributions $\nu$ with same support as $\rho$.
\end{cor}

Before presenting the next variant, we introduce a result that we will apply to it. This result holds in a more general setting and because of it we state it as a separate result.

\subsubsection{A family of models with an infinite number of invariant product distributions} 

The next theorem states that the family of mass preserving kernels having almost-geometric distributions as invariant distributions, have an infinite number of invariant distributions.

\begin{theo}\label{theo:alsdh} If $\rho^{\Z}$  is \AI by a mass preserving kernel $\TT$ for $\rho$ an almost-geometric distribution such that  \eref{eq:dfvasbv} and \eref{eq:fqd} hold, then for all almost-geometric distributions $\nu$ with same support as $\rho$, $\nu^\mathbb{Z}$ is also $\AI$ by $\TT$.
\end{theo}

\begin{proof}
	We use Theorem \ref{theo:tyuttr}.  Assume that $\nu$ is \AI by $\TT$.  Taking into account the discussion just above, consider $\Ek'=\Sup(\nu)$. A necessary condition for $\nu$ to be \AI by $\TT$ is that \eref{eq:LPCC} holds.
	Theorem \ref{theo:t3b} says: $\rho^\Z$ is \AI by $\TT$ iff $\NCycle_3^{\rho,\TT}({a,b,c})=0$ { for $(a,b,c)$ such that $\rho_a\rho_b\rho_c>0$.}
	Now, $\NCycle_3^{\rho,\TT}(a,b,c)=$
	\beq\label{eq:ryhfg}
	\sum_{u,v}\frac{\rho_{u}\rho_{v}}{\rho_{a}\rho_{b}}\TT_{(u,v)\atop (a,b)}-\TT_{(a,b)\atop (u,v)}+
	\sum_{u,v}\frac{\rho_{u}\rho_{v}}{\rho_{b}\rho_{c}}\TT_{(u,v)\atop (b,c)}-\TT_{(b,c)\atop (u,v)}+
	\sum_{u,v}\frac{\rho_{u}\rho_{v}}{\rho_{c}\rho_{a}}\TT_{(u,v)\atop (c,a)}-\TT_{(c,a)\atop (u,v)}.
	\eq
	From the definition of a mass preserving kernel, the first sum can be restricted to $I_{a+b}$, where for any $k$, $I_k$ is the set of pairs $(u,v)$ such that $u+v=k$; the second one can be restricted to $I_{b+c}$, and the last, to $I_{a+c}$.
	Now, \eref{eq:ergtu} can be used since for all $(u,v)\in I_{a+b}$, $\rho_{u}\rho_v/(\rho_{a}\rho_b)=1$, and one sees that \eref{eq:ryhfg} rewrites in this case $\NCycle_3^{\rho,\TT}(a,b,c)=$
	\beq\label{eq:sg}
	\sum_{(u,v)\in I_{a+b}}\TT_{(u,v)\atop (a,b)}-\TT_{(a,b)\atop (u,v)}+
	\sum_{(u,v)\in I_{b+c}}\TT_{(u,v)\atop (b,c)}-\TT_{(b,c)\atop (u,v)}+
	\sum_{(u,v)\in I_{a+c}}\TT_{(u,v)\atop (c,a)}-\TT_{(c,a)\atop (u,v)}.
	\eq
	The steps which brings us to \eref{eq:sg}  is valid for any $(\rho,{\sf g})$ satisfying \eref{eq:ergtu} so the theorem is proved.
\end{proof}

Now we return to a new variant of TASEP.

\paragraph{PushTASEP:} The PushASEP is the PS defined on $E_2^{\mathbb{Z}}$ where $0$ represents an empty site and 1 an occupied site. The dynamics are described as follows: each particle tries to jump to the right at rate 1, and it actually jumps if the site is empty. Moreover, each particle jumps to the closest empty site at its left with rate 1. This type of PS has range $L=\infty$. However, each configuration can be encoded by the consecutive size of the blocks along the line, where a block is constituted with an empty site together with the set of consecutive occupied sites at its left. The dynamics of the PushASEP induces a PS on the ``block size process'' with range $L=2$ and $\kappa=\infty$ (all block sizes starting by 1 are possible). For this induced PS,  the product measure with marginal the geometric distribution (for any parameter in (0,1) by Theorem \ref{theo:alsdh}) is $\AI$ by $\TT$. This provides a description of some invariant distributions for the PushTASEP.

\subsection{Projection and hidden Markov chain}
\label{sec:HMC}
This part also illustrates our theorems: with Theorem \ref{theo:t0prime} one can find JRM $\TT$ on $\Ek^\mathbb{Z}$ for some $\kappa\geq 3$ (with more than 3 colours) having some Markovian invariant distribution. Some of them, possess some nice projection properties: they allow to characterize some PS invariant distributions on $\{0,1\}$ (and probably of some PS with more than 2 colours) having as invariant distribution, the distribution of some hidden Markov chain (see Cappé \& al. for more information on these models). 

Consider $\TT$ and $\TTp$ be two JRM of two PS defined respectively ${E}^\Z$ and $F^{\mathbb{Z}}$, where $E$ and $F$ are two spaces of colours such that  $\# F < \# E$. Consider $\pi$ a surjective map from $E$ on $F$: with each colour $c$ in $F$, one or several colours $\pi^{-1}(c)$ of $E$ are associated by $\pi$ (on an exclusive basis). 
\begin{defi}
$\TT'$ is said to be the $\pi$-projection of $\TT$ if for any $a,b,c,d \in F$, any $(A,B)\in \l(\pi^{-1}(a),\pi^{-1}(b)\r)$
\beq\label{eq:Tproj}
\sum_{(C,D)\in \pi^{-1}(c)\times \pi^{-1}(d)}\T{[A,B]}{[C,D]}=\Tp{[a,b]}{[c,d]}.
\eq
\end{defi}
In words: starting from any representative $(A,B)$ of $(a,b)$, the total jump rate to the representatives of $(c,d)$ does not depend on $(A,B)$, but only on $(a,b)$. 

\begin{lem} Consider $\eta=(\eta_t, t\geq 0)$ with $\eta_t=(\eta_t(k),k\in \Z)$  a well defined PS defined on $E^{\Z}$ with JRM $\TT$, for some finite $E$. Assume that $\TT'$ is the $\pi$-projection of $\TT$ for a surjection $\pi:E\to F$, for some set $F$. 
Under these hypothesis $\eta'=(\eta'_t,t\geq 0)$ defined by
$\eta'_t=\l(\pi\l(\eta_t(k)\r),k \in \Z\r)$ 
is a PS with JRM $\TT'$. Hence, if $\mu$ is a measure invariant by $\TT$ on $E^{\Z}$, then $\mu\circ \pi^{-1}$ is invariant by $\TT'$ on $F^\Z$. 
\end{lem}
\begin{proof} Since the PS under investigation is translation invariant, we focus on the finite dimensional distribution evolution at the right of zero. Consider any word $(x_1,\cdots,x_k)$ whose projection is $(\pi(x_j), 1\leq j \leq k)=(y_1,\cdots,y_k)$. Now, consider the rate of jumps from any subword $(y_\ell,y_{\ell+1})=(a,b)$ to  $(y'_\ell,y'_{\ell+1})=(c,d)$.
By definition, $\pi(x_\ell)=y_\ell,\pi(x_{\ell+1})=y_{\ell+1}$, and $(y_\ell,y_{\ell+1})$ jumps to $(y'_\ell,y'_{\ell+1})$ if $(x_\ell,x_{\ell+1})$ jumps to any representative $(c,d)$ that is, if its jumps to $\pi^{-1}(c)\times\pi^{-1}(d)$. 
Hence, the total jump rate from $(a,b)$ to $(c,d)$ is given by $\sum_{(C,D) \in \pi^{-1}(c)\times \pi^{-1}(d)}\TT{[x_\ell,x_{\ell+1}]}{[C,D]}$, and this is indeed $\Tp{[a,b]}{[c,d]}$ (for any value $(x_\ell,x_{\ell+1})\in\pi^{-1}(a)\times\pi^{-1}(b)$). The statement concerning the invariance of $\mu\circ \pi^{-1}$ is direct.
\end{proof}

There exist in the literature several definitions for the notion of hidden Markov chains. The most classical is the following:
\begin{defi} $(Y_k,k\in \mathbb{Z})$ is said to be a hidden Markov chain taking its values in $F^{\mathbb{Z}}$, if it has the following representation:\\
  -- there exists a Markov chain $(Z_k, k \in\mathbb{Z})$ taking its values in some set $E^{\mathbb{Z}}$,\\
  -- there exists a transition kernel $K=(K(a,b))_{a \in E, b \in F}$;\\
  such that, conditionaly on $Z=(Z_k,k\in \cro{a,b})$, the $Y_j$'s are independent and  conditionaly on $Z$  the distribution of $Y_j$ is given by $K(Z_j,.)$.
\end{defi}
Hence, if $(X_k,k\in \Z)$ is a Markov chain with state space $E$, and $\pi:E\to F$ is a surjection (or just a map) then since the process $\l(\pi(X_k),k\in \mathbb{Z}\r)$ is a hidden Markov chain.  If $X$ has initial distribution $\rho$ at time 0, and kernel $M$, then
\beq\label{eq:projm}
\P(Y_j=y_j, 0 \leq j \leq n)=\sum_{(x_0,\cdots,x_n)\in E^{n+1}\atop {x_j \in \pi^{-1}(y_j) , 0\leq j \leq n}} \rho_{x_0} \prod_{j=1}^{n} M_{x_{j-1},x_j}
\eq
From there, it may be checked that a hidden Markov chain is not a Markov chain in general (with any memory), since in general, \eref{eq:projm} does not factorize suitably.

Now, we state the following result:
\begin{theo} There exist some PS on $\{0,1\}^{\mathbb{Z}}$ which admits some hidden Markov chain (which are not Markov chains) as invariant distributions.
\end{theo}
The proof is constructive, we will provide an example. 
Consider the 4-tuples $(\TT,\TT',M,\pi)$ as follows\\
-- Take $\pi:E_3\to E_2$ defined by $\pi(0)=0, ~~\pi(1)=\pi(2)=1$.\\
-- Take $L=3$ (the range), and $\TT$ with entries all 0 except 
\be
\T{[0,0,0]}{[0,1,0]}&=& 255,~~
\T{[0,0,0]}{[0,2,0]}= 15\\ 
\T{[0,1,0]}{[0,0,0]} &=& \T{[0,2,0]}{[0,0,0]} = 294\\
\T{[0,1,0]}{[0,2,0]} &=& \T{[0,2,0]}{[0,1,0]} = 49 
\ee
-- The projected JRM $\TT'$ has all its entries 0 except $\Tp{[0,0,0]}{[0,1,0]}=270$, $\Tp{(0,1,0)}{(0,0,0)}=294$ (notice that the jump $(0,1,0)\to(0,2,0)$ does not project on a ``true jump'' since it would correspond to the jump $(0,1,0)\to(0,1,0)$).  \\
-- Take the Markov kernel 
\[M=\begin{bmatrix} 7/{15} & {1}/{3} & {1}/{5} \\
  {1}/{2} & {1}/{6} & {1}/{3}\\
  {1}/{6} & {1}/{2} & {1}/{3}
\end{bmatrix}\]
and initial  distribution is $r_0 = 35/89, r_1 = 29/89, r_2 = 25/89$.\\
If $(X_j,j\geq 0)$ is a Markov chain with kernel $M$ and initial distribution $r$, then $(\pi(X_k),k\geq 0)$ is not a Markov chain since $\mu$ the projected measure satisfies $\mu([1,1,1])/\mu([1,1])=71/106 \neq \mu([1,1])/\mu([1])=53/81$, when Markovianity would imply equality of these quantities; it is a hidden Markov chain.\\
-- It remains to say that the Markov law $(\rho,M)$ is invariant by $\TT$. This can be proved by checking that for any $a,b,c,d,e \in \Ek$, $\NCycle_9^{M,\TT}(a,b,c,d,e,0,0,0,0)=0$: in fact, the corresponding function $Z^{M,\TT}\equiv 0$.

Notice that, for example, the Dirac measure $\delta_{1}^{\Z}$ is invariant on the line for this PS, and the analogous on $\ZnZ$, but  this configuration is not attractive.

\subsection{Exhaustive solution for the $\kappa=2$-colour case with $m=1$ and $L=2$}

\subsubsection{Invariant Markov laws}
\label{sec:AIMD}
To find all JRM $\TT$ for which exists an invariant Markov law when $\kappa=2$ can be solved completely by computation of a Gröbner basis.  Instead of writing $\T{(a,b)}{(c,d)}$, we write $t_{x,y}$ where $x$ and $y$ are the numbers in base 10 corresponding to $ab$ and $cd$ seen as a number in base 2: write for short $x=(a,b)_2$, and $y=(c,d)_2$, so that $3=(1,1)_2$, $1=(0,1)_2$. Hence, we have $t_{i,i}=0$ for $i$ from 0 to 3, and $t_{3,2}=\T{(1,1)}{(1,0)}$.\par
Now, set $M_{0,0}=1-M_{0,1}$, $M_{1,0}=1-M_{1,1}$ so that $M_{0,1}$ and $M_{1,1}$ are the remaining variables and now write the system which contains:\\
\tb $\Cycle_7^{M,\TT}\equiv 0$,\\
\tb the equations $M_{a,b}g_{a,b}-1$ for any $a,b\in \{0,1\}$ for additional variables  $g_{0,0},\cdots,g_{1,1}$,\\
\tb an additional equation  $(M_{0,1}-M_{1,1})x-1$ in order to remove the i.i.d. case (treated below), for some new variables $x$.\\
The Gröbner basis, of this system, is too long to be written here; nevertheless, here are the first polynomials of the obtained basis, which provide some necessary conditions:
\[\begin{array}{l} 
t_{3,0},\;\; t_{2,1},\;\;t_{1,2},\;\;  t_{0,3},\\
t_{{0,1}}t_{{1,0}}+t_{{0,1}}t_{{2,0}}+t_{{0,2}}t_{{1,0}}+t_{{0,2}}t_{{
		2,0}}-t_{{1,3}}t_{{3,1}}-t_{{1,3}}t_{{3,2}}-t_{{2,3}}t_{{3,1}}-t_{{2,3
	}}t_{{3,2}},\\
 \left( t_{{1,0}}+t_{{2,0}}-t_{{3,1}}-t_{{3,2}} \right) {M_{{1,1}}}^{2
}+ \left( -t_{{1,0}}-t_{{1,3}}-t_{{2,0}}-t_{{2,3}} \right) M_{{1,1}}+t
_{{1,3}}+t_{{2,3}}
\end{array}\]
Hence,  $t_{3,0} = t_{2,1} = t_{1,2} =  t_{0,3}=0$ is a necessary condition. The polynomial on the second line expresses somehow ``the important condition'', and the third line polynomial (and subsequent, see \cite{LF-JFM}) allow to compute the kernel $M$. We infer that
\begin{cor} 
	For $\kappa=2$. If $\TT$ is mass preserving, then it does not exist any  $(\rho,M)$-Markov law with $M_{0,1}\neq M_{1,1}$, and with coefficients in $(0,1)$, invariant by $\TT$.
\end{cor}

\subsubsection{Invariant product measure}
\label{sec:AIMD2}
First, we claim that, 
\begin{lem}
	If $\kappa = 2$, then $\rho^\Z$ is invariant by $\TT$ on the line iff  $\NCycle_2^{\rho,\TT}\equiv0$.\end{lem}
\begin{proof} By Theorem \ref{theo:t3b} $(vii)$, it suffices to prove that 
	$\NCycle_3^{\rho,\TT}\equiv 0 \equi \NCycle_2^{\rho,\TT}\equiv0$. 
	Start by the implication: from $\NCycle_3^{\rho,\TT}(a,a,a)=0=3\wZ^{\rho,\TT}_{a,a}$ we infer that $\wZ^{\rho,\TT}_{a,a}=0$ for any $a\in\Ek =\{0,1\}$.
	Now, write 
	\beq\label{eq:zezg}
	\NCycle_3^{\rho,\TT}(aab)=\wZ^{\rho,\TT}_{a,a}+\wZ^{\rho,\TT}_{a,b}+\wZ^{\rho,\TT}_{b,a}=0+\NCycle^{\rho,\TT}_2(a,b),\eq
	so that the implication holds. For the converse, note that any words $w$ with three letters on $E_2=\{0,1\}$ possesses one letter repeated. Given the cyclical structure of the equation $\NCycle_3^{\rho,\TT}$ (that is $\NCycle_3^{\rho,\TT}(abc)= \NCycle_3^{\rho,\TT}(bca)$) it suffices to prove that $\NCycle_3^{\rho,\TT}(aab)=0$ for any $a,b \in \Ek$.
	Now, from  $\NCycle_2^{\rho,\TT}\equiv0$, we deduce $\wZ^{\rho,\TT}_{a,a}=0$, and still from \eref{eq:zezg}, $\NCycle_3^{\rho,\TT}\equiv 0$.  
\end{proof}

Now solving explicitly $\NCycle_2^{\rho,\TT}\equiv0$ using a computer algebra system is possible.
The result is presented in \cite{LF-JFM}; there are 5 polynomial, including the following one (product of the 2 first lines minus the third)
\ben
&& \left( t_{{1,0}}t_{{3,0}}+t_{{1,0}}t_{{3,1}}+t_{{1,0}}t_{{3,2}}+t_{{1
,3}}t_{{3,0}}+t_{{2,0}}t_{{3,0}}+t_{{2,0}}t_{{3,1}}+t_{{2,0}}t_{{3,2}}
+t_{{2,3}}t_{{3,0}} \right)\\
&&\left( t_{{0,1}}t_{{1,3}}+t_{{0,1}}t_{{2,
3}}+t_{{0,2}}t_{{1,3}}+t_{{0,2}}t_{{2,3}}+t_{{0,3}}t_{{1,0}}+t_{{0,3}}
t_{{1,3}}+t_{{0,3}}t_{{2,0}}+t_{{0,3}}t_{{2,3}} \right)\\
&&- \left( t_{{0
,1}}t_{{3,0}}+t_{{0,1}}t_{{3,1}}+t_{{0,1}}t_{{3,2}}+t_{{0,2}}t_{{3,0}}
+t_{{0,2}}t_{{3,1}}+t_{{0,2}}t_{{3,2}}+t_{{0,3}}t_{{3,1}}+t_{{0,3}}t_{
{3,2}} \right) ^{2}
\een
which is the a necessary condition on $t$ to have a product measure as invariant distribution.

\subsection{2D applications}
\label{sec:2dapp}
The criterion provided by Theorem \ref{theo:2D} seems to depend on all the colourings of the neighbors of $\Gamma_0$, $\Gamma_1$ and $\Gamma_2$, which represents for this last case, as many as $\kappa^{14}$ possibilities, and this for each of the $\kappa^5$ different configurations in $\Gamma_2$. So, the total number of equations seems out of reach, but in fact, again, \eref{eq:lineq2D} is decomposed on a sum of ${\bf Z}_{x(h\cap C)}^{h\cap C,h}$ (defined in \eref{eq:Zh}) so that it suffices to express these functions which intersect the domain $D$ under inspection: the contribution of each square can be computed independently.
This provides a small finite set of functions with 1, 2, 3 or 4 variables (as discussed in \eref{eq:nl2} and in the proof of Theorem \ref{theo:2D}):  When $\Ek=E_2=\{0,1\}$, this provides a small quantity of functions, each of them being a sums of at most $2^3$ elementary quantities. The corresponding set of equations can be written easily, or even automatically if needed. When $\TT$ is totally specified, searching invariant distribution amounts then just to solving a polynomial system with unknown $\rho_0,\cdots,\rho_{\kappa-1}$ in the set $\{(r_0,\cdots,r_{\kappa-1}) \in [0,1]^\kappa : \sum r_i=1\}$. Here are some cases we have investigated (we insist on the fact that all these examples have been found with very few manipulations, in a very short time):\\
{\bs} If all the $\TT$'s are zero except $
\tt11100001=a, \tt00011110=1$
(this generalizes a bit \eref{eq:egrda}) then the Bernoulli product measure with parameter $\rho_1\in(0,1)$ is invariant iff $a\rho_1^2-\rho_1^2+2\rho_1-1=0$ (so that for a given $a$, the density is $\rho_1=1/(\sqrt{a}+1)$). This can be checked by hand with our criterion, or just using a reversibility argument, as \eref{eq:fqdq}, for example.\\
{\bs} Similarly, with the same methods, one checks that if all the $\TT$'s are zero except
  $
    \tt10100101=a, \tt01011010=b
  $
    then for any $(a,b)\in (0,+\infty)^2$, all product measures are invariant if $a=b$, and none otherwise (the first statement is a consequence of reversibility, but reversibility cannot be used to prove the second).\\
    {\bs} if all the $\TT$'s are zero, except
    $\tt11000110 $, then no product measures with full support are invariant.\\
{\bs}  if all the $\TT$'s are zero, except
    \[\tt11000110 =a,\tt01100011 =b,
          \tt00111001 =c,
          \tt10011100 =d \]
           then if $a=b=c=d$, all Bernoulli product measure with parameter in $(0,1)$ are invariant, otherwise, there is no invariant product measure.\\
{\bs} Now, if one lets many free parameters:
          If all the $\TT$'s are 0, except those with the form
          \[\tt{a}{b}{c}{d}{a}{{1-b}}{{1-c}}{d} \textrm{ for }a,b,c,d\in\{0,1\}.\]
In this case, the space of parameters for which there exists invariant product measures is quite complex: see \cite{LF-JFM}.\\
{\bs} In the 3 colours case $E_3=\{0,1,2\}$ with all the parameters $\TT$ equal to zero, except
  \[\tt{i}iii{{i+1 \mod 3}}{{i+1\mod 3}}{{i+1\mod 3}}{{i+1\mod 3}}=a_i.\] The set of invariant product measures with marginal having support $E_2$, are the measures $\rho\in{\cal M}(E_2)$ satisfying $ a_1\rho_1^4-a_2\rho_2^4= a_0\rho_0^4-a_2\rho_2^4=0$ and $\rho_0,\rho_1,\rho_2>0$.\\
{\bs}    We may design similarly many JRM $\TT$ preserving $P_\lambda^{\Z^2}$ where $P_\lambda$ is the Poisson$(\lambda)$ distribution, by considering mass preserving $\TT$, which moreover, preserves the Poisson distribution on a square (still using  Corollary \ref{cor:dsdq}). Many such dynamics exist, and this can be analyzed  on the square: condition on the sum $m$ of the 4 values around the square, and interpret the 4 (multinomial) concerned variables as the number of balls in 4 urns in which $m$ balls labelled from $1$ to $m$ have been dropped uniformly and independently (in a larger probability space). Picking a ball at random and moving it in the next urn around the square, or shifting the urns around the square, or taking a ball and reinserting it randomly in any of the other three urns, are three examples of dynamics that preserve the multinomial distribution. \par
  For the last example,  for each positive map $m\to W_m$, the following JRM $\TT$ preserves $P_\lambda^{\Z^2}$:
  \[
	   \tt{{x_1}}{{x_2}}{{x_3}}{{x_4}}{{y_1}}{{y_2}}{{y_3}}{{y_4}} =
	             W_{\|x\cro{1,4}\|_1}\times \dis\frac{x_i}{3}\]
 if $ y\cro{1,4}=x\cro{1,4}-e_i+e_j$  for $j\in \{ 1,2,3,4\} \setminus \{i\}$, where $e_k$ is the $k$-th canonical vector of $\R^4$.

\section{Proofs}\label{sec:proofs}

\subsection{Proof of Theorem \ref{theo:t0}}
\label{sec:PF_theo:t0}

Before proving Theorem \ref{theo:t0}, we establish a Lemma used all along the proof.  
Recall the representation formula of  $\NLineq_n^{\rho,M,\TT}$ (for $n\geq 3$) in terms of the functions $\wZ$ given in \eref{eq:rsgfqe12}.
\begin{lem}
\label{lem:balanced}
Assume that $M$ is a positive Markov kernel, then for any JRM $\TT$,
\[
	\sum_{b,c}\wZ_{a,b,c,d} M_{a,b}M_{b,c}M_{c,d}=0,\quad \forall a,d\in \Ek.
\]
\end{lem}
\begin{proof}
Just expand $\wZ$. By definition
\[
\sum_{b,c}\wZ_{a,b,c,d} M_{a,b}M_{b,c}M_{c,d}=
\sum_{u,v}\sum_{b,c}\T{(u,v)}{(b,c)}M_{a,u}M_{u,v}M_{v,d}-\sum_{b,c}\To{b,c}M_{a,b}M_{b,c}M_{c,d}.
\]
This is zero since $\sum_{b,c}\T{(u,v)}{(b,c)}=\To{u,v}$.
\end{proof}

\begin{rem}\label{rem:rthgf}
  The Lemma does not use the fact that the $(\rho,M)$-Markov law is invariant by $\TT$ on the line, but just a kind of local equilibrium and the form of $\wZ$. In fact, what is true in all generality is that, for any function $f$ taking its values in $\R^{\star}$,  for $\wZ'_{a,b,c,d}= -\To{b,c}+ \sum_{u,v}\frac{f(a,u,v,d)}{f(a,b,c,d)} \T{(u,v)}{(b,c)}$  then, one has $\sum_{b,c}\wZ'_{a,b,c,d} f({a,b,c,d})=0$. This is particularly true if $f(a,b,c,d)=M_{a,b}M_{b,c}M_{c,d}$ or $f(a,b,c,d)=\alpha_a M_{a,b}M_{b,c}M_{c,d} \beta_d$.
  \end{rem}

\noindent To prove Theorem \ref{theo:t0} we will show two cyclical implications $(i)\imp (ii)\imp(iii)\imp(iv)\imp (v)\imp (i)$ and  $(v)\imp (vi)\imp (vii)\imp (viii)\imp(ix)\imp (v)$.\\
$\bullet$ {\bf Proof of $(i)\imp (ii)$}
Observe the contribution of the term with index $j$  for a $j$ in $\cro{2,n-2}$ that is, far from 0 and from $n$ in \eref{eq:tjht}.
One sees that since for any $a$, $\sum_{b} M_{a,b}=1$, and $\rho M=\rho$, for any  $1<j<n-1$, 
\beq\label{eq:tjht2}
\sum_{x_{-1},x_0,\atop x_{n+1},x_{n+2}} \wZ^{M,\TT}_{x\cro{j-1,j+2}} \rho_{x_{-1}}\prod_{k\in\{ -1,0,n,n+1 \}}M_{x_k,x_{k+1}} =\rho_{x_1}  \wZ^{M,\TT}_{x\cro{j-1,j+2}}.
\eq
Take three arbitrary words $x\cro{1,n}$, $y\cro{1,m}$ and $a\cro{1,7}$ with letters in $\Ek$,  and $a_4'\in \Ek$. Define 
\be
w=x\cro{1,n}\,a\cro{1,4}\,000\,y\cro{1,m} \quad \text{ and }\quad w'=x\cro{1,n}\,a\cro{1,3}\,0000\,y\cro{1,m},
\ee
we recall that the concatenation gives for example: $a\cro{2,4}\,b\,y\cro{3,6}=a_2a_3a_4yb_3b_4b_5b_6$.
Using the property \eref{eq:tjht2} and the fact that the boundary terms are the same (those for $j\in\{0,1,n-1,n\}$), we get 
\[
\NLineq_{N}^{\rho,M,\TT}(w)-\NLineq_{N-1}^{\rho,M,\TT}(w')=\rho_{x_1}\Rep_7^{M,\TT}(a_1,a_2,a_3,a_4,0,0,0;0).
\]
If the $(\rho,M)$-Markov law is invariant by $\TT$, then $\NLineq_{N}^{\rho,M,\TT}\equiv 0$ as well as $\NLineq_{N-1}^{\rho,M,\TT}\equiv 0$ so that $\Rep_7^{M,\TT}(a_1,a_2,a_3,a_4,0,0,0;0)= 0$.\\
$\bullet$ \textbf{Proof of $(ii)\imp (iii)$}
 For $n\geq 4$ define the map $H_n:\Ek^{\cro{1,n}}\to \R$ by
 \ben\label{eq:etjsrgryj}
 H_n(a\cro{1,n}) &:=&\l(\sum_{j=1}^{n-3} \wZ^{M,\TT}_{a\cro{j,j+3}}\r)
 \\
&&+\l(\wZ^{M,\TT}_{a_{n-2},a_{n-1},a_n,0}
 +\wZ^{M,\TT}_{a_{n-1},a_n,0,0}+\wZ^{M,\TT}_{a_n,0,0,0}\r)-(n-4)\wZ^{M,\TT}_{0,0,0,0}.\een
If $\Rep_7^{M,\TT}(a,b,c,d,0,0,0;0) = 0$ for all $a,b,c,d \in \Ek$, then for $n\geq 4$,
\ben\label{eq:HnHnm1} H_n(a\cro{1,n})-H_{n-1}(a\cro{1,n-1})=0.\een
Indeed, since $\Rep_7(a\cro{n-3,n}\,000;0)=0$,  
\[\wZ^{M,\TT}_{a_{n-2},a_{n-1},a_n,0}+\wZ^{M,\TT}_{a_{n-1},a_n,0,0}+\wZ^{M,\TT}_{a_n,0,0,0}+4\wZ^{M,\TT}_{0,0,0,0}=\wZ^{M,\TT}_{a_{n-2},a_{n-1},0,0}+\wZ^{M,\TT}_{a_{n-1},0,0,0}+\wZ^{M,\TT}_{0,0,0,0}+4\wZ^{M,\TT}_{0,0,0,0}.\]
Hence by \eref{eq:HnHnm1}, $H_n(a\cro{1,n})=H_3(a_1,a_2,a_3)$ and then, depends only on $a_1,a_2,_3$. It follows that $H_7(a\cro{1,7}) = H_7(a_1,a_2,a_3,a_4',a_5,a_6,a_7)$ which implies $\Rep_7(a\cro{1,7};a_4')=0$.~\\
$\bullet$ \textbf{Proof of $(iii)\imp (iv)$}
We first prove that when $(iii)$ holds,  $\NCycle_4^{M,\TT}\equiv 0$. For this just observe that $\Rep_7(a,b,c,d,a,b,c;0)=0$ implies  
\[
\NCycle_4^{M,\TT}(a,b,c,d) =  \NCycle_4^{M,\TT}(a,b,c,0).
\]
By cyclical invariance of the map $\NCycle_4^{M,\TT}$ this implies that
\[\NCycle_4^{M,\TT}(a,b,c,d) =  \NCycle_4^{M,\TT}(0,0,0,0)=4\wZ^{M,\TT}_{0,0,0,0}.\]
From that equality, multiplying both sides by $M_{a,b}M_{b,c}M_{c,d}M_{d,a}$ and summing over $a,b,c,d \in \Ek$ we obtain
\be
	4\wZ^{M,\TT}_{0,0,0,0}\times \Trace(M^4) &=& \sum_{a,b,c,d\in \Ek}\NCycle_4(a,b,c,d) M_{a,b}M_{b,c}M_{c,d}M_{d,a}\\
	&=& 4\sum_{a,d\in \Ek}M_{d,a}\Big( \sum_{b,c} \wZ^{M,\TT}_{a,b,c,d}M_{a,b}M_{b,c}M_{c,d}\Big),
        \ee  By Lemma \ref{lem:balanced}, this last quantity is 0, and then $\wZ^{M,\TT}_{0,0,0,0}=0$.\par
To end the proof, it suffices to observe that when $\wZ^{M,\TT}_{0,0,0,0} = 0$, for any $a,b,c,d \in \Ek$, 

$\Rep_7^{M,\TT}(a,b,c,d,0,0,0;0) = \ME_7^{M,\TT}(a,b,c,d,0,0,0).$
~\\
$\bullet$  {\bf Proof of  $(iv)\imp (v)$} Suppose $(iv)$. $\ME_7(0,0,0,0,0,0,0)=0$ implies that $\wZ^{M,\TT}_{0,0,0,0}=0$. Recall the map $H_n$ defined in \eref{eq:etjsrgryj} and replace $\wZ^{M,\TT}_{0,0,0,0}$ by 0 inside.
 Using now that  $\ME_7(a\cro{n-3,n}\,000)=0$, one has for any $n\geq 4$,
 \[H_n(a\cro{1,n})=H_{n-1}(a\cro{1,n-1}).\]
 Hence $H_n(a\cro{1,n})$ is a function of $(a_1,a_2,a_3)$ only and it does not depend on $n$. Therefore, since 
$
\ME_7(a,b,c,d,e,f,g)=H_{7}(a,b,c,d,e,f,g)-H_{6}(a,b,c,e,f,g)
$
we see that this quantity is 0.\\
{\bf Proof of  $(v)\imp (i)$} 	
We will need three intermediate results:
\begin{lem}\label{lem3} Assume that $M$ is a positive Markov kernel.
  If $\ME_7^{M,\TT}\equiv 0$ then for all $n\geq 3$, all $m \in \cro{ 2, n-1 }$, 
	\[	\NLineq_n^{\rho,M,\TT}(x\cro{1,n})=\NLineq_{n-1}^{\rho,M,\TT}(x\cro{1,n}^{\{m\}})   , \textrm{ for all } x\cro{1,n}\in \Ek^n    \]
        and then
        \[\NLineq_n^{\rho,M,\TT}(x\cro{1,n})=        \NLineq_{2}^{\rho,M,\TT}(x_1,x_n)  , \textrm{ for all } x\cro{1,n}\in \Ek^n .	\]
\end{lem}
\begin{proof}The second statement is a consequence of the first one.
  Following Remark \ref{rem:kyrje}, when $\ME_7^{M,\TT}\equiv 0$, we may replace any sum of the form
  \[S(x\cro{-1,n+2}):=\sum_{j=0}^{n} \wZ_{x\cro{j-1,j+2}},\] which depends on the word $x\cro{-1,n+2}$, with the sum $S(x\cro{-1,n+2}^{\{m\}})$, that is corresponding with the same word with the letter with index $m$ removed for any $m\in \cro{2,n-1}$. Hence from \eref{eq:rsgfqe12},
    \be
    \NLineq_n^{\rho,M,\TT}(x\cro{1,n}) &=& \sum_{x_{-1},x_0,\atop x_{n+1},x_{n+2}} S(x\cro{-1,n+2}) \rho_{x_{-1}}\prod_{k\in\{ -1,0,n,n+1 \}}M_{x_k,x_{k+1}} \\
    &=&\sum_{x_{-1},x_0,\atop x_{n+1},x_{n+2}}S(x\cro{-1,n+2}^{\{m\}})
  \rho_{x_{-1}}\prod_{k\in\{ -1,0,n,n+1 \}}M_{x_k,x_{k+1}}    \ee
and this is  $\NLineq_n^{\rho,M,\TT}(x\cro{1,n}^{\{m\}})$ since $m$ is not equal to 1 or to $n$.   
\end{proof}

\begin{lem}\label{lem2} Assume that $M$ is a positive Markov kernel.
 For all  $a \in \Ek$,
	\[
	\sum_{b}\NLineq_2^{\rho,M,\TT}(a,b)M_{a,b}=\NLineq_{1}^{\rho,M,\TT}(a) = \sum_{b}\NLineq_2^{\rho,M,\TT}(ba)M_{b,a}.
    \]
    and more generally, for any $x\cro{1,n}$ for $n\geq 1$,
    	\[
	\sum_{b}\NLineq_{n+1}^{\rho,M,\TT}(x\cro{1,n}b)M_{x_n,b}=\NLineq_n^{\rho,M,\TT}(x\cro{1,n})= \sum_{b}\NLineq_{n+1}^{\rho,M,\TT}(bx\cro{1,n})M_{b,x_1}.
    \]
          \end{lem}
      \begin{proof}By \eref{eq:rsgfqe12} and \eref{eq:tjht}
\[\sum_{x_2}\NLineq_2^{\rho,M,\TT}(x_1,x_2)M_{x_1,x_2}= \sum_{x_{-1},x_0,x_2,\atop x_{3},x_{4}}  \l(\rho_{x_{-1}}\prod_{k=-1}^3 M_{x_k,x_{k+1}}\r) \sum_{j=0}^{2} \wZ^{M,\TT}_{x\cro{j-1,j+2}}.\]
The contribution of the term $j=2$, is
\[ \sum_{x_{-1},x_0}  \l(\rho_{x_{-1}}\prod_{k=-1}^0 M_{x_k,x_{k+1}}\r) \sum_{x_2,x_3,x_4} M_{x_1,x_2}M_{x_2,x_3}M_{x_3,x_4}\wZ^{M,\TT}_{x\cro{1,4}},\]
which is 0 by the Lemma \ref{lem:balanced}.
Therefore \[\sum_{x_2}\NLineq_2^{\rho,M,\TT}(x_1,x_2)M_{x_1,x_2}=  \sum_{x_{-1},x_0,x_2,\atop x_{3},x_{4}}  \l(\rho_{x_{-1}}\prod_{k=-1}^3 M_{x_k,x_{k+1}}\r) \sum_{j=0}^{1} \wZ^{M,\TT}_{x\cro{j-1,j+2}};\]
the sum on $x_4$ simplifies (because $\sum_{x_4}M_{x_3,x_4}=1$), which gives the expected result. The proof of the second statement and of the generalization to larger words, can be obtained similarly.
\end{proof}

\begin{lem}\label{lem:4}
If $M$ is a positive Markov kernel and if $\ME_7^{M,\TT}\equiv 0$, then $\NLineq_2^{\rho,M,\TT}\equiv 0 $ and $\NLineq_1^{\rho,M,\TT} \equiv 0$. 
\end{lem}
\begin{proof}
Using Lemma \ref{lem2} and then Lemma \ref{lem3} which asserts that one can suppress the middle letter in the argument of $\NLineq_3^{\rho,M,\TT}$, for any $a,b\in \Ek$,
\beq\label{eq:ethhgzf}
\NLineq_2^{\rho,M,\TT}(ab) 
= \sum_{c}\NLineq_3^{\rho,M,\TT}(abc)M_{b,c}
= \sum_{c}\NLineq_2^{\rho,M,\TT}(ac)M_{b,c}.
\eq Set $M^t$ the matrix transposed of $M$, 
and, for a fixed $a\in \Ek$, consider the row vector
\[v_a= \begin{bmatrix} \NLineq_2^{\rho,M,\TT}(ab) , b \in \Ek \end{bmatrix}.\]
The equality between the leftmost and rightmost quantities in \eref{eq:ethhgzf} 
can be written
$v_a= v_a M^t$,
so that it is apparent that $v_a$ is a left eigenvector of $M^t$, associated with the eigenvalue 1. Since $M$ is a positive Markov kernel, 
$v_a = \lambda_a \begin{bmatrix} 1,\cdots,1\end{bmatrix}$
for some $\lambda_a\in \R$. Therefore
$\NLineq_2^{\rho,M,\TT}(a,b)=\lambda_a,$
 then  $\NLineq_2^{\rho,M,\TT}(a,b)$ does not depend on $b$.
  Now, notice that by translation invariance $\sum_{a} \Lineq_2^{\rho,M,\TT}(a,b)= \sum_{a} \Lineq_2^{\rho,M,\TT}(b,a)$ since both measure the balance of the state $b$.
  Since $\Lineq_2^{\rho,M,\TT}(a,b)=\NLineq_{a,b}^{\rho,M,\TT}M_{a,b}$ the previous considerations lead to
  \[\sum_a \lambda_a M_{a,b}= \sum_a \lambda_b M_{b,a}\]
  and since the RHS is $\lambda_b$, this says $\lambda=(\lambda_a,a\in \Ek)$ is a right eigenvector of $M$ associated with the eigenvalue 1, so that $\lambda_a=\alpha \rho_a$ for a constant $\alpha$. It remains to compute $\alpha$. 

Write 
 \[\sum_a \NLineq_1^{\rho,M,\TT}(a)=\sum_a\sum_{b} \NLineq_2^{\rho,M,\TT}(a,b)M_{a,b}=\alpha \sum_a \rho_a=  \alpha.\]
Now, using that $\ME_7^{M,\TT}\equiv 0$, let us prove that $\alpha=0$. For this consider
\be
\alpha=	\sum_{x_1}\NLineq_1^{\rho,M,\TT}(x_1)
	&=&\sum_{x_{-1},x_0,x_1,x_2}\rho_{x_{-1}}\wZ_{x_{-1},x_0,x_1,x_{2}}\prod_{k \in \{-1,0,1\}}M_{x_k,x_{k+1}}\\
	&&\hspace{0.2cm}+\sum_{x_0,x_1,x_2,x_{3}}\rho_{x_0}\wZ_{x_{0},x_1,x_{2},x_{3}}\prod_{k \in \{0,1,2\}}M_{x_k,x_{k+1}}
	\ee
and this is 0 by Lemma \ref{lem:balanced}. Hence $\alpha=0$, and therefore $\NLineq_1^{\rho,M,\TT}\equiv 0$ and $\NLineq_2^{\rho,M,\TT}\equiv 0$.
\end{proof}
Putting together the three previous Lemmas, we see that when $\ME_7^{M,\TT}\equiv  0$, then $\NLineq^{\rho,M,\TT}\equiv 0$ for any $n$, and then $\Lineq^{\rho,M,\TT}\equiv 0$ too.

$\bullet$  {\bf Proof of  $(v)\imp (vi)$}  Observe the linear form of $\NCycle_n^{M,\TT}$ given in \eref{eq:CyZ} (valid for $n\geq 3$) and check that for $n\geq 7$ and any $x\in\Ek^{\cro{0,n-1}}$, 
\beq\label{eq:yjr}
	\NCycle_n^{M,\TT} (x) =\NCycle_{n-1}^{M,\TT} (x^{\{n-4\}})+\ME_7^{M,\TT}(x\cro{n-7,n-1}),
\eq
which rewrites $\NCycle_n^{M,\TT} (x) =\NCycle_{n-1}^{M,\TT} (x^{\{n-4\}})$, since $\ME_7^{M,\TT}\equiv 0$.
This formula implies that $\NCycle_n^{M,\TT}(x)$ does not depend on $x_{n-4}$, and then since  $\NCycle_n^{M,\TT}$ is cyclically invariant, does not depend on any letter. It is then equal to $\NCycle_n^{M,\TT} (0^n)$ where $0^n$ is the word formed with $n$ repetitions of $0$, and then since $\NCycle_n^{M,\TT} (0^n)=n\wZ_{0,0,0,0}$, we can conclude using $\ME^{M,\TT}_7(0^7)= Z_{0,0,0,0}=0$.\par
  It remains to treat the case $n=3$ to 6. But observe \eref{eq:CyZ}, and consider for $m\in\cro{3,6}$ 
a word $w$ of size $m$, and the word $w^g$ obtained by the concatenation of $g$ copies of $w$ for the $g\geq 3$ of your choice. Then one sees that   
\[\NCycle_{m}^{M,\TT}(w)=\NCycle_{gm}^{M,\TT}(w^g)/g=0,\]
since $gm\geq 9$.\\
$\bullet$ \textbf{Proof of $(vi)\imp (vii)\imp (viii)$} Trivial\\
$\bullet$ \textbf{Proof of $(viii)\imp (ix)$.}
Consider the map
\[\W{a,b,c}:=	\wZ_{0,0,0,a}+\wZ_{0,0,a,b}+\wZ_{0,a,b,c}.\]
We will use $\NCycle_7^{M,T}(a,b,c,d,0,0,0)\equiv 0$ with different parameters.  
Start with $a=b=c=d=0$ to obtain that $\wZ^{M,\TT}_{0,0,0,0}=0$.
Now use arbitrary $a,b,c\in \Ek$ and $d=0$ to obtain that  
 \[ 	\W{a,b,c}=\wZ_{0,0,0,a}+\wZ_{0,0,a,b}+\wZ_{0,a,b,c}=  -\wZ_{a,b,c,0}-\wZ_{a,b,0,0}-\wZ_{a,0,0,0}.
\]
 	Now for arbitrary $a,b,c,d\in \Ek$,  $\NCycle_7^{M,T}(a,b,c,d,0,0,0)\equiv 0$ is equivalent to
 	\be
 	\wZ_{a,b,c,d} &=& -\wZ_{0,0,0,a}-\wZ_{0,0,a,b}-\wZ_{0,a,b,c}-\wZ_{b,c,d,0}-\wZ_{c,d,0}-\wZ_{d,0,0,0}\\
 	& =& -\W{(a,b,c)}+\W{(b,c,d)}.
 	\ee
        $\bullet$ \textbf{Proof of $(ix)\imp (v)$.}
        If $\wZ_{a,b,c,d} = -\W{(a,b,c)}+\W{(b,c,d)}$, then a telescopic simplification allows us to see that  for all $n\geq 4$
\[
	\sum_{w \in \Sub{a\cro{1,n}}{4}} \wZ_w = \W{(a_{n-2},a_{n-1},a_n)}-\W{(a_1,a_2,a_3)}
        \]
        from what we infer by \eref{eq:Meq} that $\ME_7^{M,\TT}\equiv 0$.
~~~\hfill $\Box$

\subsection{Proof of Theorem \ref{theo:t3b}}
\label{sec:theo:prop1}
The proof is almost the same as that of  Theorem \ref{theo:t0}. The only differences concern $(v)\imp(vi)$ and $(viii)\imp(ix)$.\\
To prove $(v) \imp (vi)$, take the proof of the corresponding statement in Theorem \ref{theo:t0} noticing that now the linear form \eref{lasduhf} of $\NCycle_n^{\rho,\TT} $ is valid from $n\geq 2$ and replace \eref{eq:yjr} by
\[
	\NCycle_n^{\rho,\TT} (x) =\NCycle_{n-1}^{\rho,\TT} (x^{\{n-2\}})+\ME_3^{\rho,\TT}(x\cro{n-3,n-1}).
\]
For $(viii)\imp(ix)$: Take $ab0=000$ to deduce $\wZ^{\rho,\TT}_{0,0}=0$. Use this and take $ab0=a00$ to find that 
	$
	\wZ^{\rho,\TT}_{a,0}+\wZ^{\rho,\TT}_{0,a} =0.
	$
	For general $ab0$ we obtain the identity
	$
		\wZ^{\rho,\TT}_{ab} = \wZ^{\rho,\TT}_{0,b}-\wZ^{\rho,\TT}_{0,a}.
	$
	So it is enough to define $\W{(a)}= \wZ^{\rho,\TT}_{0,a}+C$ (for any constant $C$).

\subsection{Proof of Theorem \ref{theo:ppppp}\label{sec:CIP2}}
We will adapt the proof of Theorem 3.1. in \cite{FGS} (steps 4 and 5). The main difference is that they use that a measure is invariant iff 
		$
		\int Gf(\eta)d\rho^\Z(\eta) = 0
		$
		for every bounded cylinder function $f:\Ek^{\Z^d} \to \R$.
                This is equivalent to $\Lineq^{\rho,\TT,p}(x(A))= 0$ for any $A\subset \Z^d$ finite, and $x(A)\in \Ek^{A}$.		
	We do not need to take the limit to get (91) and (92) and the last part of step 5, just $n$ sufficiently large, given that our $p$ is a finite rate transition probability.

\subsection{Proof of Theorem \ref{theo:cand3}}\label{sec:CFF}\label{seq:Pcand3}

$\bullet$ Let us first assume that $\nu_{a,b,c}=\frac{M_{a,b}M_{b,c}M_{c,a}}{t^3}$ for $t=\Tr(M^3)^{1/3}$ for some Markov kernel $M$.
In this case, $N_a$ given in \eref{eq:qff} satisfies  $N_a=(1/t)\begin{bmatrix} \dis \frac{M_{x,y}M_{y,a}}{M_{x,a}} \end{bmatrix}_{x,y\in \Ek}$, and then a main right eigenvector of $N_a$ is given by $r_a'={}^t\begin{bmatrix} 1/M_{y,a} \end{bmatrix}_{ y \in E_k}$. One sees that $\lambda=1/t$ is the common main eigenvalue to all the $N_a$'s. In the same way, one sees that $\ell'_a=\begin{bmatrix} \rho_xM_{x,a} \end{bmatrix}_{ x \in E_k}$ is a main left eigenvector associated to $N_a$. \par
The vectors $r_a$ and $\ell_a$ of the theorem are obtained after normalisation: $\ell_a=\begin{bmatrix} \rho_x M_{x,a}/\rho_a\end{bmatrix}_{x\in \Ek}$, $r_a=\rho_a r_a'={}^t\begin{bmatrix} \rho_a/M_{y,a} \end{bmatrix}_{ y \in E_k}$. Now 
$L_{a,b}= \l(M_{a,b}M_{b,x}M_{x,a}/t^3,x\in \Ek\r)$, giving $L_{a,b}r_a=\rho_a M_{a,b}/t^3$ and indeed, $\sum_{a,b} L_{a,b}r_a=1/t^3=\lambda^3$.\\
$\bullet$ Now, assume that $\nu$ is given, and \eref{eq:gfghs} possesses a positive recurrent solution $M$. From the previous point, the $N_a$'s have same main eigenvalues. 
The main argument of the proof we will develop  relies  on the structure of $\nu$, which allows to show that $M$ exists, it is characterized by \eref{eq:gfghs}.
Equation \eref{eq:gfghs} motivates to consider $\nu_{a,b,c}$ as the weight of a cycle $abc$ of length 3, which may be expanded as a product on its edges: 
\be
\nu_{a,b,c}=\prod_{e\in\{(a,b),(b,c),(c,a)\}} w_e,
\ee
where 
\[w_{(u,v)}={M_{u,v}}/{t},~~~\textrm{ for }t=\Tr(M^3)^{1/3}.\]
More generally, for any directed graph $G=(V,E)$, let 
\[W(G)= \prod_{e\in E} w_{e}.\]
We will see that the knowledge of $\nu_{a,b,c}$ allows to determine the weight of the cycles of every size, and then, by taking a limit, we will determine $M$. 
First, taking $(a,b,c)=(a,a,a)$ provides
\be
M_{a,a}= t\, \nu_{a,a,a}^{1/3} 
\ee
 and for the cycle $(a,b,a)$, since $\nu_{a,b,a}=M_{a,b}M_{b,a}M_{a,a}/t^3$,
\[
	M_{a,b}M_{b,a}=t^2\,\nu_{a,b,a}\,\nu_{a,a,a}^{-1/3}.
\]
Consider a cycle $\mathcal{C}_n = (a_1,\dots ,a_{n-1},a_{n},a_1)$ of length $n$ on $E_\kappa$, and for some $1<j<n$ add the directed edge $(a_1,a_j)$ as well as the edge $(a_j,a_1)$ to get the graph $\mathcal{C}'_n$. We may partition this oriented graph also as the union of
 $\mathcal{C}_{j}$ of length $j$ and $\mathcal{C}_{j,n} = (a_{j},\dots ,a_{n-1},a_{n},a_1,a_j)$. 
Therefore 
\beq \label{eq:addt}
W(\mathcal{C}_n)=\frac{W(\mathcal{C}'_n)}{w_{a_1,a_j}w_{a_j,a_1}}=  \frac{W(\mathcal{C}_{j})W(\mathcal{C}_{j,n})}{w_{a_1,a_j}w_{a_j,a_1}}= W(\mathcal{C}_{j})W(\mathcal{C}_{j,n})\times \frac{\nu_{a_1,a_1,a_1}^{1/3}}{\nu_{a_1,a_j,a_1}}.
\eq
A simple iteration argument allows one to express the weight of a cycle of any length with the weights of cycles of length $3$. A particular way to do that, is to see \eref{eq:addt} as the algebraic effect of the addition of the edge $(a_1,a_j)$ and $(a_j,a_1)$ in the cycle $\mathcal{C}_n$: adding all the edges from and to $a_1$ yields to
\beq\label{eq:sgsge}
	W(\mathcal{C}_n) =\nu_{a_1,a_2,a_3}\prod_{j=3}^{n-1}\frac{\nu_{a_1,a_j,a_{j+1}}\nu_{a_1,a_1,a_1}^{1/3}}{\nu_{a_1,a_j,a_{1}}}.
\eq

Using the matrices $L,N,R$, \eref{eq:sgsge} implies that
\[
\sum_{a_3,\cdots,a_n}W(\mathcal{C}_n)=L_{a_1,a_2}N^{n-3}_{a_1} {\bf 1}.
\]
Using Perron-Frobeniüs theorem and $(i)$ (here is used the fact that the $N_a$'s have the same eigenvalues),
\beq\label{eq:yrjhkx}
\sum_{a_3,\cdots,a_n}\frac{W(\mathcal{C}_n)}{\lambda_{a_1}^{n-3}}=\sum_{a_3,\cdots,a_n}\frac{W(\mathcal{C}_n)}{\lambda^{n-3}}\sous{\longrightarrow}{n\to+\infty}  L_{a_1,a_2} r_{a_1}\ell_{a_1} R = L_{a_1,a_2}r_{a_1}>0.
\eq
It is important to notice that the formula hence obtained, is independent from the Markov kernel $M$ solution of \eref{eq:gfghs} chosen, so that every $M$ which solves \eref{eq:gfghs} must satisfy $t^n\,	W(\mathcal{C}_n) =M_{a_n,a_1}\prod_{j=1}^{n-1} M_{a_j,a_{j+1 }}$.
Summing the previous relation over all the values of $a_3,\cdots,a_n$, we get that it must also satisfy
\ben\label{eq:jdbdfd}
\sum_{a_3,\cdots,a_n} t^n\,	W(\mathcal{C}_n) =M_{a_1,a_2}M^{n-1}_{a_2,a_1}.
\een
Now we make some connections.
Compare \eref{eq:yrjhkx} with \eref{eq:jdbdfd}. Taking $\lambda=1/t$, 
we see that 
\[M_{a_1,a_2}M^{n-1}_{a_2,a_1} \xrightarrow[n\to +\infty]{}  t^3 L_{a_1,a_2}r_{a_1}.\]
Since $M$ is assumed to be positive recurrent, $M^{n-1}_{a_2,a_1}\to \rho_{a_1}$ for some probability measure $\rho$. Hence we have established that for all pair $(\rho,M)$ where $M$ is a positive recurrent Markov kernel $M$, and $\rho$ its invariant probability measure, satisfies
$\rho_{a_1}M_{a_1,a_2}= t^3 L_{a_1,a_2}r_{a_1}$ for any $(a_1,a_2)\in E_\kappa^2$.
But a unique pair $(\rho,M)$ is solution of this equation when the RHS is given, since $\rho_a$ must be equal to $\sum_{b}t^3 L_{a,b}r_{a}$.

\section{Annexe}
\label{sec:ann}

\subsection{Proof of Theorem \ref{theo:t2}}
For any $I\subset \N$, denote by $\mathcal{E}_{I}=\{(M,\TT):  \NCycle_i^{M,\TT}\equiv 0, i \in I \}$.
First, we claim that 
\begin{lem}\label{lem:E456}
 $\mathcal{E}_{4,5,6} = \mathcal{E}_7$.
\end{lem}
\begin{proof}
  $\bullet$ From Theorem \ref{theo:t0}, if $(M,\TT)\in\mathcal{E}_7$ then $(M,\TT)\in \mathcal{E}_{4,5,6}$. \par
$\bullet$  For the converse: we use the following identity
  \be
  \NCycle_7(abcdefg)&=& -\NCycle_6(e f g d e f)+\NCycle_4(e f g d)+\NCycle_6(e f a d e f)\\
  &&-\NCycle_4(e f a d) +\NCycle_6(e f g c e f)-\NCycle_4(e f g c) \\
  && -\NCycle_6(e f a c e f)+\NCycle_4(e f a c) +\NCycle_6(a b c d e f)  \\
&&+\NCycle_6(a b c e f g)-\NCycle_5(a b c e f),
  \ee
  which can be checked by expansion in terms of $\wZ_w$, and by making an inventory of the multiplicity of each word $w$ involved. Hence, $\NCycle_7^{M,\TT}$ is a linear combination of some instances of $\NCycle_j$ for $j$ from 4 to 6. 
\end{proof}
Here is a short explanation of the origin of the formula appearing in the lemma  proof:  
Take $(M,\TT)$ a solution of $\mathcal{E}_{4,5,6}$. Since $0=\NCycle_4^{M,\TT}(a,b,a,b)$, therefore $\wZ_{a,b,a,b}= -\wZ_{b,a,b,a}$.
Now, since 
\be
0&=&\NCycle_6^{M,\TT}(a,b,c,d,a,b)-\NCycle_4^{M,\TT}(a,b,c,d)\\
&=& -\wZ_{d,a,b,c}+\wZ_{d,a,b,a}+\wZ_{a,b,a,b}+\wZ_{b,a,b,c},	
\ee
replacing $\wZ_{a,b,a,b}$ by  $-\wZ_{b,a,b,a}$ in this equation, gives the identity
\[
\wZ_{d,a,b,c}-\wZ_{d,a,b,a}-\wZ_{b,a,b,c}+\wZ_{b,a,b,a} \quad \forall a,b,c,d\in\N.
\]
In these equations, the parameters of $\wZ$ have the form $Z_{x,a,b,y}$ for different values of $x$ and $y$. Two terms depend on $d$ and two on $c$: this implies that the differences between the elements that depend on $d$ (respectively $c$) do not depend on $d$ (respectively $c$). This provides new identities. Playing with the dependence of the differences in the variables involve, leads eventually to the formula. But the formula can be checked directly independently from these considerations as indicated in the Theorem proof.
\begin{proof}[Proof of Theorem \ref{theo:t2}]
$(ii)$ We start by proving that if $(M,\TT) \in \mathcal{E}_{4,6}$, then $(M,\TT) \in \mathcal{E}_5$.
For this just check that
\be
\NCycle_5(abcde)&=&\NCycle_6(d e e c d e)-\NCycle_6(d e a c d e)+\NCycle_4(d e a c)\\
&&\NCycle_6(d e a 0 d e)-\NCycle_4(d e e c)-\NCycle_4(d e a 0)\\
&&-\NCycle_6(d e e 0 d e)+\NCycle_4(d e e 0).
\ee
$(i)$ We prove that if $(M,\TT)\in \mathcal{E}_{5,6}$ then it is in $\mathcal{E}_4$ too. By expansion, one checks that
\be
\NCycle_4^{M,\TT}(a,b,c,d)&=&  \NCycle_5^{M,\TT}(a,b,c,d,a)+ \NCycle_5^{M,\TT}(a,b,c,c,d)\\
&&-\NCycle_6^{M,\TT}(a,b,c,c,d,a).
\ee
\end{proof}

\subsection{Proof of Theorem \ref{theo:prop1}}\label{sec:CIP}

Here is a Lemma which implies Theorem \ref{theo:prop1}.
\begin{lem}\label{lem:l0} Let $M$ and $M'$ be two positive Markov kernels on $\Ek$, for $1\leq \kappa\leq +\infty$.
	The following three properties $(1)$ $(2)$ and $(3)$ are equivalent
	\begin{enumerate}
		\item[(1)] $\dis\frac{M_{a,u}M_{u,v}M_{v,d}}{M_{a,b}M_{b,c}M_{c,d}} =\frac{M'_{a,u}M'_{u,v}M'_{v,d}}{M'_{a,b}M'_{b,c}M'_{c,d}}\quad\forall a,b,c,d,u,v$.
		\item[(2)] $\dis\frac{M_{a,u}M_{u,v}M_{v,a}}{M_{a,b}M_{b,c}M_{c,a}} =\frac{M'_{a,u}M'_{u,v}M'_{v,a}}{M'_{a,b}M'_{b,c}M'_{c,a}}\quad\forall a,b,c,u,v$.
		\item[(3)] $\dis\frac{M'_{a,b}M'_{b,c}}{M_{a,b}M_{b,c}}= \alpha \frac{M'_{a,c}}{M_{a,c}}\quad\forall a,b,c$, for some $\alpha>0$, independent of $a,b,c$.
              \end{enumerate}
	Moreover, if $M$ and $M'$ are positive recurrent, then each of the previous properties implies that $M=M'$. 
\end{lem}
\begin{proof}
Taking $d=a$ in $(1)$ suffices to see that $(1)\imp (2)$.\\
Proof of $(2)\imp (3)$. Taking  $u=a$ and $v=c$ in $(2)$ provides
     \ben\label{eq:erhgfze}\frac{M'_{a,a}M'_{a,c}}{M_{a,a}M_{a,c}} =\frac{M'_{a,b}M'_{b,c}}{M_{a,b}M_{b,c}}.\een

Taking now $b=c$ in the last equation, gives
\ben\label{eq:etejryj}
\frac{M'_{a,a}}{M_{a,a}} =\frac{M'_{c,c}}{M_{c,c}},\een so that $a\mapsto \frac{M'_{a,a}}{M_{a,a}}$ is constant, say, equals to $\alpha$. Replacing $\frac{M'_{a,a}}{M_{a,a}}$ by $\alpha$ in \eref{eq:erhgfze} gives $(3)$.\\
To prove $(3)\imp (1)$, it may be useful to see $(3)$ as an equation ruling the addition of any letter $b$ between $a$ and $c$. Let us add two letters $u$ and $v$ (or $b$ and $c$) between $a$ and $d$...
	\begin{align*}
		\alpha\frac{M'_{a,d}}{M_{a,d}} =\frac{M'_{a,v}M'_{v,d}}{M_{a,v}M_{v,d}}=\frac{M'_{a,c}M'_{c,d}}{M_{a,c}M_{c,d}}
		\imp &\frac{M'_{a,u}M'_{u,v}M'_{v,d}}{M_{a,u}M_{u,v}M_{v,d}}=\alpha^2\frac{M'_{a,d}}{M_{a,d}} =\frac{M'_{a,b}M'_{b,c}M'_{c,d}}{M_{a,b}M_{b,c}M_{c,d}}.
\end{align*}
        This gives $(1)$.\par
        It remains to prove the last statement. 
Using point $(3)$, for a right $\alpha>0$, for  any word $x\in E_\kappa^n$
	\[
	\frac{M'_{a,x_1}M'_{x_1,x_2}\dots M'_{x_n,d}}{M_{a,x_1}M_{x_1,x_2}\dots M_{x_n,d}} = \alpha\frac{M'_{a,x_2}M'_{x_2,x_3}\dots M'_{x_n,d}}{M_{a,x_2}M_{x_2,x_3}\dots M_{x_n,d}}=\dots=\alpha^n\frac{M'_{a,d}}{M_{a,d}}.
	\]
	Then multiplying by the LHS denominator and summing over all values of $x_1,\cdots,x_n$, we obtain
	\[
	\frac{(M')^{n+1}_{a,d}}{(M)^{n+1}_{a,d}} =\alpha^n\frac{M'_{a,d}}{M_{a,d}}.
	\]
	Since $M$ and $M'$ are positive recurrent, taking the limit when $n\rightarrow \infty$, we get
        \ben
	\label{eq:inter}
	\frac{\rho'_d}{\rho_d} = \frac{M'_{a,d}}{M_{a,d}}\lim_{n\rightarrow \infty}\alpha^n,
	\een
        where $\rho$ and $\rho'$ are the invariant measures for the  Markov kernels $M$ and $M'$. 	Hence $\alpha=1$.
        Taking $b=c=a$ in $(3)$ then gives
        $M_{a,a}'/M_{a,a}=1$ for any $a$.  Using this relation and taking $d=a$ and $\alpha=1$ in \eref{eq:inter}, we obtain $\rho'_a=\rho_a$, and still from \eref{eq:inter}, this implies $M'_{a,d}=M_{a,d}$ for any $(a,d)$.
\color{black}
\end{proof}

\subsection{Proof of Theorem  \ref{theo:t0prime}}
We start with a preliminary lemma.
\begin{lem}(Analogue of Lemma \ref{lem:balanced})\label{lem:balance2}
If $M$ is a positive Markov kernel with memory $m$, then for all $a,c \in \Ek^{\cro{1,m}}$, any function $f,g:\Ek^m\to \R$, 
\[	\sum_{b \in \Ek^{L}} f(a)g(c)\wZ_{w} \prod_{i=1}^{L+m}M_{w\cro{i,i+m}}=0
\]
where in the sum, $w=w\cro{1,L+2m}$ is used instead of $abc$ (meaning that $w\cro{1,m}=a$ and $ w\cro{L+m+1,L+2m}=c$).
\end{lem}
The proof is the same as that of Lemma \ref{lem:balanced} taking into account  Remark \ref{rem:rthgf}.

\noindent The proof of Theorem  \ref{theo:t0prime} is very similar to that of Theorem \ref{theo:t0}; we discuss only the main differences. \\
\noindent We will prove the two cyclical implications $(i)\imp (ii)\imp (iii)\imp (iv)\imp (v)\imp (i)$ and $(v)\imp (vi)\imp (vii)\imp (viii)\imp (ix)\imp (v)$.\\
$\bullet$ {\bf Proof of $(i)\imp (ii)$} The following comparison gives the result if we consider $n = 2k+1$ large enough
\[
	\NLineq_n(x\cro{1,k+1 }0^k)-\NLineq_n(x\cro{1,k}0^{k+1}) = \rho_1\Rep_{{\sf h}}(x\cro{k+2-{\sf s},k+1}0^{{\sf s}-1};0)
\]
and the LHS is zero, because the solution is invariant.\\
$\bullet$ {\bf Proof of $(ii)\imp (iii)\imp (iv)\imp (v)$} The proof of Theorem \ref{theo:t0} may be adapted.\\
$\bullet$ {\bf Proof of $(v)\imp (i)$} The proof of the following lemmas can be adapted
\begin{lem}(Analogue of Lemma \ref{lem3})\label{lem3:1} Let $M$ be a Markov kernel with memory $m$ and positive entries.  If $\ME_{{\sf h}}^{M,\TT}\equiv 0$, then for all $n\geq 2L-1$, all $k \in \cro{ 2, n-1 }$, all $x\in \Ek^{n}$,
	\[
	\NLineq_n^{\rho,M,\TT}(x)=\NLineq_{n-1}^{\rho,M,\TT}(x^{\{k\}}). 
	\]
\end{lem}

\begin{lem}(Analogue of Lemma \ref{lem2})\label{lem2:1} If $M$ is a positive Markov kernel with memory $m$ then for all $n\geq m$, for all
$x\in\Ek^{\cro{1,n}}$, then
	\be
	\sum_{y\in \Ek}\NLineq_{n+1}^{\rho,M,\TT}(xy)M_{x\cro{n-m+1,n}y}&=&\NLineq_{n}^{\rho,M,\TT}(x)\\
        &=& \sum_{y\in \Ek}\NLineq_{n+1}^{\rho,M,\TT}(yx\cro{1,n})M_{x\cro{n-m,n}}.
	\ee	
Moreover if $n\leq m-1$, then
	\[
	\sum_{y}\NLineq_{n+1}^{\rho,M,\TT}(x\cro{1,n}y)=\NLineq_{n}^{\rho,M,\TT}(x\cro{1,n})= \sum_{y}\NLineq_n^{\rho,M,\TT}(yx\cro{1,n}).
	\]	
\end{lem}

To prove this modification, just see that from $\Lineq$ to $\NLineq$ we divided by $\prod_{j=1}^{n-m}M_{x\cro{j,m+j}}$ if $n\geq m+1$ and {$\NLineq_n^{\rho,M,\TT}\equiv\Lineq_n^{\rho,M,\TT}$} if $n\leq m$, hence summing over $x_n$ (respectively $x_1$), using $\sum_{b\in \Ek}M_{a,b}=1$ (respectively $\sum_{a\in \Ek^m} \rho_aM_{a,b}=\rho_b$) and Lemma \ref{lem:balance2} gives the result.

\begin{lem}(Analogue of Lemma \ref{lem:4})\label{lem:asluf} 
	Consider $M$ a positive Markov kernel with memory $m$, and $\TT$ a JRM with range $L$. If $\ME_{{\sf h}}^{M,\TT}\equiv 0$ then for all  $n$, 	$\NLineq_n^{\rho,M,\TT}\equiv 0$.
\end{lem}

\begin{proof}
  Since by Lemma \ref{lem2:1} one can deduce the nullity of $\NLineq_n^{\rho,M,\TT}$ from $\NLineq_{n+1}^{\rho,M,\TT}$, and since by Lemma \ref{lem3:1} we may deduce the nullity of $\NLineq_n^{\rho,M,\TT}$ from that of $\NLineq_{n-1}^{\rho,M,\TT}$ for $n\geq 2L-1$, it suffices to prove $\NLineq_N^{\rho,M,\TT}\equiv 0$ from the $N$ of our choice, as far as it is larger than $\max\{2L-1,2m\}$ (where $2m$ is chosen for commodity). \par
  We will adapt the argument of Lemma \ref{lem:4}. The argument is a bit more involved here.  Take $A\in \Ek^N$.   Using iteratively Lemma \ref{lem2:1},
    	\beq\label{eq:rgfq}
	\sum_{b\cro{1,m}\in \Ek^m}\NLineq_{N+m}^{\rho,M,\TT}(Ab\cro{1,m})\prod_{j=1}^{m}M_{A\cro{N-(m-j),N} b\cro{1,j}}=\NLineq_{N}^{\rho,M,\TT}(A).\eq
By Lemma \ref{lem3}, $\NLineq_{N+m}^{\rho,M,\TT}(Ab\cro{1,m})$ is unaffected by the suppression of inner letters (as long as it remains at least $2L-2$ letters), so that $\NLineq_{N+m}^{\rho,M,\TT}(Ab\cro{1,m})=\NLineq_{N}^{\rho,M,\TT}(A\cro{N-m}b\cro{1,m})$. Hence, \eref{eq:rgfq} becomes
\beq\label{eq:zrhyj}
\sum_{b\cro{1,m} \in \Ek^m}\NLineq_{N}^{\rho,M,\TT}(A\cro{1,N-m}b\cro{1,m})\prod_{j=1}^{m}M_{A\cro{N-(m-j),N} b\cro{1,j}}=\NLineq_{N}^{\rho,M,\TT}(A).
\eq       
Consider the matrix $\Gamma=(\Gamma_{u,b})_{u,b \in \Ek^m}$ defined by 
        \[\Gamma_{u,b}= \prod_{j=1}^{m}M_{u\cro{j,m} b\cro{1,j}} \]
        \eref{eq:zrhyj} is equivalent to
        \beq\label{eq:etfqeher}\sum_{B\in \Ek^m}\NLineq_{N}^{\rho,M,\TT}(A\cro{1,N-m}B) \Gamma_{A\cro{N-m+1,N},B}= \NLineq_{N}^{\rho,M,\TT}(A)\eq
        Now, $\Gamma_{u,b}$ is a Markov kernel: it is $\P(X\cro{m+1,2m}=b \,|\, X\cro{1,m}=u)$ for a Markov chain with memory $m$ and kernel $K$. Therefore
        rewriting  in \eref{eq:etfqeher},
        $A$ under the form $A\cro{1,N-m}A'\cro{1,m}$ where $A'\cro{1,m}$ is the suffix of $A$, this equation is equivalent to
\[\NLineq_{N}^{\rho,M,\TT}(A\cro{1,N-m}A'\cro{1,m})= \sum_{B\in \Ek^m}\NLineq_{N}^{\rho,M,\TT}(A\cro{1,N-m}B) \Gamma_{A'\cro{1,m},B}\]  from what appears that       
        \[v_{A\cro{1,N-m}}=\begin{bmatrix}\NLineq_{N}^{\rho,M,\TT}(A\cro{1,N-m} B),~~  B \in \Ek^{m}\end{bmatrix}\]
        is a left eigenvector to the  $\Gamma^t$. Taking into account the hypothesis on $M$, 
        \[v_{A\cro{1,N-m}}=\lambda_ {A\cro{1,N-m}}\begin{bmatrix}1 & \cdots & 1\end{bmatrix}\]
        which means that $\NLineq_{N}^{\rho,M,\TT}(A\cro{1,N-m} B)=\lambda_ {A\cro{1,N-m}} $ does not depend on $B$.  

Since $N\geq 2L-1$, and since by Lemma \ref{lem3}, $\NLineq_{N}^{\rho,M,\TT}(A\cro{1,N-m} B)$ is unaffected by the suppression/addition of inner letters (as long as these operations are done on words with more than $2L-1$ letters for the suppression and $2L-2$ letters for the addition),   for any $C\in\Ek^{N-2m}$ 
  \be
  \NLineq_{N}^{\rho,M,\TT}(A\cro{1,N-m} B) &=& \NLineq_{2N-2m}^{\rho,M,\TT}(A\cro{1,N-m}C B)\\
  &=&\NLineq_{N}^{\rho,M,\TT}(A\cro{1,m}C B)
  \ee
  so that $\lambda_ {A\cro{1,N-m}}$ depends only of the $m$ first letters of $A$ (we keep $m$ letters for commodity, we could have kept only $A_1$).
  Hence, there exists a function $f$ such that
\beq \label{eq:PLEK}
  \NLineq_{N}^{\rho,M,\TT}(A\cro{1,m} B)=f(A\cro{1,m}),
\eq
for any word $B\in \Ek^{N-m}$. 
Now, we claim that for any  $k\geq m$,  $\NLineq_{k}^{\rho,M,\TT}(A\cro{1,k})=f(A\cro{1,m})$. If $k\geq N$, this can be proved using the argument above. For $m\leq k \leq N$, by Lemma \ref{lem3:1}, $\NLineq_{N}^{\rho,M,\TT}(A\cro{1,k})=\sum_{y\cro{1,N-m}}\NLineq_{N}^{\rho,M,\TT}(A\cro{1,k}y\cro{1,N-k}) \prod_{j=1}^{N-k}M_{y\cro{j-m,j}}$ where $y_{j}=A_{k+j}$ for $j\leq 0$.  Plugging that for any  $y\cro{1,N-k}$, $\NLineq_{N}^{\rho,M,\TT}(A\cro{1,k}y\cro{1,N-k})=f(A\cro{1,m})$, we get the result.

Now, consider the matrix ${\bf M}=({\bf M}_{u,v})_{u,v\in \Ek^m}$ defined by
\[{\bf M}_{a\cro{1,m},b\cro{1,m}}=\1_{a\cro{2,m}=b\cro{1,m-1}} M_{a\cro{1,m},b_m}.\]
If $X=(X_k,k\in \Z)$  is a Markov chain with kernel $M$ and memory $m$, ${\bf M}$ is simply the Markov kernel of the Markov chain $(Y_k, k \in \Z)$ defined by $Y_k=(X\cro{k,k+m-1})$.
Hence, 
\[\P(Y_{m}= b ~|~ Y_0=a)=\prod_{j=0}^{m-1} M_{a\cro{j,m-1}b\cro{0,j}}={\bf M}^m_{a,b}.\]
Let $C$ be a word with $m$ letters. Measuring the balance at $C$ gives the relation, $
\sum_{B\in \Ek^m} \Lineq_{2m}^{\rho,M,\TT}(BC)=\sum_{B\in \Ek^m} \Lineq_{2m}^{\rho,M,\TT}(CB)
$
so that, by \eref{eq:nli} and given that ${\bf M}$ is a Markov kernel
\be
\sum_{B\in \Ek}f(B) {\bf M}^m_{BC}&=&\sum_{B\in \Ek} \NLineq_{2m}^{\rho,M,\TT}(BC) {\bf M}^m_{BC}\\
&=&\sum_{B\in \Ek^m} \NLineq_{2m}^{\rho,M,\TT}(CB){\bf M}^m_{C,B} = f(C),
\ee
for the $f$ given in \eref{eq:PLEK}. Since ${\bf M}^m$ is positive,  $f(C)=\alpha \rho_C$, where $\rho_C$ is the invariant distribution of the Markov kernel ${\bf M}^{m}$.\par
It remains to check that $\alpha=0$.
For this, write 
\be
\sum_{B \in \Ek^{m}} \Lineq_m^{\rho,M,\TT}(B)=\sum_{B \in \Ek^{m}} f(B)=\sum_{B \in \Ek^{m}} \alpha \rho_B=\alpha 
\ee
and can be rewritten, taking into consideration Lemma \ref{lem2:1} and the discussion below it
\be
\sum_{B \in \Ek^{m}} \Lineq_m^{\rho,M,\TT}(B)&=& \sum_{A\in \Ek^{q}}\rho_{A\cro{1,m}} \sum_{B\in \Ek^m}\sum_{C\in \Ek^{q}} \l(\prod_{\ell=1}^{2q-m} M_{w\cro{\ell,\ell+m}}\r)\\
&&\times \sum_{j} Z_{w\cro{j+1,j+{\sf s}}}1_{\cro{j+1,j+{\sf s}}\cap\cro{q+1,q+m}\neq \varnothing}
\ee
where $w\cro{1,2q+m}=ABC$. The contribution of each $j$ such that $\cro{j+1,j+{\sf s}}$ intersects $\cro{q+1,q+m}$, that is the indices of the letters of $B$, can be considered apart, and the summation of each index of $w$ which does not enter in $\wZ$ can be simplified. 
The contribution of the $j$-th term becomes
\[\sum_{w\cro{j+1,j+{\sf s}}}\rho_{w\cro{j+1,j+m}}\wZ_{w\cro{j+1,j+{\sf s}}} \prod_{i=j+1}^{j+L+m}M_{w\cro{i,i+m}}\]
which is 0 by Lemma \ref{lem:balance2}.
\end{proof}

\subsection{Proof of Theorem \ref{eq:fqqds}}
First, the linearity principle (Remark \ref{rem:LP}) can be applied to a larger $L$ and $M$: if a $(\rho,M)$-Markov law is invariant by $\TT$, then $\Cycle_n^{M,\TT}\equiv 0$ for any $n$.

For the converse, starting from $\NCycle_n^{M,\TT}\equiv 0$ for $n\geq \kappa^m$ we want to prove that $\NLineq_n^{\rho,M,\TT}\equiv 0$ for any $n\geq 1$.
If $n-m+1\geq \kappa^{m}+1$, for any word $w$ of size $n$, by the pigeon hole principle, there is a word $w'$ of size $m$ which appears twice as a factor of $w$. The sum on $\wZ$ (indexed by $\ell$ letters) along each factors of size $\ell$ of $w$ between the two occurrences of $w'$ produce the same contribution as a cycle, and then can be simplified. After simplification, remains only the words $w$ with at most  $\kappa^{m}+m$ letters. The number of remaining words after simplification is then finite. Therefore, 
\[\l\|\frac{\partial}{\partial t} \mu_n^t\r\|_{\infty}\leq C:=\sup_{N \leq \kappa^{m}+m} ~~\sup_{x\in \Ek^N} \Lineq^{\rho,M,\TT}_N(x) <+\infty.\]
The  constant $C$ does not depend on $n$, and therefore
\[d_{TV}(\mu_n^t,\mu_n^0)\leq Ct,\]
where $d_{TV}$ denotes the total variation distance.
At time $`e>0$, for a fixed $r$, then we have $d_{TV}(\mu_r^{`e},\mu_r^0)\leq C`e$. 
Let us prove that the mixture condition (irreducibility and aperiodicity) implies that it is in fact 0. 
Recall the following property of the distance in variation:
\[d_{TV}(\mu,\nu)= 2 \inf \mathbb{E}( \1_{X'\neq Y'})\]
where the  infimun is taken over all couplings, that is, on all pairs $(X',Y')$ where $X'$ is $\mu$ distributed and $Y'$ is $\nu$ distributed. Now, take two pairs $(X_1,X_2)\sim \mu_{1,2}$ and $(Y_1,Y_2)\sim \nu_{1,2}$ with independent marginals, where $X_1$ and $X_2$ are $\mu$ distributed, $Y_1$ and $Y_2$ are $\nu$ distributed. Suppose that $(X_i,Y_i)$ for $i=1,2$ are optimal couplings for the marginals, that is  $d:=d_{TV}(\mu,\nu) = 2 \mathbb{E}( \1_{X_1\neq Y_1}) = 2 \mathbb{E}( \1_{X_2\neq Y_2})$. We have then
\[d_{TV}\l(\mu_{1,2},\nu_{1,2}\r) =d^2 + 2d (1-d)=2d-d^2> (3/2)d,\]
where this last equality is valid when $d$ is small ($d<1/2$).
Hence, if one knows that the distance 
$d_{TV}\l(\mu_{1,2},\nu_{1,2}\r) < `e < 1/2$ and that the marginals are independent, then  $d_{TV}(\mu,\nu)<(2/3)`e$.\par
The strategy is as follows: we will deduce from  the inequality $d_{TV}(\mu_n^t,\mu_n^0)\leq Ct$ for any $n$, that $d_{TV}(\mu_n^t,\mu_n^0)\leq  (3/4) Ct$ for any $n$, so that necessarily  $d_{TV}(\mu_n^t,\mu_n^0)=0$.\par
Take $I_r(k)=\cro{1,r} \cup \cro{(k-1)r+1,kr}$.
Now write $d_{TV}(\mu^t_{I_r(k)},\mu^0_{I_r(k)})\leq d_{TV}(\mu^t_{\cro{0,kr}},\mu^{0}_{\cro{0,kr}})<`e$. Since $I_r(k)$ is the union of two intervals, $\mu_{I_r(k)}^t$ is (for a clear notation) the distribution of the pair $(X^t\cro{1,r},X^t\cro{(k-1)r+1,kr})$. 
According to the previous discussion, to conclude it suffices to prove that when $k\to +\infty$, the two pairs $A_t(k):=(X^t\cro{1,r},X^0\cro{1,r})$ and $B_t(k):=(X^t\cro{(k-1)r+1,kr},X^0\cro{(k-1)r+1,kr})$ converges to two independent variables with the same distribution. \par
By the hypothesis we made on the Markov kernel, for any $`e'>0$, it is possible to find a $k$ large enough such that the variation distance of the initial configuration $(X^0\cro{1,r},X^0\cro{(k-1)r+1,kr})$ with a pair of independent r.v. with the same  marginals is smaller than $`e''>0$ for the $`e''$ of our choice. \par The fact that the initial configuration converges to independent vectors with same distribution when $k\to+\infty$, is not sufficient. We need to show that their evolution till time $t$ are asymptotically (in $k$)  independent too. \par The argument is routine: 
Since the number of colours is finite, $L$ is finite,  $\max_{w,w'\in \Ek^L} \T{(w)}{(w')}<+\infty$. For $t<\infty$ fixed we build a dependence graph $G_t$ as follows: first, the vertex set of the graph is the set of intervals of size $L$. For each  jump that has occurred in a interval $I$ before time $t$ we add an edge between this interval and the intervals which intersect it (to encode, the fact that the state at time $t$ of these intervals may have been modified by the jump in $I$). 
Since $\max_{w,w'\in \Ek^L} \T{(w)}{(w')}<+\infty$, when $k\to+\infty$, the probability that the two intervals $\cro{1,r}$ and $\cro{(k-1)r+1,kr}$ intersect distinct connected components of $G_t$ goes to 1. This suffices to deduce the asymptotic independence of $(A_t(k),B_t(k))$ when $k\to +\infty$.

\small
\bibliographystyle{abbrv}

\end{document}